\documentclass[a4paper
, leqno]{article}
\usepackage[a4paper,centering,vscale=0.77,hscale=0.77]{geometry}\usepackage{amsmath,amsthm,amssymb,bbm,bm}
\usepackage{xcolor}
\usepackage{verbatim}
\usepackage{todonotes}
\usepackage{mathrsfs}
\usepackage[all]{xy}
\usepackage{thmtools}
\usepackage[pdfdisplaydoctitle,colorlinks,breaklinks,urlcolor=blue,linkcolor=blue,citecolor=blue]{hyperref}
\usepackage{enumitem}

\newtheorem{theorem}{Theorem}[section]
\newtheorem{lemma}[theorem]{Lemma}
\newtheorem{proposition}[theorem]{Proposition}
\newtheorem{corollary}[theorem]{Corollary}
\newtheorem{definition}[theorem]{Definition}
\newtheorem{remark}[theorem]{Remark}
\newtheorem{example}[theorem]{Example}

\newcommand{\F}{\mathcal F}

\begin{document}

\title{An introduction to Malliavin calculus}
\author{Luciano Tubaro, Margherita Zanella}

\maketitle

\tableofcontents

\section{Introduction}

Malliavin calculus is named after P. Malliavin who first initiated this calculus with his seminal work \cite{Mal78}, see also \cite{Malliavin}.
There he laid the foundations of what is now known as the “Malliavin calculus", an infinite-dimensional differential calculus in a Gaussian framework, and used it to give a probabilistic proof of Hörmander’s theorem. This new calculus proved to be extremely successful and soon a number of authors studied variants and simplifications, see e.g. \cite{KS84, KS85, KS87, W84, Nualart, Stroock, Bis1, Bis2, Shi1, Shi2, Zakai}.

The techniques of Malliavin calculus prove to be sufficiently flexible to obtain results for a wide range of problems, including for example the study of the regularity of the image law of solutions of stochastic partial differential equations (see e.g. \cite{BoZa2,BoZa1,FerZan, MNQS13, M99, M95, MN08, PZS93, MCMS01, Nualart} and the therein references), the study of ergodic problems, (see among all \cite{HaiMat}), or the study of integrations by parts formulas on level sets in infinite-dimensional spaces (see e.g. \cite{Addo1,BDPT,BoTuZa,DLT}). Moreover, there are several applications in finance, see e.g. \cite{AL21} and in numerical analysis (see e.g. \cite{Talay1}, \cite{Talay2} and \cite{Crisan}).

The general context for Malliavin calculus consists of a probability space $(\Omega, \mathcal{F}, \mathbb{P})$ and a Gaussian separable Hilbert space $\mathcal{H}_1$, that is a closed subspace of $L^2(\Omega,\mathcal{F}, \mathbb{P})$ consisting of centered Gaussian random variables. The space $\mathcal{H}_1$ (also known as the first Wiener Chaos) induces an orthogonal decomposition, known as the Wiener Chaos Decomposition, of the corresponding $L^2(\Omega,\sigma(\mathcal{H}_1),\mathbb{P})$ space of square integrable random variables that are measurable with respect to the $\sigma$-field generated by $\mathcal{H}_1$. 
To characterize elements in $\mathcal{H}_1$ it is useful to fix a separable Hilbert space $\mathcal{H}$ and to consider a unitary operator $W$ between the two spaces.
The fundamental operators of Malliavin calculus are: the generator of the Ornstein–Uhlenbeck semigroup $L$, the derivative operator $D$ and the adjoint of the Malliavin derivative, usually called divergence operator $\delta$.  The three operators are linked by a specific relation.
These notions can be introduced without assuming any topological or linear structure on the probability space $\Omega$. We will introduce them in the first part of these notes. In details, in Section \ref{GHS_sec} we introduce the Gaussian Hilbert space $\mathcal{H}_1$, in Section \ref{WCD_sec} we prove the Wiener Chaos Decomposition Theorem. In Section \ref{OU_sec} we introduce the Ornstein-Uhlenbeck semigroup and its infinitesimal generator. Section \ref{Malliavin_sec} is devoted to the construction of the Malliavin derivative operator and in Section \ref{div_sec} we introduce the divergence operator. The relation of these operators is analyzed in Section \ref{relation_sec}.

In the usual applications $\Omega$ is often a linear topological space and the Malliavin derivative operator can be introduced as a differential operator. 
This is for instance the case of the Malliavin derivatives presented in the books \cite{Bogachev} by Bogachev and \cite{DaPrato} by Da Prato. In Section \ref{top_sec} we recall the construction of these two derivatives. Then we show that they can be interpreted as two (different) examples of the general notion of Malliavin derivative of Section \ref{Malliavin_sec}, which is general enough to admit as special cases the definitions of Malliavin derivative given in probability spaces with a linear topological structure.

We recall some notions needed throughout the lecture notes in the Appendices. 

\bigskip The material for these lecture notes was mostly inspired by the books  \cite{Nualart}, \cite{Janson}, \cite{Bogachev} and \cite{DaPrato}.

\section{Gaussian Hilbert spaces}
\label{GHS_sec}

\subsection{Preliminaries}

We recall some notions from probability theory and functional analysis, at the same time fixing some basic notation. For more details see e.g. \cite{Jac_Pro}.
\\
Let $(\Omega, \mathcal{F}, \mathbb{P})$ be a probability space i.e. a measure space $(\Omega, \mathcal{F})$ endowed with a measure $\mathbb{P}$ with $\mathbb{P}(\Omega)=1$.

\textbf{Random variables.}
Let $(E, \mathcal{E})$  be a measurable space. A \emph{random variable} with values in $E$, or a $E$-valued random variable, is a measurable function $X:(\Omega, \mathcal{F}, \mathbb{P}) \rightarrow (E, \mathcal{E})$, in the sense that 
\[
X^{-1}(A):=\{\omega \in \Omega \ : \ X(\omega) \in A\} \in \mathcal{F}, \quad \forall A \in \mathcal{E}.
\]
In what follows, $E$ will be a topological space and $\mathcal{E}$ will be the $\sigma$-algebra generated by the open sets (in the topology of $E$), namely the borelian sets $\mathcal{B}(E)$. When $(E, \mathcal{E})=(\mathbb{R}, \mathcal{B}(\mathbb{R}))$ we speak of \emph{real} random variables.
The \emph{$\sigma$-algebra generated by a $E$-valued random variable} $X$ is the $\sigma$-algebra $\sigma(X)\subset \mathcal{F}$ generated by the sets $\{X^{-1}(A) \ : \  A  \in \mathcal{E}\}$; it is the smallest $\sigma$-algebra on $\Omega$ that makes $X$ measurable.
Given a set $A$ of random variables on $(\Omega, \mathcal{F}, \mathbb{P})$, we denote by $\sigma(A)$ the $\sigma$-field generated by the random variables in $A$, that is, the smallest $\sigma$-field that makes all the random variables in $A$ measurable. Clearly, $\sigma(A) \subseteq \mathcal{F}$.

The \emph{distribution} or \emph{law} of the random variable $X$ on $(E, \mathcal{E})$ is the probability measure $\mu:= \mathbb{P} \circ X^{-1}$ defined as
\[
\mu(A)=\mathbb{P}(X \in A):=\mathbb{P}\left(\{\omega \in \Omega \ : \ X(\omega) \in A\} \right), \qquad A \in \mathcal{E}.
\]
The measure $\mu$ correspond to the push-forward measure $X_{\sharp}\mathbb{P}$ of $\mathbb{P}$ via $X$.

Let $(E, \mathcal{E})$ and $(G, \mathcal{G})$ be two measurable spaces and let $X$ and $Y$ be two random variables with values in $(E, \mathcal{E})$ and $(G, \mathcal{G})$, with distribution $\mu$ and $\nu$, respectively. Then $(X,Y): (\Omega, \mathcal{F}, \mathbb{P}) \rightarrow (E \times G, \mathcal{E} \otimes \mathcal{G})$ is a random variable with values in the product space $E \times G$ endowed with the product $\sigma$-algebra $\mathcal{E} \otimes \mathcal{G}$. The distribution of $(X,Y)$ is the probability measure $\pi$ on $(E \times G, \mathcal{E} \otimes \mathcal{G})$ defined as 
\[
\pi(A \times B)=\mathbb{P}\left( \{\omega \in \Omega\ : \ X(\omega) \in A, \ Y(\omega) \in B\}\right), \qquad A \in \mathcal{E}, \ B \in \mathcal{G}.
\]
This defines uniquely $\pi$ on the whole $\mathcal{E} \otimes \mathcal{G}$ by definition of product $\sigma$-algebra. The measure $\pi$ is usually called the \emph{joint distribution} (or joint law) of the random variables $X$ and $Y$. Given the joint distribution $\pi$ of $X$ and $Y$, one has 
\[
\mu(A)=\pi(A \times \Omega) =\mathbb{P}(X \in A), \quad A \in \mathcal{E}, \qquad \nu(B)=\pi(\Omega\times B) =\mathbb{P}(Y \in B), \quad B \in \mathcal{G}.
\]
The probability measures $\mu$ and $\nu$ are called the \emph{marginal} distributions of $X$ and $Y$, respectively.
We emphasize that the underlying probability space $(\Omega, \mathcal{F}, \mathbb{P})$ is arbitrary and often we will not mention it explicitly. Obviously, whenever we talk about joint distributions of random variables, it is understood that they are defined on the same probability space. 

Let $(E, \mathcal{B}(E))$ be a Banach space endowed with the Borel $\sigma$-field and let $\mu$ be a probability measure on $(E, \mathcal{B}(E))$. The \emph{characteristic function}, also called \emph{Fourier transform}, of $\mu$ is defined as 
\[
\hat \mu:E^* \rightarrow \mathbb{C}, \qquad \hat \mu(y):= \int_E e^{i \langle y,x\rangle}\,{\rm d}\mu(x), \quad y \in E^*.
\]
 It is known that characteristic functions determine measures, that is, given $\mu$ and $\nu$, two probability measures on $(E, \mathcal{B}(E))$, then 
 \[
 \mu = \nu \quad \iff \quad \hat \mu= \hat \nu.
 \]
 By characteristic function $\varphi_X$ of a $E$-valued random variable $X$ we mean the characteristic function of its law, that is $\varphi_X=\hat \mu$, where $\mu=\mathbb{P} \circ X^{-1}$. It holds 
 \[
 \varphi_X(y)=\int_Ee^{i \langle x,y\rangle}\,{\rm d}\mu(x)=\mathbb{E}\left[ e^{i\langle X,y\rangle}\right], \quad y \in E^*.
 \]
We also recall that the moment-generating function $M_X$ of a $E$-valued random variable $X$ is defined as
 \[
M_X(y)
 =\mathbb{E}\left[ e^{\langle X,y\rangle}\right], \quad y \in E^*,
 \]
provided this expectation exists.

\textbf{Moments and $L^p(\Omega)$ spaces.}
Let $X$ be a random variable with values in $(\bar{\mathbb{R}}, \mathcal{B}(\bar{\mathbb{R}})$, with $\bar{\mathbb{R}}:=\mathbb{R} \cup \{ \pm \infty\}$. The \emph{expected value} of $X$ is defined as
\[ 
\mathbb{E}[X]:= \int_\Omega X(\omega)\, {\rm d}\mathbb{P}(\omega), 
\]
and the integral has to be understood according to the usual measure theory. We call $\mathbb{E}\left[|X|^p\right]$ the \emph{$p$-th moment of $X$}, for all $p \in [1, \infty)$. 
The variance of $X$ is defined as 
\[
\mathbb{V}\text{ar}(X):=\mathbb{E}\left[ |X- \mathbb{E}[X]|^2\right]= \mathbb{E}\left[|X|^2\right]- \mathbb{E}[X]^2,
\]
and we define the \emph{covariance} of two real-valued random variables $X$, $Y$ as
\[
\mathbb{C}\text{ov}(X,Y):= \mathbb{E}\left[ (X- \mathbb{E}[X])(Y- \mathbb{E}[Y])\right]=\mathbb{E}[XY]-\mathbb{E}[X]\mathbb{E}[Y].
\]
We say that the two random variables $X$ and $Y$ are \emph{uncorrelated} if $\mathbb{C}\text{ov}(X,Y)=0$.
Notice that if $X$ and $Y$ are real random variables with zero mean, then $\mathbb{E}[|X|^2]=\mathbb{V}ar(X)$ and $\mathbb{E}[XY]=\mathbb{C}\text{ov}(X,Y)$. In particular, two such variables are orthogonal iff they are uncorrelated. 

The Lebesgue space $L^p(\Omega, \mathcal{F}, \mathbb{P};\mathbb{R})$ (usually abbreviated as $L^p(\Omega)$ or $L^p$) corresponds to the (equivalent class up to sets of zero probability of) random variables with values in $(\bar{\mathbb{R}}, \mathcal{B}(\bar{\mathbb{R}})$ with finite $p$-th moment.
\\
If $E$ is an Hilbert space, endowed with the scalar product $\langle \cdot, \cdot\rangle_E$, and $\mathcal{B}(E)$ is its Borel $\sigma$-algebra, for $p\in [1, \infty]$, the space $L^p(\Omega, \mathcal{F}, \mathbb{P};E)$ (usually abbreviated as $L^p(\Omega; E)$) is the linear space of all random variables $X:(\Omega, \mathcal{F}, \mathbb{P})\rightarrow (E, \mathcal{B}(E))$ such that the norm $\|X\|_{L^p(\Omega; E)}$ is finite. Here, $\|X\|_{L^p(\Omega; E)}= 
 \left(\mathbb{E}\left[\|X\|_E^p\right]\right)^{\frac 1p}$ for $p \in [1,\infty)$ and for $p=\infty$ this is modified to $\|X\|_{L^\infty(\Omega; E)}=  \text{ess sup}\|X\|_E$.
 $L^p(\Omega;E)$ is a Banach space for $p \in [1, \infty]$.
We will mainly use $L^2(\Omega; E)$ which is a Hilbert space with inner product $\langle X, Y\rangle_{L^2(\Omega; E)}=\mathbb{E}\left[\langle X,Y\rangle_E \right]$.

\subsection{Gaussian random variables}
\begin{definition}
[Gaussian measures on $\mathbb{R}$] A probability measure $\gamma$ on $(\mathbb{R}, \mathcal{B}(\mathbb{R}))$ is called \emph{Gaussian} if it is either a Dirac measure $\delta_\mu$ at a point $\mu$, or a measure absolutely continuous with respect to the Lebesgue measure with density
\begin{equation*}
\frac{1}{\sqrt{2\pi \sigma^2}}e^{-\frac{(x-\mu)^2}{2\sigma^2}},
\end{equation*}
with $\mu \in \mathbb{R}$ and $\sigma>0$. In this case we call $\mu$ the \emph{mean}, $\sigma^2$ the \emph{variance} of $\gamma$ and we set $\gamma=:\mathcal{N}(\mu, \sigma^2)$. If $\gamma=\delta_{\mu}$, we set $\gamma=:\mathcal{N}(\mu, 0)$ and call $\gamma$ \emph{degenerate}. We say that $\gamma$ is \emph{centered} if $\mu=0$ and \emph{standard} if $\mu=0$, $\sigma=1$. If $\gamma=\delta_\mu$, $\mu$ is still called the mean of $\gamma$.
\end{definition}
By simple computations one verifies that 
\begin{equation*}
\mu=\int_{\mathbb{R}}x\, \gamma({\rm d} x), \qquad \sigma^2=\int_{\mathbb{R}}(x-\mu)^2\, \gamma({\rm d} x).
\end{equation*}
For every $\mu \in \mathbb{R}$, $\sigma\ge 0$, the \emph{characteristic function} of a Gaussian measure $\gamma=\mathcal{N}(\mu, \sigma^2)$ (i.e. its Fourier transform) is given by
\begin{equation}
\label{car_fun_Gau_1D}
\widehat \gamma (t):= \int_{\mathbb{R}}e^{itx}\, \gamma({\rm d}x)=e^{i \mu t-\frac{\sigma^2t^2}{2}}.
\end{equation}
Conversely, one proves that a probability measure on $\mathbb{R}$ is Gaussian iff its characteristic function has this form. Thus it is easy to recognize a Gaussian measure from its characteristic function.
\begin{definition}[Real-valued Gaussian random variable]
Let $\xi:(\Omega, \mathcal{F}, \mathbb{P}) \rightarrow (\mathbb{R}, \mathcal{B}(\mathbb{R}))$ be a random variable. We say that $\xi$ is a  \emph{Gaussian} or \emph{normal} random variable if its law $\gamma$  is a Gaussian measure; in particular if  $\gamma=\mathcal{N}(\mu, \sigma^2)$ we write $ \xi \sim \mathcal{N}(\mu, \sigma^2)$. In this case $\xi$ has mean $\mu$ and variance $\sigma^2$.
\end{definition}
The random variable is said to be \emph{centered} if $\mu=0$ and \textit{standard} if $\mu=0$ and $\sigma=1$. In the degenerate case $\sigma=0$ the random variable is a.s. equal to $\mu$. 
Moreover Gaussian random variables have finite moment of every order. 

\begin{definition}
($\mathbb{R}^n$-valued Gaussian random variable). Let $\xi:(\Omega, \mathcal{F}, \mathbb{P}) \rightarrow (\mathbb{R}^n, \mathcal{B}(\mathbb{R}^n))$, $\xi=(\xi_1,...,\xi_n)$ be a random variable. We say that $\xi$ is Gaussian (equivalently, $\xi_1, ... \xi_n$ have \emph{joint gaussian distribution}) if $\sum_{i=1}^n \lambda_i\xi_i$ has gaussian distribution for any $\lambda_1,...\lambda_n$ in $\mathbb{R}$.
\end{definition}
We say that an infinite set of random variables have a joint gaussian distribution if every finite subset has. It is easy to verify that, if $\xi_1,...\xi_n$ are real-valued independent Gaussian random variables, then $\xi=(\xi_1,...,\xi_n)$ is Gaussian.
\\
Characteristic functions are of help when studying Gaussian random variables: $\xi$ is a $\mathbb{R}^n$-valued Gaussian random variable iff its characteristic function has the form
\[
\varphi_\xi(u)=e^{i\langle \mu, u\rangle -\frac 12 \langle u, Qu\rangle}, \qquad u \in \mathbb{R}^n,
\]
where $\mu \in \mathbb{R}^n$ and $Q$ is a $n \times n$ symmetric nonnegative semi-definite matrix. $\mu=(\mathbb{E}[\xi_1],...,\mathbb{E}[\xi_n])$ is the mean of $\xi$, $Q$, given by $Q_{ij}=\mathbb{C}\text{ov}(\xi_i,\xi_j)$, $i,j=1,...,n$, is the covariance matrix of $\xi$.
The moment-generating function of the $\mathbb{R}^n$-valued Gaussian random variable $\xi$ is of the form
\[
M_\xi(u)=e^{\langle \mu, u\rangle +\frac 12 \langle u, Qu\rangle}, \qquad u \in \mathbb{R}^n.
\]
We recall that a $\mathbb{R}^n$-valued Gaussian random variable $\xi$ admits a density on $\mathbb{R}^n$ iff the covariance matrix $Q$ is non-degenerate, that is det$Q\ne 0$. This means that, when $n \ge 2$, there exist non-constant random variables without densities (when $n=1$ a normal variable is either constant or with a density).

Two or more jointly Gaussian random variables are independent if and only if their covariance vanish. For centered Gaussian random variables, this is equivalent to the variables being orthogonal.

Two elementary but important properties of $\mathbb{R}^n$-valued random variables are as follow. If $\xi$ is a $\mathbb{R}^n$-valued Gaussian random variable and $\zeta$ is a $\mathbb{R}^m$-valued Gaussian random variable, such that $\xi$ and $\zeta$ are independent, then $Z=(\xi, \zeta)$ is a $\mathbb{R}^{n+m}$-valued Gaussian random variable. If $\xi$ is a $\mathbb{R}^n$-valued Gaussian random variable with mean $\mu$ and covariance matrix $Q$, given a $m \times n$ matrix $A$ and a vector $b$ in $\mathbb{R}^m$, $\zeta=A\xi+b$ is a $\mathbb{R}^m$-valued Gaussian random variable with mean $A\mu+b$ and covariance matrix $AQA^T$.


\subsection{Gaussian Hilbert spaces}
For this part we mainly refer to \cite{Janson}, see also \cite{Neveu}.
\begin{definition}
\label{Gaussian space}
\begin{itemize}
\item [(i)] A \emph{Gaussian linear space} is a real linear space of random variables, defined on some probability space $(\Omega, \mathcal{F}, \mathbb{P})$, such that each variable in the space is centered Gaussian. Obviously, a Gaussian linear space is a linear subspace of $L^2(\Omega; \mathbb{R})$, and we use the norm and the inner product of $L^2$ on it.
\item [(ii)] A \emph{Gaussian Hilbert space} is a Gaussian linear space which is complete, i.e. a closed subspace of $L^2(\Omega; \mathbb{R})$ consisting of centered Gaussian random variables.
\end{itemize}
\end{definition}
Note we require all random variables in a Gaussian space to be real.
Definition \ref{Gaussian space} requires that each variable in a Gaussian space has a normal distribution but it is easy to see that the variables are jointly normal too.
\begin{proposition}
\label{Gaus_joint_distr}
Any set of random variables in a Gaussian linear space has a joint normal distribution.
\end{proposition}
\begin{proof}
It is sufficient to consider a finite set of random variables. Let $\xi_1, ..., \xi_n$ belong to a Gaussian linear space $G$ and let $\lambda_1,...,\lambda_n \in \mathbb{R}$. By the linear structure of the space immediately follows that $\sum_{i=1}^n \lambda_i \xi_i \in G$. This yields that  $\sum_{i=1}^n \lambda_i \xi_i$ has a normal distribution which implies that $\xi_1,..., \xi_n$ has a joint normal distribution.
\end{proof}

A Gaussian linear space can always be completed to a Gaussian Hilbert space.

\begin{theorem}
\label{closure_linear_space}
If $G\subset L^2(\Omega, \mathcal{F},\mathbb{P})$ is a Gaussian linear space, then its closure $\bar G$ in $L^2(\Omega, \mathcal{F},\mathbb{P})$ is a Gaussian Hilbert space.
\end{theorem}
\begin{proof}
$\bar G$ is an Hilbert space being a closed subset of the Hilbert space $L^2(\Omega, \mathcal{F},\mathbb{P})$. It remains to prove that every elements in $\bar  G$ has a centered normal distribution. Let $\xi \in \bar G$, there exists a sequence $\xi_n \in G$ such that $\xi_n \rightarrow \xi$ in $L^2$. Let $\sigma_n^2=\|\xi\|^2_{L^2}=\mathbb{E}\left[|\xi|^2\right]$ and $\sigma_n^2=\|\xi_n\|^2_{L^2}=\mathbb{E}\left[|\xi_n|^2\right]$.
Then $\sigma_n^2 \rightarrow \sigma^2$ as $n \rightarrow \infty$. Since convergence in $L^2$ implies convergence in distribution and $\xi_n \sim \mathcal{N}(0, \sigma_n^2)$ converges in distribution to $\mathcal{N}(0, \sigma^2)$, as $n \rightarrow \infty$, we get $\xi \sim \mathcal{N}(0, \sigma^2)$ and this concludes the proof.
\end{proof}

Since Gaussian random variables have moments of any orders, the Gaussian Hilbert spaces are subset of $L^p$ spaces.
\begin{proposition}
\label{L^p_Gauss}
A Gaussian Hilbert space is a closed subspace of $L^p(\Omega, \mathcal{F},\mathbb{P})$, $1\le p< \infty$.
\end{proposition}
\begin{proof}
See \cite[Theorem 1.4]{Janson}.
\end{proof}

In view of Theorem \ref{closure_linear_space}, with no loss of generality, we can consider only Gaussian Hilbert spaces. We will deal with these in most of what follows and we will use the symbol $\mathcal{H}_1$ to denote a generic Gaussian Hilbert space.
We emphasize that incomplete Gaussian spaces naturally appear in many concrete situations. However, thanks to Theorem \ref{closure_linear_space}, it will be sufficient to consider their closure to obtain a Gaussian Hilbert space; an example is the following.

\begin{example}
\label{X^gam_ex_sec1}
Let $X$ be a real Banach space, or more generally, a locally convex topological vector space. 
A Borel probability measure $\gamma$ on $X$ is said to be centered Gaussian if every continuous liner functional $x^* \in X^*$, regarded as a random variable defined on the probability space $(X, \mathcal{B}(X), \gamma)$, is centered Gaussian. It is immediate to see that $X^*$ is a Gaussian linear space and, by Theorem \ref{closure_linear_space}, its closure in $L^2(X, \mathcal{B}(X), \gamma)$ is a Gaussian Hilbert space (notice that some elements $x^*$ may be $0$-$\mu$-a.e., in which case the Gaussian space is really a quotient space of $X^*$). 
This example is of great interest in applications. We will come back to it in Section \ref{top_sec}, considering an infinite dimensional separable real Banach space (see in particular Definition \ref{Xstargammadef} and Corollary \ref{X^* Gauss}).\end{example}


Having provided the definition of Gaussian Hilbert spaces, we devote the remaining part of this Section to provide some ways of characterizing them.
 The following results highlight how Gaussian Hilbert spaces can be constructed (and characterized) in terms of the closed linear span of sets of jointly Gaussian centered random variables.

\begin{proposition}
\label{lin_span_ex}
If $\{\xi_i\}_{i \in I}$ is any set of centered jointly Gaussian random variables on $(\Omega, \mathcal{F}, \mathbb{P})$, then the linear span of $\{\xi_i\}_{i \in I}$ is a Gaussian linear space. The closure in $L^2(\Omega)$ of the linear span of $\{\xi_i\}_{i \in I}$  is a Gaussian Hilbert space.
\end{proposition}
\begin{proof}
It is immediate to see that the linear span of $\{\xi_i\}_{i \in I}$ is a Gaussian linear space. Its closure is a Gaussian Hilbert space in view of Theorem \ref{closure_linear_space}.
\end{proof}

\begin{corollary}
\label{lin_span_ex_ind}
Let $\{\xi_i\}_{i \in I}$ be any set (finite or infinte, possibly uncountable) of independent centered Gaussian random variables. Their closed linear span $\{\sum_i \lambda_i \xi_i : \sum_i \lambda_i^2< \infty\}$ is a Gaussian Hilbert space. Conversely, every Gaussian Hilbert space $\mathcal{H}_1$ is of this form.
\end{corollary}
\begin{proof}
The first assertion is an immediate consequence of Proposition \ref{lin_span_ex}. For what concerns the second assertion, given any orthonormal basis $\{\xi_i\}_{i \in I}$ in the space $\mathcal{H}_1$, its elements are uncorrelated, thus independent, centered Gaussian random variables. Their closed linear span is then equal to the given space. 
\end{proof}

As particular cases of Proposition \ref{lin_span_ex} we provide the following examples.

\begin{example}
\label{1DGauss}
Let $\xi$ be any non-degenerate centered Gaussian random variable on a probability space $(\Omega, \mathcal{F}, \mathbb{P})$. Then $\mathcal{H}_1=\{\lambda \xi, \ \lambda \in \mathbb{R}\}$ is a one-dimensional Gaussian Hilbert space. 
Notice that the existence of such $\mathcal{H}_1$ relies on the existence of a Gaussian random variable $\xi$. Two possible ways of constructing such a random variable are as follows:
\begin{itemize}
\item consider the probability space $(\mathbb{R}, \mathcal{B}(\mathbb{R}), \gamma)$, with $\gamma$ the standard Gaussian measure, i.e. $\gamma(A)=\int_A\frac{1}{\sqrt{2\pi}}e^{-\frac{x^2}{2}}\, {\rm d}x$, $A \in \mathcal{B}(\mathbb{R})$. Then $\xi= \emph{Id}$ is a standard normal Gaussian random variable and $\mathcal{H}_1= \{\lambda \emph{Id}, \lambda\in\mathbb{R}\} \subset L^2(\mathbb{R}, \mathcal{B}(\mathbb{R}), \gamma)$. Notice that we can also think $\mathcal{H}_1$ as $\mathbb{R}^*$.
\item Consider the probability space $((0,1), \mathcal{B}(0,1), \mathcal{L})$, with $\mathcal{L}$ the Lebesgue measure on $(0,1)$. Denoted by $\Phi$ the moment generating function of a standard Gaussian random variable, that is $\Phi(t)= \int_{-\infty}^t \frac{1}{\sqrt{2\pi}}e^{-\frac{x^2}{2}}\, {\rm d}x$, $t \in \mathbb{R}$, it is easy to see that $\xi= \Phi^{-1}$ is a standard normal Gaussian random variable.
\end{itemize}
\end{example}

\begin{example}
Let $\xi_1,..., \xi_n$ have a joint gaussian distribution with zero mean. Then their linear span $Span(\xi_1,..., \xi_n)=\{ \sum_{i=1}^n \lambda_i\xi_i, \ \lambda_i \in \mathbb{R}\}$ is a finite-dimensional Gaussian Hilbert space. 
\end{example}

\begin{example}
\label{BM_ex}
\begin{itemize}
\item [(i)]
Let $\{B_t\}_{t \ge 0}$ be a standard Brownian motion on the probability space $(\Omega,\mathcal{F}, \mathbb{P})$. By Proposition \ref{lin_span_ex}, the space $\mathcal{H}_1$ 
given by the closure of $Span \{B_t\}_{t \ge 0}$ in $L^2(\Omega)$
 is a Gaussian Hilbert space. 
\item [(ii)]
More generally, if $\{X_t\}_{t \in T}$ is any Gaussian stochastic process in discrete or continuous time formed by zero-mean random variables, then the closed linear span of $\{X_t\}_{t \in T}$ is a Gaussian Hilbert space.
\end{itemize}
\end{example}

Different characterizations of Gaussian Hilbert spaces are possible. Of particular interest is the one provided in Proposition \ref{Gaussian_isometries} below.
We recall that a \emph{linear isometry} of a Hilbert space into another is a linear map that preserves the inner product. Linear isometries that are onto are called \emph{unitary operators} (for more details see e.g. \cite{RS} and \cite{Brezis}).

\begin{proposition}
\label{Gaussian_isometries}
Let $\mathcal{H}$ be a real 
Hilbert space. Then, there exists a Gaussian Hilbert space $\mathcal{H}_1$ (with the same dimension of $\mathcal{H}$) and a unitary operator $h \mapsto U(h)$ of $\mathcal{H}$ onto  $\mathcal{H}_1$. That is, $\mathcal{H}_1=\{U(h), \ h \in \mathcal{H}\}$ and  for any $h, k \in \mathcal{H}$,
\begin{equation*}
\mathbb{E}\left[U(h)U(k) \right]= \langle h, k \rangle_{\mathcal{H}}.
 \end{equation*}

\end{proposition}
\begin{proof}
Let $\{e_i\}_{i \in I}$ be an orthonormal basis in $\mathcal{H}$. 
Let $\{\xi_i\}_{i \in I}$ be a collection of independent standard gaussian random variables, defined on some probability space $(\Omega, \mathcal{F},\mathbb{P})$, with the same index set $I$ (the cardinality of $I$ can be finite or infinite countable). Every element $h \in \mathcal{H}$ can be uniquely written as $h=\sum_{i \in I} \langle h, e_i\rangle_{\mathcal{H}}e_i$. We introduce the mapping $\mathcal{H} \ni h \mapsto U(h):= \sum_{i \in I}\langle h, e_i\rangle_{\mathcal{H}} \xi_i$. 
By construction, the random variable $U(h)$ is Gaussian. Moreover, since the $\xi_i$ are independent, centered and have unit variance, $U(h)$ is centered and, for any $h, k \in \mathcal{H}$,
\begin{equation*}\mathbb{E}\left[U(h)U(k) \right]=\mathbb{E} \left[ \sum_{i \in I}\langle h, e_i\rangle_{\mathcal{H}} \xi_i \sum_{j \in I}\langle k, e_j\rangle_{\mathcal{H}} \xi_j \right]=\sum_{i \in I}\langle h, e_i\rangle_{\mathcal{H}}\langle k, e_i\rangle_{\mathcal{H}}= \langle h, k \rangle_{\mathcal{H}}.
\end{equation*}
This entails that $U$ is a unitary operator of $\mathcal{H}$ onto the Gaussian Hilbert space $\mathcal{H}_1:=\{U(h) \ : \ h \in \mathcal{H}\}$, and concludes the proof.
 \end{proof}

The role of the space $\mathcal{H}$ and the operator $U$, in the above result, is to suitable index the elements in $\mathcal{H}_1$. We point out that, fixed a generic Gaussian Hilbert space $\mathcal{H}_1$, there are infinitely many possible choices of real Hilbert spaces $\mathcal{H}$ (with the same dimension as $\mathcal{H}_1$) and unitary operators $U$ such that $\mathcal{H}_1=\{U(h) \ : \ h \in \mathcal{H}\}$. For instance, since $\mathcal{H}_1$ is itself a real Hilbert space (w.r.t. the usual $L^2(\Omega)$ inner product), it follows that $\mathcal{H}_1$ can be trivially represented by choosing $\mathcal{H}$ equal to $\mathcal{H}_1$ itself and $U$ equal to the identity operator. In general, given an Hilbert space $\mathcal{H}$, there are infinitely many different ways of choosing an orthonormal basis $\{e_i\}_{i \in I}$ in $\mathcal{H}$ and an orthonormal basis $\{\xi_i\}_{i \in I}$ in $\mathcal{H}_1$, each choice giving a \emph{different} unitary operator $U$ of the form $\mathcal{H} \ni h=\sum_{i \in I}\langle h, e_i\rangle_{\mathcal{H}}e_i \mapsto U(h)=\sum_{i \in I}\langle h, e_i\rangle_{\mathcal{H}}\xi_i$.

Plainly, the subtlety in the use of Proposition \ref{Gaussian_isometries}, is that one has to select an Hilbert space $\mathcal{H}$ and a unitary operator $U$
that are well adapted to the specific problem at hand. 
We provide some concrete examples; see \cite[Chapter 1]{Janson} and \cite[Chapter 2.1]{NP} for further examples.

\begin{example}[\textbf{Euclidean spaces}]
Fix an integer $d \ge 1$ and let $(\xi_1,...,\xi_d)$ be a Gaussian vector whose components are independent standard Gaussian random variables. Let $\mathcal{H}_1$ be the Gaussian Hilbert space given by the the linear span of $\{\xi_i\}_{i=1}^n$. Set $\mathcal{H}=\mathbb{R}^d$ and let $(e_1,...,e_n)$ be the canonical orthonormal basis of $\mathbb{R}^d$ (w.r.t. the usual Euclidean inner product). Any $h \in \mathbb{R}^d$ can be uniquely written as $h=\sum_{i=1}^n \langle h, e_i\rangle e_i$. Define $U:\mathbb{R}^d \rightarrow \mathcal{H}_1$ ad $U(h):=\sum_{i=1}^n\langle h, e_i\rangle \xi_i$. $U$ is an unitary operator from $\mathbb{R}^d$ onto $\mathcal{H}_1$. In particular, $\mathcal{H}_1=\{U(h) \ : \ h \in \mathbb{R}^d\}$.
\end{example}

\begin{example}[\textbf{Gaussian Hilbert spaces derived from covariances}]
\label{ex_cov}
Let $X=\{X_t\}_{t \ge 0}$ be a real-valued centered Gaussian process on a probability space $(\Omega, \mathcal{F}, \mathbb{P})$ with covariance function given by \begin{equation}
\label{Cov_Y}
 R(s,t):=\mathbb{E}\left[ X_sX_t\right], \quad s,t \ge 0.
\end{equation}
Starting from $X$ one can construct a Gaussian Hilbert space $\mathcal{H}_1$,  through the characterization of Proposition \ref{Gaussian_isometries}, as follows.
\begin{itemize}
\item [(i)] Define $\mathcal{E}$ as the collection of all finite linear combinations of indicator functions of the form $\pmb{1}_{[0,t]},\ t \ge 0$.
\item [(ii)] Define $\mathcal{H}_R$ to be the Hilbert space given by the closure of $\mathcal{E}$ with respect to the scalar product
\begin{equation*}
\langle h, k\rangle_{\mathcal{H}_R}=\sum_{i,j}a_i b_j R(s_i, t_j), 
\end{equation*}
for $h=\sum_i a_i {\pmb 1}_{[0,s_i]}$ and $k=\sum_j b_j {\pmb 1}_{[0,t_j]}$, elements in $\mathcal{E}$.
\item[(iii)] Introduce the linear map $I$ that associates to an element $h=\sum_i a_i {\pmb 1}_{[0,s_i]}$ in $\mathcal{E}$ the element $I(h)=\sum_j a_i X_{s_i}$. The mapping $I$ defines an isometry from $\mathcal{E}$ onto the (incomplete) linear Gaussian space spanned by $X$.
\item [(iv)] For $h \in \mathcal{H}_R$ define $I(h)$ as the $L^2(\Omega)$ limit of any sequence of type $I(h_n)$ where $\{h_n\}_n \subset \mathcal{E}$ converges to $h$ in $\mathcal{H}_R$.
\item [(v)] By construction, $\{I(h), h \in \mathcal{H}_R\}$ is a Gaussian Hilbert space. In fact it equals the closure in $L^2(\Omega)$ of the linear Gaussian space spanned by $X$, which is a Gaussian Hilbert space in view of Corollary \ref{lin_span_ex_ind}.
\end{itemize}

Gaussian Hilbert spaces of this type naturally appear in the study of stochastic differential equations driven by a Gaussian noise. 
\begin{footnote}{
Extension of the above construction to Hilbert-valued centered Gaussian processes is also possible, see for instance \cite{DalQue}. 
}
\end{footnote}
For instance, 
\begin{itemize}
\item [1.] $R(s,t)=s\wedge t$ is the covariance function of a standard Brownian motion $B$. In this case, the isometry $I$ of $L^2([0, \infty))$ onto the Gaussian Hilbert space $\mathcal{H}_1$ spanned by $B$, is known as the stochastic \emph{It\^o-Wiener} integral and it is usually written as $h \mapsto \int_0^\infty h(s)\, {\rm d}B_s$. Thus we have the characterization $\mathcal{H}_1=\{\int_0^\infty h(s)\, {\rm d}B_s, \ h \in L^2([0,\infty))\}$.
\item [2.] $R(s,t)=\frac 12 \left( s^{2H}+t^{2H}-|t-s|^{2H}\right)$ is the covariance function of a fractional Brownian motion with Hurst parameter $H$. The Gaussian Hilbert space associated to the fractional Brownian motion has been considered for instance in \cite{BoZa1} and \cite{BoZa2}.
\end{itemize}
\end{example}

\begin{example}
Let $T>0$ and let $D\subset \mathbb{R}^2$ be a smooth domain. Let $Q:L^2(D) \rightarrow L^2(D)$ be a positive symmetric bounded linear operator.
We define $L^2_Q$ as the completition of the space of all square integrable functions $\varphi:D \rightarrow \mathbb{R}$ with respect to the scalar product
\begin{equation*}
\langle \varphi, \psi \rangle_{L^2_Q}= \langle Q\varphi, \psi\rangle_{L^2}.
\end{equation*}
Set $\mathcal{H}=L^2(0,T;L^2_Q)$. 
This space is a real separable Hilbert space with respect to the scalar product 
\begin{equation}
\langle f,g\rangle_{\mathcal{H}}=\int_0^T\langle f(s), g(s)\rangle_{L^2_Q} \, 
{\rm d}s =\int_0^T\langle Qf(s), g(s)\rangle_{L^2} \, {\rm d}s.
\end{equation}
Let $\{\xi_i\}_{i=1}^\infty$ be a collection of real independent Gaussian random variables, defined on some probability space $(\Omega, \mathcal{F}, \mathbb{P})$. 
Given an orthonormal basis $\{e_i\}_{i=1}^\infty$ in $\mathcal{H}$, every element $h \in \mathcal{H}$ can be uniquely written as $h=\sum_{i=1}^\infty \langle h, e_i\rangle_{\mathcal{H}}e_i$. By orthonormality, one has $\|h\|^2_{\mathcal{H}}=\sum_{i=1}^\infty \langle h, e_i\rangle^2_{\mathcal{H}}< \infty$. From this fact we deduce that, as $n \rightarrow \infty$, the sequence 
$\sum_{i=1}^n\langle h, e_i\rangle_{\mathcal{H}} \xi_i$ converges in $L^2(\Omega)$ and almost surely to some random variable that we denote by $U(h)=\sum_{i=1}^\infty\langle h, e_i\rangle_{\mathcal{H}} \xi_i$.
By construction, the random variable $U(h)$ is a centered Gaussian random variable with variance given by $\|h\|^2_{\mathcal{H}}$
and $U$ is a unitary operator of $\mathcal{H}$ onto the Gaussian Hilbert space $\mathcal{H}_1:=\{U(h) \ : \ h \in \mathcal{H}\}$.
This Gaussian Hilbert space has been considered for instance in \cite{FerZan}.
\end{example}

\begin{example}
In Section \ref{top_sec} we will characterize the Gaussian Hilbert space of Example \ref{X^gam_ex_sec1} as $\mathcal{H}=\{U(h), h \in \mathcal{H}\}$, in terms of different choices of Hilbert spaces $\mathcal{H}$ and unitary operators $U$ (see Corollaries \ref{char_X^*_gam-H} and \ref{corX_gam_Hil}).
\end{example}


\subsection{The Isserlis formula}
The Isserlis formula, later rediscovered by Wick, is a formula of the joint product moments $\mathbb{E} \left[ \xi_1 \cdot \cdot \cdot \xi_n\right]$, where $\xi_1, ... , \xi_n$ are centered jointly Gaussian variables.

\begin{theorem}
\label{Iss_thm}
[\textbf{Isserlis formula or Wick product}].
Let $\xi_1, ... , \xi_n$ be centered jointly Gaussian random variables. Then 
\begin{equation}
\label{Iss_formula}
\mathbb{E} \left[ \xi_1 \cdot \cdot \cdot \xi_n\right] = \sum\prod_k \mathbb{E}\left[ \xi_{i_{k}} \xi_{j_{k}}\right],
\end{equation}
where the sum is over all the partitions of $\{1,...,n\}$ into disjoint pairs $\{i_k, j_k\}$.
\end{theorem}

\begin{remark}
Obviously, a partitions of $\{1,...,n\}$ into disjoint pairs exists if and only if $n=2m$ is even, in this case there are $(2m)!/(2^m m!)=(2m-1)!!$ such partitions. If $n$ is odd it does not exist a partition of $\{1,...,n\}$ into disjoint pairs and the expected value in \eqref{Iss_formula} is zero.
\end{remark}
The Isserlis formula basically states that the product moment $\mathbb{E} \left[ \xi_1 \cdot \cdot \cdot \xi_n\right] $ can be fully expressed using only the pairwise covariances $\mathbb{E}\left[ \xi_{i} \xi_{j}\right]$ for $i,j=1...n$. 
For example, for $n=4$ there are three partitions of $\{1,..,4\}$ into disjoint pairs: $\{(1,2),(3,4)\}$, $\{(1,3),(2,4)\}$, $\{(1,4),(2,3)\}$ and 
\begin{equation*}
\mathbb{E} \left[ \xi_1\xi_2\xi_3 \xi_4\right]
=
\mathbb{E}\left[ \xi_1 \xi_2\right]\mathbb{E}\left[ \xi_3 \xi_4\right]+\mathbb{E}\left[ \xi_1\xi_3\right]\mathbb{E}\left[ \xi_2 \xi_4\right]+ \mathbb{E}\left[ \xi_1 \xi_4\right]\mathbb{E}\left[ \xi_2 \xi_3\right].
\end{equation*}
For $n=5$ there are zero partition of $\{1,...,5\}$ into disjoint pairs and $\mathbb{E} \left[ \xi_1\xi_2\xi_3 \xi_4 \xi_5\right]=0$, for $n=6$ there are 15 partitions of $\{1,...,6\}$ into disjoint pairs and 
\begin{align*}
\mathbb{E} \left[ \xi_1\xi_2\xi_3 \xi_4 \xi_5 \xi_6\right]
&=\mathbb{E}\left[ \xi_1 \xi_2\right]\mathbb{E}\left[ \xi_3 \xi_4\right]\mathbb{E}\left[ \xi_5 \xi_6\right]
+\mathbb{E}\left[ \xi_1 \xi_2\right]\mathbb{E}\left[ \xi_3 \xi_5\right]\mathbb{E}\left[ \xi_2 \xi_6\right]
+\mathbb{E}\left[ \xi_1 \xi_2\right]\mathbb{E}\left[ \xi_3 \xi_6\right]\mathbb{E}\left[ \xi_2 \xi_4\right]
\\
&+\mathbb{E}\left[ \xi_1 \xi_3\right]\mathbb{E}\left[ \xi_2 \xi_4\right]\mathbb{E}\left[ \xi_5 \xi_6\right]
+\mathbb{E}\left[ \xi_1 \xi_3\right]\mathbb{E}\left[ \xi_2 \xi_5\right]\mathbb{E}\left[ \xi_3 \xi_6\right]
+\mathbb{E}\left[ \xi_1 \xi_3\right]\mathbb{E}\left[ \xi_2 \xi_6\right]\mathbb{E}\left[ \xi_3 \xi_4\right]
\\
&+\mathbb{E}\left[ \xi_1 \xi_4\right]\mathbb{E}\left[ \xi_2 \xi_3\right]\mathbb{E}\left[ \xi_5 \xi_6\right]
+\mathbb{E}\left[ \xi_1 \xi_4\right]\mathbb{E}\left[ \xi_2 \xi_5\right]\mathbb{E}\left[ \xi_3 \xi_6\right]
+\mathbb{E}\left[ \xi_1 \xi_4\right]\mathbb{E}\left[ \xi_2\xi_6\right]\mathbb{E}\left[ \xi_3 \xi_5\right]
\\
&+\mathbb{E}\left[ \xi_1 \xi_5\right]\mathbb{E}\left[ \xi_2 \xi_3\right]\mathbb{E}\left[ \xi_4 \xi_6\right]
+\mathbb{E}\left[ \xi_1 \xi_5\right]\mathbb{E}\left[ \xi_2 \xi_6\right]\mathbb{E}\left[ \xi_3 \xi_4\right]
+\mathbb{E}\left[ \xi_1 \xi_5\right]\mathbb{E}\left[ \xi_2 \xi_4\right]\mathbb{E}\left[ \xi_3 \xi_6\right]
\\
&+\mathbb{E}\left[ \xi_1 \xi_6\right]\mathbb{E}\left[ \xi_2 \xi_3\right]\mathbb{E}\left[ \xi_4 \xi_5\right]
+\mathbb{E}\left[ \xi_1 \xi_6\right]\mathbb{E}\left[ \xi_2 \xi_4\right]\mathbb{E}\left[ \xi_3 \xi_5\right]
+\mathbb{E}\left[ \xi_1 \xi_6\right]\mathbb{E}\left[ \xi_2 \xi_5\right]\mathbb{E}\left[ \xi_4 \xi_6\right].
\end{align*}

Notice that, if we take all $\xi_i$ equal in Theorem \ref{Iss_thm}, we obtain the formula for moments of a centered Gaussian variable, that is, if $\xi \sim \mathcal{N}(0, \sigma^2)$, then 
\begin{equation*}
\mathbb{E}\left[\xi^n\right]
=
\begin{cases}
(n-1)!!\sigma^n & n \ \text{even},
\\
0 & n \ \text{odd}.
\end{cases}
\end{equation*}

Theorem \ref{Iss_thm} will be proved using moment generating functions and will rely on elementary computations from multivariate calculus.
Let us fix some notation. Denote by
\begin{equation}
\label{Gamma_def}
C=\{c_{i,j}\}_{i, j}
\quad \text{with} \quad 
c_{i,j}=\mathbb{E}\left[\xi_i \xi_j\right], \quad  i, j=1,...,n
\end{equation}
the covariance matrix of the centered Gaussian vector $(\xi_1,...,\xi_n)$ and its elements, respectively. Notice that $c_{i,j}=c_{j,i}$ for all $i,j=1,...,n$.
\\
Define
$\mathcal{Q}:\mathbb{R}^n \rightarrow \mathbb{R}$ as the quadratic form
\begin{equation}
\label{Q_Iss_def}
\mathcal{Q}(\lambda):= \frac 12\sum_{i,j=1}^n c_{i,j}\lambda_i \lambda_j, \quad \lambda=(\lambda_1,..,\lambda_n) \in \mathbb{R}^n,
\end{equation}
and 
$\mathcal{F}: \mathbb{R}^n \rightarrow\mathbb{R}$ as
\begin{equation*}
\mathcal{F}(\lambda):=e^{\mathcal{Q}(\lambda)}.
\end{equation*}
We set
\begin{equation*}
q(\lambda):= \frac{\partial}{\partial \lambda_1}\mathcal{Q}(\lambda)=\sum_{j=1}^n c_{1,j}\lambda_j
\end{equation*}
and, for every $i=1,...,n$,
\begin{equation}
\label{Fi}
\mathcal{F}_i(\lambda):= \frac{\partial}{\partial \lambda_i}\mathcal{F}(\lambda) \qquad \text{and} \quad \mathcal{F}_{1,...,n}(\lambda):= \frac{\partial^n}{\partial \lambda_1\cdot \cdot \cdot \partial \lambda_n}\mathcal{F}(\lambda).
\end{equation}

\begin{lemma}
\label{Lemma_Iss}
For any $n \ge 1$, 
\begin{equation*}
\mathcal{F}_{1,2,3,\ldots,n}(0)
=
\begin{cases}
0 & \text{when} \ n \ \text{is odd}
\\
\sum\prod_k c_{i_{k}} c_{j_{k}} & \text{when} \ n \ \text{is even,}
\end{cases}
\end{equation*}
where the sum is over all the partitions of $\{1,...,n\}$ into disjoint pairs $\{i_k, j_k\}$.
\end{lemma}

\begin{proof}
By dropping the dependence on $\lambda$ for simplicity, we can easily compute
\[
\F_{1}=q\;\F
\]
\[
\F_{1,2}=c_{1,2} \;\F+q\; \F_{2}
\]
\[
\F_{1,2,3}=c_{1,2} \F_{3}+ c_{1,3} \F_{2}+q\; \F_{2,3}
\]
\[
\F_{1,2,3,4}=c_{1,2}\F_{3,4}+ c_{1,3} \F_{2,4}+ c_{1,4} \F_{2,3}+q\;   \F_{2,3,4}
\]
\[
\F_{1,2,3,4,5}=c_{1,2}\F_{3,4,5}+ c_{1,3} \F_{2,4,5}+ c_{1,4} \F_{2,3,5}+ c_{1,5} \F_{2,3,4}+q\;   \F_{2,3,4,5}
\]
\[
\F_{1,2,3,4,5,6}=c_{1,2}\F_{3,4,5,6}+ c_{1,3} \F_{2,4,5,6}+ c_{1,4} \F_{2,3,5,6}+c_{1,5} \F_{2,3,4,6}+c_{1,6}\F_{2,3,4,5}+q\;   \F_{2,3,4,5,6}
\]
For a generic $n \ge 1$ we obtain 
\begin{equation*}
\F_{1,2,3,\ldots,n}(\lambda )=\sum_{j=2}^nc_{1,j}\F_{2,\ldots,\hat j,\ldots,n}(\lambda) + q(\lambda) \F_{2,3,\ldots,n}(\lambda),
\end{equation*}
where by $2,\ldots,\hat j,\ldots,n$ we denote the sequence $2,3,\ldots,n$ from which we have removed the term $j$.
\\
Since $q(0)=0$, we get
\begin{equation*}
\F_{1,2,3,\ldots,n}(0)=\sum_{j=2}^nc_{1,j}\F_{2,\ldots,\hat j,\ldots,n}(0).
\end{equation*}
By noticing that $\mathcal{F}_i(0)=0$ for any $i$, proceeding by induction, it is easy to prove that 
\begin{equation*}
\F_{1,2,3,\ldots,n}(0)=0
\end{equation*}
when $n$ is odd. When $n$ is even, by noticing that $\mathcal{F}(0)=1$ and using an induction argument, one instead gets 
\begin{equation*}
\F_{1,2,3,\ldots,n}(0)=\sum\prod_k c_{i_{k}} c_{j_{k}},
\end{equation*}
where the sum is over all the partitions of $\{1,...,n\}$ into disjoint pairs $\{i_k, j_k\}$.
This concludes the proof.
\end{proof}

We have now all the ingredients to prove Theorem \ref{Iss_thm}.

\begin{proof}
[Proof of Theorem \ref{Iss_thm}.]
The random variables $\xi_1,..., \xi_n$ are centered jointly Gaussian, thus we have that, for $\lambda \in \mathbb{R}^n$, the random variable $ \sum_{i=1}^n \lambda_i \xi_i \sim \mathcal{N}\left(0, \mathbb{E}\left[ \left|\sum_{i=1}^n \lambda_i \xi_i\right|^2\right]\right)$. Therefore 
\begin{equation*}
\mathbb{E}\left[e ^{\sum_{i=1}^n \lambda_i \xi_i} \right]
=e^{\frac 12 \sum_{i,j=1}^n\lambda_i \lambda_j \mathbb{E}\left[ \xi_i \xi_j \right]}
= e^{\mathcal{Q}(\lambda)},
\end{equation*}
where in the second equality we used the definitions \eqref{Gamma_def} and \eqref{Q_Iss_def}.
The dominated convergence Theorem yields:
\begin{equation*}
\mathbb{E}\left[ \xi_1\cdot \cdot \cdot \xi_n e ^{\sum_{i=1}^n \lambda_i \xi_i}\right] 
= \frac{\partial ^n}{\partial \lambda_1\cdot \cdot \cdot \partial \lambda_n}\mathbb{E}\left[e ^{\sum_{i=1}^n \lambda_i \xi_i} \right]
= \frac{\partial ^n}{\partial \lambda_1\cdot \cdot \cdot \partial \lambda_n}e^{\mathcal{Q}(\lambda)}=:\mathcal{F}_{1,...,n}(\lambda),
\end{equation*}
where we used the notation \eqref{Fi}.
Taking $\lambda =0$ in the above expression, from Lemma \ref{Lemma_Iss} we thus infer 
\begin{equation*}
\mathbb{E}\left[ \xi_1\cdot \cdot \cdot \xi_n\right] 
=\mathcal{F}_{1,...,n}(0)
=\begin{cases}
0 & \text{when} \ n \ \text{is odd}
\\
\sum\prod_k c_{i_{k}} c_{j_{k}} & \text{when} \ n \ \text{is even},
\end{cases}
\end{equation*}
and this concludes the proof.
\end{proof}

\section{The Wiener Chaos Decomposition}
\label{WCD_sec}


Every Gaussian Hilbert space induces an orthogonal decomposition, known as the Wiener Chaos Decomposition, of the corresponding $L^2$ space of square-integrable random variables that are measurable with respect to the $\sigma$-field generated by the Gaussian Hilbert space. \begin{footnote}{The fact that we work with the $\sigma$-algebra generated by the Gaussian Hilbert space $\mathcal{H}_1$ is central: all the objects
involved in Malliavin calculus are functionals of $\mathcal{H}_1$ and so we work in the universe generated by $\mathcal{H}_1$.}\end{footnote}
In this Section we will prove this result mainly following the approach contained in the book by Janson (see \cite[Chapter 2]{Janson}).
An important consequence of the Wiener Chaos Decomposition Theorem is the fact that the class of polynomials in elements of $\mathcal{H}_1$ is dense in $L^2$: this fact will be crucial  to introduce the Malliavin derivative operator.

Let $\mathcal{H}_1$ be a Gaussian Hilbert space defined on a probability space $(\Omega, \mathcal{F}, \mathbb{P})$. In view of Proposition \ref{L^p_Gauss}, random variables in $\mathcal{H}_1$ belong to $L^p(\Omega)$ for every finite $p$, thus by the H\"older inequality it is immediate to see that any finite product of random variables in $\mathcal{H}_1$ belongs to $L^2(\Omega)$ (actually, to $L^p(\Omega)$ for any $p<\infty$).

\begin{definition}
For $n \ge 0$ let $\overline{\mathcal{P}}_n(\mathcal{H}_1)$ be the closure in $L^2(\Omega, \mathcal{F}, \mathbb{P})$ of the linear space
\begin{equation*}
\mathcal{P}_n(\mathcal{H}_1):= \left\{p(\xi_1,...,\xi_m): \ p \ \text{is a polynomial of degree $\le n$}, \ \xi_1,..., \xi_m \in \mathcal{H}_1, \ m < \infty \right\}
\end{equation*}
and let 
\begin{equation*}
\mathcal{H}_n:= \overline{\mathcal{P}}_n(\mathcal{H}_1) \ominus  \overline{\mathcal{P}}_{n-1}(\mathcal{H}_1) =  \overline{\mathcal{P}}_n(\mathcal{H}_1) \cap  \overline{\mathcal{P}}_{n-1}(\mathcal{H}_1)^{\perp}.
\end{equation*}
The space $\mathcal{H}_n$ is called $n$-th Wiener Chaos (associated to $\mathcal{H}_1$).
\end{definition}

Plainly,  $\mathcal{H}_0= \overline{\mathcal{P}}_0(\mathcal{H}_1)=\mathbb{R}$ and $\mathcal{H}_1$ is the closure in 
$L^2(\Omega)$ of the set $$\left\{\sum_{i=1}^m c_i \xi_i, \ \xi_1,...,\xi_m  \in \mathcal{H}_1, \ c_1,...,c_m \in \mathbb{R}, \ m < \infty\right\}$$ and this  representation is consistent with Proposition \ref{lin_span_ex}.

\begin{remark}
\label{rem_complex}
In the above definition we may consider either real or complex spaces (we use the superscripts $\mathbb{R}$ and $\mathbb{C}$ to distinguish the two cases). The relation between real and complex cases is as follows (see \cite[Theorem 2.2]{Janson}. If $\mathcal{H}_1$ is a Gaussian Hilbert space, then $\mathcal{H}_0^\mathbb{R}=\mathbb{R}$, $\mathcal{H}_0^\mathbb{C}=\mathbb{C}$, $\mathcal{H}_1^\mathbb{R}=\mathcal{H}_1$, $\mathcal{H}_1^\mathbb{C}=\mathcal{H}_1+ \emph{i} \mathcal{H}_1$, and more generally, $\mathcal{H}_n^\mathbb{C}=\mathcal{H}_n+ \emph{i} \mathcal{H}_n$, the complexification of $\mathcal{H}_n^\mathbb{R}$ consisting of all complex-valued random variables whose real and imaginary parts belong to $\mathcal{H}_n^\mathbb{R}$.
\end{remark}

\begin{remark}
\label{ort_basis}
In the definition of $\mathcal{P}_n(\mathcal{H}_1)$ we may assume, without loss of generality, that the random variables $\xi_1,..., \xi_m$ are orthonormal.
In fact, if $\xi_1,..., \xi_m \in \mathcal{H}_1$, then there exists an orthonormal sequence $\zeta_1,..\zeta_n\in \mathcal{H}_1$ such that each $\xi_i$ is a linear combination of the random variables $\zeta_i$. Therefore, any polynomial of $\xi_1,...\xi_m$ can be written as a polynomial of at most the same degree in  $\zeta_1,..\zeta_n$.
\end{remark}

\begin{remark}
\label{rem_Janson}
If $\mathcal{H}_1$ is a finite-dimensional space, then $\mathcal{P}_n(\mathcal{H}_1)$ 
is also finite-dimensional. In particular, $\mathcal{P}_n(\mathcal{H}_1)$ is already closed, thus $\overline{\mathcal{P}}_n(\mathcal{H}_1)=\mathcal{P}_n(\mathcal{H}_1)$ and it is superfluous to take the closure in the definition of $\mathcal{H}_n$. On the other hand, if $\mathcal{H}_1$ has infinite dimension, taking the closure is essential. For example, as shown in \cite[Example 6.5]{Janson}, if $\{\xi_i \}_{i=1}^{\infty}$ is an orthonormal sequence in $\mathcal{H}_1$, then $\sum_{i=1}^{\infty}2^{-i}\xi_i^2$ belongs to $\overline{\mathcal{P}}_2(\mathcal{H}_1)$ but it is not a polynomial in any finite set of variables in $\mathcal{H}_1$.
\end{remark}

By construction, $\left\{ \overline{\mathcal{P}}_n(\mathcal{H}_1)\right\}_{n=0}^{\infty}$ is an increasing sequence of closed subspaces of $L^2(\Omega)$ and, if $n \ne m$, then $\mathcal{H}_n$ and $\mathcal{H}_m$ are orthogonal for the usual inner product of $L^2(\Omega)$. Moreover, 
\begin{equation*}
\overline{\mathcal{P}}_n(\mathcal{H}_1)= \bigoplus_{k=0}^n \mathcal{H}_k.
\end{equation*}
As a consequence,
\begin{equation}
\label{eq_big+}
\bigoplus_{n=0}^\infty \mathcal{H}_n=\overline{\bigcup_{n=0}^\infty \overline{\mathcal{P}}_n(\mathcal{H}_1)}.
\end{equation}
The next result, known as the Wiener chaotic decomposition of  $L^2(\Omega)$, shows that $\bigoplus_{n=0}^\infty \mathcal{H}_n$ coincides with $L^2(\Omega, \sigma(\mathcal{H}_1), \mathbb{P})$, where we recall that $\sigma(\mathcal{H}_1)$ is the $\sigma$-field generated by the random variables in $\mathcal{H}_1$.

\begin{theorem}
\label{WCD}
[\textbf{Wiener Chaos Decomposition}]
The spaces $\mathcal{H}_n$, $n \ge 0$, are mutually orthogonal, closed subspaces of $L^2(\Omega, \mathcal{F}, \mathbb{P})$ and 
\begin{equation*}
\bigoplus_{n=0}^\infty \mathcal{H}_n=L^2(\Omega, \sigma(\mathcal{H}_1), \mathbb{P}).
\end{equation*}
\end{theorem}

The proof of Theorem \ref{WCD} relies on the following preliminary result.
\begin{lemma}
\label{prel_WCD}
If $Y \in L^1(\Omega, \sigma(\mathcal{H}_1), \mathbb{P})$ and $\mathbb{E}\left[ Y e^{i\xi}\right]=0$ for every $\xi \in \mathcal{H}_1$, then $Y=0$ $\mathbb{P}$-a.s.
\end{lemma}
\begin{proof}
Let $\{\xi_n\}_{n \in \mathbb{N}}$ be an orthonormal basis in $\mathcal{H}_1$. \begin{footnote}{We are considering here the case of a \emph{separable} Gaussian Hilbert space $\mathcal{H}_1$. The same proof (simpler) works when dim$\mathcal{H}_1< \infty$.}\end{footnote} Fix $n \ge 1$ and $\xi_1,..., \xi_n \in \mathcal{H}_1$. Let us denote by $\mu$ the distribution of the Gaussian vector $\xi:=(\xi_1,...,\xi_n)$ on $\mathbb{R}^n$ and by $\mathcal{G}^n= \sigma(\xi_1,..., \xi_n)$ the $\sigma$-field generated by $\xi_1,..., \xi_n$; obviously $\mathcal{G}^n \subset \sigma(\mathcal{H}_1)$. For a given 
complex-valued
$\varphi \in L^1(\mathbb{R}^n, \mu)$ we can write
\begin{align}
\label{est1lem}
\mathbb{E}\left[Y \varphi(\xi_1,...,\xi_n)\right]
= \mathbb{E}\left[\mathbb{E}\left[ Y\varphi(\xi_1,..,\xi_n)|\mathcal{G}^n\right] \right]
= \mathbb{E}\left[\varphi(\xi_1,..,\xi_n)\mathbb{E}\left[ Y|\mathcal{G}^n\right] \right],
\end{align}
where in the last equality we used the fact that $\varphi(\xi_1,..,\xi_n)$ is a $\mathcal{G}^n$-measurable random variable. Since $\mathbb{E}\left[ Y|\mathcal{G}^n\right]$ is $\mathcal{G}^n$-measurable by definition, there exists a measurable function $\psi: \mathbb{R}^n\rightarrow\mathbb{R}$ such that $\mathbb{E}\left[ Y|\mathcal{G}^n\right]=\psi(\xi_1,...,\xi_n)$.
Combined with \eqref{est1lem} this gives
\begin{equation}
\label{est2lem}
\mathbb{E}\left[Y \varphi(\xi_1,...,\xi_n)\right]
=\mathbb{E}\left[\varphi(\xi_1,..,\xi_n)\psi(\xi_1,..,\xi_n)\right]
=\int_{\mathbb{R}^n}\varphi(z_1,..,z_n)\psi(z_1,..,z_n)\,{\rm d}\mu(z).
\end{equation}
At this point we make two appropriate choices of the function $\varphi$.
As first choice we take $\varphi(\cdot)=\pmb{1}_B(\cdot)$ with $B \in \mathcal{B}(\mathbb{R}^n)$: \eqref{est2lem} yields
\begin{equation*}
\mathbb{E}\left[Y \pmb{1}_B(\xi)\right]
=\int_B\psi(z)\,{\rm d}\mu(z)
\end{equation*}
and we introduce the measure $\nu$ on $\mathbb{R}^n$ as 
\begin{equation}
\label{est3lem}
\nu (B):= \mathbb{E}\left[Y \pmb{1}_B(\xi)\right], \quad B \in \mathcal{B}(\mathbb{R}^n).
\end{equation}
As second choice, given $t=(t_1,...,t_n) \in \mathbb{R}^n$, we consider $\varphi(\cdot)=e^{i \langle t, \cdot\rangle_{\mathbb{R}^n}}
$. Since $\mathcal{H}_1$ is an Hilbert space and $\xi_1,...,\xi_n \in \mathcal{H}_1$, it immediately follows that $\langle t, \xi \rangle \in \mathcal{H}_1$ and thus, by assumption, we have that 
\begin{equation}
\label{est4lem}
\mathbb{E}\left[Ye^{i\langle t, \xi \rangle}\right]=0.
\end{equation}
On the other hand, \eqref{est2lem} and \eqref{est3lem} yield
\begin{equation}
\label{est5lem}
\mathbb{E}\left[Ye^{i\langle t, \xi \rangle}\right]
= \int_{\mathbb{R}^n} e^{i\langle t, z \rangle}\psi (z)\, {\rm d}\mu(z)
=\int_{\mathbb{R}^n} e^{i\langle t, z \rangle}\, {\rm d}\nu(z) =\widehat {\nu}(t).
\end{equation}
From \eqref{est4lem} and \eqref{est5lem} we thus conclude that $\widehat {\nu}=0$, that is the measure $\nu$ has Fourier transform that vanishes identically, which implies that $\nu \equiv 0$, by the injectivity of the Fourier tansform. Thus, by the definition of $\nu$ given in \eqref{est3lem}, 
\begin{equation}
\label{est6lem}
\mathbb{E}\left[Y\pmb{1}_{(\xi \in B)}\right]=0 \qquad \forall \ B \in \mathcal{B}(\mathbb{R}^n).
\end{equation}

Let us now call $\Omega \supset G^n_B:=\left\{ (\xi \in B) \subset \Omega: \ B \in \mathcal{B}(\mathbb{R}^n) \right\}
$; we have that $\sigma(G^n_B)=\mathcal{G}^n \subset \sigma(\mathcal{H}_1)$. Let us introduce the family $\mathcal{A}:=\left\{ \Gamma \subset \Omega \ : \ \mathbb{E}\left[Y \pmb{1}_\Gamma \right]=0\right\}$. It is easy to prove that $\mathcal{A}$ is a $\sigma$-field and it is non empty since $ G^n_B \in \mathcal{A}$ in virtue of \eqref{est6lem}. As a consequence, $\mathcal{A}$ contains $\mathcal{G}^n$. 

At this point let $n$ vary: for any $n \ge 1$ we have $\mathcal{G}^n \subset \mathcal{A}$, that is $\{\mathcal{G}_{n}\}_{n \in \mathbb{N}}\subset \mathcal{A}$. Since the family of $\sigma$-fields \{$\mathcal{G}^n\}_{n \in \mathbb{N}}$ generates the $\sigma$-field $\sigma(\mathcal{H}_1)$, it follows that 
$\sigma(\mathcal{H}_1)\subset \mathcal{A}$.
Therefore we obtain
\begin{equation*}
\mathbb{E}\left[Y\pmb{1}_{\mathcal{O}}\right]=0 \qquad \forall \ \mathcal{O} \in \sigma(\mathcal{H}_1),
\end{equation*}
and by approximation $\mathbb{E}\left[YX\right]=0$ for any $\sigma(\mathcal{H}_1)$-measurable random variable $X$. This implies $Y=0$ $\mathbb{P}$-a.s. which concludes the proof.
\end{proof}

Let us come to the proof of Theorem \ref{WCD}.
\begin{proof}[Proof of Theorem \ref{WCD}]
Let $K:=\overline{\bigcup_{n=0}^\infty \overline{\mathcal{P}}_n(\mathcal{H}_1)}$. It is evident that $\overline{\mathcal{P}}_n(\mathcal{H}_1)\subseteq L^2(\Omega, \sigma(\mathcal{H}_1), \mathbb{P})$ for any $n \ge 0$ and thus $K\subseteq L^2(\Omega, \sigma(\mathcal{H}_1), \mathbb{P})$. It only remains to prove that equality holds. In virtue of \eqref{eq_big+}, this is equivalent to show that, if $Y \in L^2(\Omega, \sigma(\mathcal{H}_1), \mathbb{P})$ is orthogonal to $K$, then $Y=0$.
It is simplest to show this first for the complex case, the real case then follows from Remark \ref{rem_complex}. Hence we assume for the rest of the proof that we are using complex scalars. 

We start by proving that, whenever $\xi \in \mathcal{H}_1$, $e^{i\xi} \in K$: let $\xi \in \mathcal{H}_1$, we have
\begin{equation}
\label{est_ei}
\left|e^{i\xi}-\sum_{k=0}^n\frac{(i\xi)^k}{k!}\right|  \le 1 + \sum_{k=0}^n\frac{|\xi|^k}{k!} \le 1+ e^{|\xi|} \le 1+ e^{\xi}+e^{-\xi}.
\end{equation}
Since the r.h.s. of \eqref{est_ei} belongs to $L^2(\Omega)$
and the l.h.s. tends to zero pointwise as $n \rightarrow \infty$, the Dominated Convergence Theorem yields $\sum_{k=0}^n \frac{(i\xi)^k}{k!} \rightarrow e^{i \xi}$ in $L^2$. Since $\xi^k \in \mathcal{P}_k(\mathcal{H}_1)\subset K$ for any $k \ge 0 $, we have that $e^{i\xi} \in K$ whenever $\xi \in \mathcal{H}_1$, as wanted.

Consequently, if $Y \in L^2(\Omega, \sigma(\mathcal{H}_1), \mathbb{P})$ is orthogonal to $K$, then in particular (since we have shown above that $e^{i\xi} \in K$ whenever $\xi \in \mathcal{H}_1$),
\begin{equation*}
\langle Y, e^{i\xi}\rangle_{L^2}=\mathbb{E}\left[ Y e^{i\xi}\right]=0, \quad \forall \ \xi \in \mathcal{H}_1.
\end{equation*}
In virtue of Lemma \ref{prel_WCD} we have that $Y=0$ $\mathbb{P}$-a.s. and this concludes the proof.
\end{proof}

\begin{corollary}
\label{corWCD}
If the Gaussian Hilbert space $\mathcal{H}_1$ generates the $\sigma$-field $\mathcal{F}$, then $L^2(\Omega, \mathcal{F}, \mathbb{P})$ has the orthogonal decomposition 
\begin{equation*}
L^2(\Omega, \mathcal{F}, \mathbb{P})=\bigoplus_{n=0}^\infty \mathcal{H}_n.
\end{equation*}
\end{corollary}


In view of Theorem \ref{WCD} every random variable $X$ in $L^2(\Omega)$ admits a unique expansion. For $n \ge 0$ let us denote by $J_n$ the orthogonal projection of $L^2(\Omega)$ onto $\mathcal{H}_n$; in particular, $J_0(X)=\mathbb{E}\left[X\right]$. Then, every random variable $X \in L^2(\Omega, \sigma(\mathcal{H}_1), \mathbb{P})$ admits the unique expansion of the type
\begin{equation*}
X=\mathbb{E}\left[X\right] + \sum_{n=1}^{\infty} J_n(X),
\end{equation*}
with the series converging in $L^2(\Omega)$.

Another consequence of Theorem \ref{WCD} is the fact that the set of polynomials in elements of $\mathcal{H}_1$ is dense in $L^2(\Omega)$. We define the space of polynomials in elements of $\mathcal{H}_1$ as 
\begin{equation*}
\mathcal{P}(\mathcal{H}_1):= \bigcup_{n=0}^\infty \mathcal{P}_n(\mathcal{H}_1).
\end{equation*}
We aso define 
\[
\bar {\mathcal{P}}(\mathcal{H}_1) := \bigcup_{n=0}^\infty \bar {\mathcal{P}}_n(\mathcal{H}_1)=\bigoplus_{n=0}^\infty \mathcal{H}_n,
\]
as the space of all elements in $L^2(\Omega, \sigma(\mathcal{H}_1), \mathbb{P})$ having finite chaos decomposition. Notice that, if $\mathcal{H}_1$ has finite dimension, then $\bar {\mathcal{P}}(\mathcal{H}_1)$ equals the space ${\mathcal{P}}(\mathcal{H}_1)$, but if $\mathcal{H}_1$ is infinite dimensional, then $\bar{\mathcal{P}}(\mathcal{H}_1)$ is strictly larger, see Remark \ref{rem_Janson}.

\begin{corollary}
\label{cor_den_pol}
The set of polynomial variables $\mathcal{P}(\mathcal{H}_1)$ is dense in $L^2(\Omega, \sigma(\mathcal{H}_1), \mathbb{P})$.
\end{corollary}
\begin{proof}
As an immediate consequence of Theorem \ref{WCD} we have that  $\bar {\mathcal{P}}(\mathcal{H}_1) 
$ is dense in $L^2(\Omega, \sigma(\mathcal{H}_1), \mathbb{P})$. Since $
\overline{\bigcup_{n=0}^\infty \mathcal{P}_n(\mathcal{H}_1)}=\overline{\bigcup_{n=0}^\infty \overline{\mathcal{P}}_n(\mathcal{H}_1)}$, the result immediately follows. 
\end{proof}

More generally, the following result holds true.
\begin{proposition}
\label{PdenseLp}
If $0 < p< \infty$, then the set of polynomial variables $\mathcal{P}(\mathcal{H}_1)$ is a dense subspace of $L^p(\Omega, \sigma(\mathcal{H}_1), \mathbb{P})$.
\end{proposition}
\begin{proof}
The polynomial variables belong to $L^p(\Omega)$ thanks to the H\"older inequality. In order to show denseness for $1 \le p < \infty$ one can proceed as in the proof of Theorem \ref{WCD}. If $0< p<1$, the result follows from the $L^2$ case since $L^2(\Omega, \sigma(\mathcal{H}_1), \mathbb{P})$ is a dense subspace of $L^p(\Omega, \sigma(\mathcal{H}_1), \mathbb{P})$. \end{proof}

\subsection{Hermite Polynomials}
\label{Hermite_sec}
Explicit orthonormal bases of each Wiener Chaos $\mathcal{H}_n$ may be constructed; in view of Theorem \ref{WCD} their union gives an explicit orthonormal basis of $L^2(\Omega, \sigma(\mathcal{H}_1), \mathbb{P})$. To this aim, we introduce the notion of Hermite polynomials. For more details see e.g. \cite[Chapter 1.1.1]{Nualart} and \cite[Chapter 8.1]{Lunardi}.

\begin{definition}
The sequence of \emph{Hermite polynomials} in $\mathbb{R}$ is defined by
\begin{equation}
\label{H_n}
H_n(x):=\frac{(-1)^n}{n!}e^{\frac{x^2}{2}}\frac{d^n}{dx^n}e^{-\frac{x^2}{2}}, \quad n \ge 1, \ x \in \mathbb{R},
\end{equation}
and $H_0(x)=1$.
\end{definition}

Thus,
\begin{equation*}
H_0(x)=1, \quad H_1(x)=x, \quad H_2(x)=\frac{x^2-1}{2}, \quad H_3(x)=\frac{x^3-3x}{6},
\end{equation*}
\begin{equation*}
H_4(x)=\frac{x^4-6x^2+3}{4!}, \quad H_5(x)=\frac{x^5-10x^3 +15x}{5!}, \quad \text{etc.}
\end{equation*}

\begin{remark}
Beware that the definition given here may differ from the one given in other references  by a factor $n!$, see e.g. \cite{Malliavin}.
\end{remark}
These polynomials are the coefficients of the expansion in powers of $t$ of the function $F(x,t):= e^{tx-\frac{t^2}{2}}$. In fact, 
\begin{equation*}
F(x,t)= e^{\frac{x^2}{2}-\frac12(x-t)^2}
= e^{\frac{x^2}{2}}\sum_{n=0}^\infty \frac{t^n}{n!} \left(\frac{d^n}{dt^n}e^{-\frac{(x-t)^2}{2}} \right)|_{t=0}
=\sum_{n=0}^\infty t^n H_n(x).
\end{equation*}
Using this development it is not difficult to prove the following properties of Hermite polynomials, that hold for any $n \ge 1$,
\begin{equation*}
H_n'(x)=H_{n-1}(x),
\end{equation*}
\begin{equation*}
(n+1)H_{n+1}(x)=nH_n(x)-H_{n-1}(x), 
\end{equation*}
\begin{equation*}
H_n(-x)=(-1)^n H_n(x).
\end{equation*}
Next, we define the Hermite polynomials on $\mathbb{R}^d$.
\begin{definition}
If $\alpha =(\alpha_1,..., \alpha_d) \in (\mathbb{N} \cup \{0\})^d$ is a multi-index, we define the Hermite polynomials $H_\alpha$ by 
\begin{equation*}
H_\alpha(x)=\prod_{i=1}^d H_{\alpha_i}(x_i), \qquad x=(x_1,..., x_d) \in \mathbb{R}^d.
\end{equation*}
\end{definition}


Let us now see how to define Hermite polynomials in the infinite-dimensional case.
Let $\mathcal{H}$ be a separable Hilbert space. We introduce the set $\Lambda$ of multi-indices $\alpha \in (\mathbb{N} \cup \{0\})^{\mathbb{N}}$, $\alpha=(\alpha_i)_i$, such that all the terms, except a finite number of them, vanish. For $\alpha \in \Lambda$ we set $\alpha!:= \prod_{i=1}^\infty a_i!$ and $|\alpha|=\sum_{i=1}^\infty a_i$.
\begin{definition}
For any multi-index $\alpha \in \Lambda$ we define the \emph{generalized Hermite polynomial} $H_\alpha(x)$, 
 by
\begin{equation*}
H_\alpha(x)=\prod_{i=1}^\infty H_{\alpha_i}(x_i), \qquad \alpha \in \Lambda, \quad x \in \mathbb{R}^{\mathbb{N}}.
\end{equation*}
\end{definition}
Notice that the above product is well defined since $H_0=1$ and $\alpha_i \ne 0$ only for a finite number of indices.

Let us now clarify the connection between the spaces $\mathcal{H}_n$ and the Hermite polynomials.
On the probability space $(\Omega, \mathcal{F}, \mathbb{P})$ let us consider the separable Hilbert space $\mathcal{H}_1$. Let $\{\xi_i\}_{i \in \mathbb{N}}$ be an orthonormal basis in $\mathcal{H}_1$. For any $\alpha \in \Lambda$ let us define
\begin{equation*}
\Phi_\alpha := \sqrt{\alpha!} \prod_{i=1}^\infty H_{\alpha_i}(\xi_i).
\end{equation*}
\begin{proposition}
For any $n \ge 1$ the random variables 
\begin{equation*}
\{\Phi_\alpha, \alpha \in \Lambda, |\alpha|=n\}
\end{equation*}
form a complete orthonormal system in $\mathcal{H}_n$.
\end{proposition}
\begin{proof}
See \cite[Proposition 1.1.1]{Nualart}. 
\end{proof}
As an immediate consequence of Theorem \ref{WCD} we obtain the following result.
\begin{corollary}
The family $\{\Phi_\alpha, \alpha \in \Lambda\}$ is a complete orthonormal system in $L^2(\Omega, \sigma(\mathcal{H}_1), \mathbb{P})$.
\end{corollary}

\begin{example}
\label{Her_ex}
As in Example \ref{1DGauss} consider the probability space $(\mathbb{R}, \mathcal{B}(\mathbb{R}), \gamma)$, where $\gamma$ is a standard Gaussian measure, and the Gaussian Hilbert space $\mathcal{H}_1 = \left\{ \lambda \xi: \ \lambda \in \mathbb{R}\right\}$ defined on it. In this case the space $\overline{\mathcal{P}_n}(\mathcal{H}_1)=\mathcal{P}_n(\mathcal{H}_1)$ is the space of polynomials of degree at most $n$ in the variable $\xi$. The space $\overline{\mathcal{P}_n}(\mathcal{H}_1)$ has dimension $n+1$ and the space $\mathcal{H}_n$ is one-dimensional. 
Recalling \eqref{H_n}, for any $n\ge 0$, $\mathcal{H}_n=\emph{Span}(\sqrt{n!}H_n(x))$ and 
the family $\{\sqrt{n!}H_n(x)\}_{n \in \mathbb{N}}$ is an orthonormal basis in  $L^2(\mathbb{R}, \mathcal{B}(\mathbb{R}), \gamma)$.
\end{example}


\begin{remark}
Notice that, in contrast to the finite-dimensional case, when $\mathcal{H}_1$ is infinite dimensional, for any fixed $n \in \mathbb{N}$, there are infinitelely many Hermite polynomials $H_\alpha$, with $|\alpha|=n$, so that $\mathcal{H}_n$ is infinite dimensional.
\end{remark}

\section{The Ornstein-Uhlenbeck semigroup and its infinitesimal generator}
\label{OU_sec}


In this Section we introduce the Ornstein-Uhlenbeck semigroup and its infinitesimal generator. Their construction relies on the Wiener Chaos decomposition. See Appendix \ref{semigroup_sec} for basic notions and results on semigroups.

We assume that $\mathcal{H}_1$ is a Gaussian Hilbert space defined on a complete probability space $(\Omega, \mathcal{F}, \mathbb{P})$, and that $\mathcal{F}$ is the $\sigma$-field generated by $\mathcal{H}_1$. We recall that $J_n$ denotes the orthogonal projection on the $n$-th Wiener chaos.

\subsection{The semigroup of Ornstein-Uhlenbeck}

 \begin{proposition}
 \label{OUS_pro}
 The family of operators $\{T_t\}_{t \ge 0}$ defined as
 \begin{equation*}
 T_tF:= \sum_{n=0}^\infty e^{-nt}J_n F, \quad t \ge 0, \ F \in L^2(\Omega),
 \end{equation*}
 is a one-parameter $C_0$-semigroup of contraction operators in $L^2(\Omega)$.
 \end{proposition}
 
\begin{definition}
\label{OU_def}
The semigroup introduced in Proposition \ref{OUS_pro} is called \emph{Ornstein-Uhlenbeck semigroup}.
 \end{definition}
\begin{proof}[Proof of Proposition \ref{OUS_pro}]
We start by noticing that the definition of $T_tF$ is well posed, that is the series converges in $L^2(\Omega)$. We have
\begin{align*}
\left \Vert \sum_{n=0}^\infty e^{-nt}J_nF\right\Vert^2_{L^2(\Omega)}
&=\mathbb{E}\left[\sum_{n=0}^\infty e^{-2nt}|J_nF|^2 \right]
\le \mathbb{E}\left[\sum_{n=0}^\infty|J_nF|^2 \right] =\|F\|_{L^2(\Omega)}^2< \infty.
\end{align*}
The above inequality also shows that $\|T_t\|_{\mathcal{L}(L^2(\Omega))} \le 1$, that is the family of operators $\{T_t\}_{t \ge 0}$ is contracting.
Let us now prove that $\{T_t\}_{t \ge 0}$ is a $C_0$-semigroup.
\begin{itemize}
\item [(S1)] \underline{$T_0=I$.} 
\\
For any $F \in L^2(\Omega)$,
\begin{equation*}
T_0F=\sum_{n=0}^\infty J_n F=F,
\end{equation*}
as an immediate consequence of Theorem \ref{WCD}.
\item [(S2)] \underline{$T_tT_s=T_{t+s}$.} 
\\
For any $s,t \ge0$, for any $F \in L^2(\Omega)$, using the fact that $J_n$ is an orthogonal projection (thus $J_n^2=J_n$), we get
\begin{equation*}
T_tT_sF=\sum_{n=0}^\infty e^{-nt}J_n \left( \sum_{m=0}^\infty e^{-ms}J_mF \right)
=\sum_{n=0}^\infty e^{-nt}J_n(e^{-ns}J_nF)=\sum_{n=0}^\infty e^{-(t+s)n}J_nF=T_{t+s}F.
\end{equation*}
\item [(S3)] \underline{$\{T_t\}_{t \ge 0}$ is a $C_0$-semigroup.} 
\\
Let $F \in L^2(\Omega)$, from Theorem \ref{WCD} we know that $\|F\|^2_{L^2(\Omega)}=\sum_{n=0}^\infty \mathbb{E}\left[|J_nF|^2\right]< \infty.$
Thus, for a fixed $\varepsilon>0$, there exists $M_{\varepsilon}\in \mathbb{N}$ such that 
\begin{equation}
\label{vare}
\sum_{n=M_{\varepsilon}+1}^\infty\mathbb{E}\left[ |J_nF|^2\right] < \varepsilon.
\end{equation}
Bearing in mind \eqref{vare} we compute, for $F \in L^2(\Omega)$
\begin{align*}
\limsup_{t \rightarrow \infty}& \|T_tF-F\|^2_{L^2(\Omega)} 
=\limsup_{t \rightarrow \infty}\mathbb{E}\left[ \left| \sum_{n=0}^\infty(e^{-nt}-1)J_nF\right|^2\right]
\\
&=\limsup_{t \rightarrow \infty}\left(\sum_{n=0}^{M_\varepsilon}(e^{-nt}-1)^2\mathbb{E}\left[|J_nF|^2\right]+ \sum_{n=M_\varepsilon+1}^\infty(e^{-nt}-1)^2\mathbb{E}\left[|J_nF|^2\right]\right)
\\
& \le \varepsilon + \sum_{n=0}^{M_\varepsilon}\limsup_{t \rightarrow \infty}(e^{-nt}-1)^2\mathbb{E}\left[|J_nF|^2\right]= \varepsilon.
\end{align*}
By the arbitrariness of $\varepsilon$ we conclude that
\begin{equation*}
\exists \ \lim_{t \rightarrow \infty} \|T_tF-F\|_{L^2(\Omega)}=0,
\end{equation*}
that is $\{T_t\}_{t \ge 0}$ is a $C_0$-semigroup on $L^2(\Omega)$.
\end{itemize}
\end{proof}

\subsection{The generator of the Ornstein-Uhlenbeck semigroup}

\begin{definition}
Let $F \in L^2(\Omega)$. We define the operator $L$ as follows:
\begin{equation}
\label{L_op}
LF:=\sum_{n=0}^\infty-nJ_nF,
\end{equation}
provided the series converges in $L^2(\Omega)$. The domain of the operator is the set
\begin{equation}
\label{Dom_L}
\text{\emph{Dom}}(L):=\{F \in L^2(\Omega): \ \sum_{n=0}^\infty n^2 \mathbb{E}\left[ |J_nF|^2\right] < \infty\}.
\end{equation}
\end{definition}

Notice that $L$ is an unbounded self-adjoint operator on $L^2(\Omega)$, i.e. $\mathbb{E}\left[FLG\right]=\mathbb{E}\left[GLF\right]$ for any $F,G \in$ Dom$(L)$, and hence it is closed.

The next result shows that $L$ coincides with the infinitesimal generator of the Ornstein-Uhlenbeck semigroup $\{T_t\}_{t \ge 0}$ introduced in Definition \ref{OU_def}.

\begin{proposition}
The operator $L$ coincides with the infinitesimal generator of the Ornstein-Uhlenbeck semigroup $\{T_t\}_{t \ge 0}$.
\end{proposition}
\begin{proof}
We have to show that $F\in L^2(\Omega)$ belongs to the domain of $L$ if and only the limit $\lim_{t \rightarrow 0}\frac{T_tF-F}{t}$ exists in $L^2(\Omega)$ and, in this case, the limit is equal to $LF$.

Let us first assume that $F \in$ Dom$(L)$. Then, for a fixed $\varepsilon>0$ there exists $M_\varepsilon \in \mathbb{N}$ such that $\sum_{n=M_\varepsilon+1}^\infty n^2 \mathbb{E}\left[|J_nF|^2\right] < \varepsilon$. We compute
\begin{align*}
\limsup_{t \rightarrow 0} &\left\Vert \frac{T_tF-F}{t}-LF\right\Vert^2_{L^2(\Omega)}
=\limsup_{t \rightarrow 0} \mathbb{E}\left[\left|\sum_{n=0}^\infty \left( \frac{e^{-nt}-1}{t}+n \right) J_nF\right|^2 \right]
\\
&=\limsup_{t \rightarrow 0}\left(\sum_{n=0}^{M_\varepsilon}\left(\frac{e^{-nt}-1}{t}+n \right)^2\mathbb{E}\left[|J_nF|^2 \right] + \sum_{n=M_\varepsilon+1}^\infty\left(\frac{e^{-nt}-1}{t}+n \right)^2\mathbb{E}\left[|J_nF|^2 \right] \right)
\\
&< \sum_{n=0}^{M_\varepsilon} \limsup_{t \rightarrow 0}\left(\frac{e^{-nt}-1}{t}+n \right)^2\mathbb{E}\left[|J_nF|^2 \right] + 4\varepsilon = 4\varepsilon,
\end{align*}
where in the last inequality above we used the estimate 
\begin{align*}
\left|\frac{e^{-nt}-1}{t}\right| & = \frac 1t \left|\int_0^t-ne^{-ns}\, {\rm d}s\right| 
\le \frac{n}{t} \int_0^t e^{-ns}\, {\rm d}s \le n 
\\
& \Longrightarrow \ 
 \sum_{n=M_\varepsilon+1}^\infty\left(\frac{e^{-nt}-1}{t}+n \right)^2\mathbb{E}\left[|J_nF|^2 \right] \le 4 \sum_{n=M_\varepsilon+1}^\infty n^2 \mathbb{E}\left[|J_nF|^2\right] < 4 \varepsilon.
\end{align*}
By the arbitrariness of $\varepsilon$ we conclude that the limit $\lim_{t \rightarrow 0}\frac{T_tF-F}{t}$ exists in $L^2(\Omega)$ and it is qual to $LF$.

Conversely, let us suppose that $F \in L^2(\Omega)$ is such that there exists $\lim_{t \rightarrow 0} \frac{T_tF-F}{t}=AF=:G$ in $L^2(\Omega)$. 
For any $n \in \mathbb{N}$, from the continuity of $J_n$ we infer the existence of a subsequence $\{t_k\}_{k \in \mathbb{N}}$ such that $J_n\left( \frac{T_{t_k}F-F}{t_k}\right) \rightarrow J_nG$ $\mathbb{P}$-a.s. as $t_k \rightarrow \infty$. On the other hand, $J_n\left( \frac{T_tF-F}{t}\right) = \left(\frac{e^{-nt}-1}{t} \right)J_nF\rightarrow -nJ_nF$ $\mathbb{P}$-a.s. as $t \rightarrow \infty$. Thus we infer $J_nG=-nJ_nF$ for any $n \in \mathbb{N}$ which yields $F \in$ Dom$(L)$ and $LF=G$.
This concludes the proof.
\end{proof}

\begin{remark}
The Ornstein-Uhlenbeck operator admits the spectral decomposition \eqref{L_op}, where $J_n$ is the orthogonal projection of $L^2(\Omega)$ onto the $n^{th}$ Wiener Chaos $\mathcal{H}_n$. The space $\mathcal{H}_n$ is thus the eigenspace of $L$ with eigenvaule $-n$, see e.g. \cite[Proposition 14.12]{Lunardi}. We recall (see Section \ref{Hermite_sec}) that the space $\mathcal{H}_n$ is generated by the (normalized) generalized Hermite polynomials $\Phi_\alpha$ with $\alpha \in \Lambda$, $|\alpha|=n$. The Hermite polynomials $\Phi_\alpha$ with $\alpha \in \Lambda$, $|\alpha|=n$ are thus the eigenfunctions of $L$ with eigenvalue $-|\alpha|$, that is (see e.g. \cite[Proposition 14.1.1]{Lunardi})
\begin{equation*}
\forall \ \alpha \in \Lambda \quad \Phi_\alpha \in \emph{Dom}(L) \ \text{and} \ L\Phi_\alpha =-|\alpha|\Phi_\alpha.
\end{equation*}
\end{remark}
\begin{example}
In the framework of Example \ref{Her_ex}, we have that $\sqrt{n}H_n$ is an eigenfunction of the one-dimensional Ornstein-Uhlenbeck operator.
\end{example}




\section{The Malliavin derivative}
\label{Malliavin_sec}
In this section we introduce the notion of Malliavin derivative.
Let $(\Omega, \mathcal{F}, \mathbb{P})$ be a complete probability space and let $\mathcal{H}_1$ be a (infinite dimensional) separable Gaussian Hilbert space
. From now on we will assume $\mathcal{F}$ to be the $\sigma$-field generated by $\mathcal{H}_1$.
We aim to introduce the derivative $DF$ of a square integrable random variable $F:(\Omega, \sigma(\mathcal{H}_1), \mathbb{P}) \rightarrow \mathbb{R}$, namely we want to differentiate $F$ with respect to $\omega \in \Omega$. Notice that, any random variable $F$ measurable w.r.t. the $\sigma$-algebra generated by $\mathcal{H}_1$ can be viewed as a function of the elements in $\mathcal{H}_1$
. Usually, in the concrete situations, $\Omega$ is a topological vector space and the Malliavin derivative operator can be introduced as a differential operator (see Section \ref{top_sec} for more details). Nevertheless, as done for instance in \cite{Nualart} (see also \cite{NP}), it is possible to introduce a notion of Malliavin derivative without assuming any topological structure on the probability space $\Omega$. This approach proves particularly flexible and useful in many applications. Moreover, it is general enough to admit as special cases the definitions of Malliavin derivative given in probability spaces with a topological structure (see Section \ref{top_sec}).
\\
We will deal at first with the notion of Malliavin derivative of real-valued random variables, then we will show how to extend the definition to Hilbert-valued random variables.

\subsection{The Malliavin derivative of real-valued random variables}

In order to introduce the Malliavin derivative operator we will need to characterize the Gaussian Hilbert space $\mathcal{H}_1$ as the range of a given separable Hilbert space $\mathcal{H}$ through a unitary operator $W$ (see Proposition \ref{Gaussian_isometries}). Let thus $\mathcal{H}$ be an (infinite dimensional) real separable Hilbert space with scalar product denoted by $\langle \cdot, \cdot\rangle_{\mathcal{H}}$ (the norm of an element $h \in \mathcal{H}$ will be denoted by $\|h\|_{\mathcal{H}}$) and let 
\begin{equation*}
W:\mathcal{H} \rightarrow \mathcal{H}_1\subset L^2(\Omega, \sigma(\mathcal{H}_1), \mathbb{P})
\end{equation*}
be a unitary operator i.e. a 
 linear and surjective operator that preserves the scalar product, that is 
\begin{equation}
\label{sca_pro_preserved}\mathbb{E}\left[W(h)W(k)\right]=\langle h,k\rangle_{\mathcal{H}}, \qquad \forall \ h, k \in \mathcal{H}.
\end{equation}
The existence of $W$ is ensured by Proposition \ref{Gaussian_isometries}. From now on we will characterize the elements in $\mathcal{H}_1$ as $W(h)$ for $h \in \mathcal{H}$. 

\begin{definition}
\label{def_Mal1}
We define the derivative of an element $W(h) \in \mathcal{H}_1$, $h \in \mathcal{H}$ as 
\begin{equation}
\label{defDWh}
DW(h):=h.
\end{equation}
\end{definition}
The above definition can be equivalently reformulated by saying that $D:=W^{-1}$ on $\mathcal{H}_1$. Since $W$ is an unitary operator we also have that $D=W^*$ on $\mathcal{H}_1$, where $W^*$ denotes the adjoint operator of $W$.

\begin{example}
Consider the Gaussian Hilbert space of Wiener integrals 
$$
\mathcal{H}_1= \left\{ \int_0^\infty f(s)\, {\rm d}B(s): \ f \in L^2(0, \infty)\right\}
$$ 
examined in Example \ref{ex_cov}(1).
In this case identity \eqref{defDWh} pointwise reads as
\begin{equation*}
D_t\left(\int_0^\infty f(s)\, {\rm d}B_s\right)=f(t).
\end{equation*}
Here $D_tF$ stands for $(DF)(t)$. In particular, $D_t(B_s)=\pmb{1}_{[0,s]}(t)$.
Roughly speaking, the derivative operator can be interpreted here as the inverse operator of the Wiener integral.
\end{example}

Definition \ref{def_Mal1} can be naturally extended to random variables that are polynomials in elements of $\mathcal{H}_1$. We introduce the following shorthand notation: for any (smooth) $\varphi:\mathbb{R}^n\rightarrow \mathbb{R}$ we set $\partial_i \varphi(x):= \frac{\partial }{\partial x_i}\varphi(x)$ for $i=1,..,n$.

\begin{definition}
\label{defDpol}The derivative of a random variable $F \in \mathcal{P}(\mathcal{H}_1)$ of the form $F=p(W(h_1),...,W(h_m))$ is the $\mathcal{H}$-valued random variable
\begin{equation*}
DF=\sum_{i=1}^m\partial_i p(W(h_1),...,W(h_m)) h_i.
\end{equation*}
If $p$ is a polynomial of degree zero then $DF=0$.
\end{definition}

\begin{remark}
So far we have 
\[
D:\mathcal{P}(\mathcal{H}_1) \subset L^2(\Omega) \rightarrow L^2(\Omega; \mathcal{H}).
\]
It is instructive to compare this to the following well-known situation in Sobolev-analysis. Take $f \in L^2(\mathcal{O})$, for some domain $\mathcal{O} \subset \mathbb{R}^n$. The \emph{gradient} operator $\nabla$ maps an appropriate subset of $L^2(\mathcal{O})$ into $L^2(\mathcal{O};\mathbb{R}^n)$. The space $\mathbb{R}^n$ comes into play since it is (isomorphic to) the tangent space at any point of $\mathcal{O}$. In our setting, the space $\mathcal{H}$ plays the role of the tangent space.
\end{remark}

In order to extend the class of differentiable random variables to a larger class than $\mathcal{P}(\mathcal{H}_1)$ the following integration by parts formula will play a crucial role.

\begin{proposition}
\label{i.b.p._first}
Let $F \in \mathcal{P}(\mathcal{H}_1)$ and let $h\in \mathcal{H}$, then
\begin{equation}
\label{ibp_0}
\mathbb{E}\left[ \langle DF, h \rangle_{\mathcal{H}}\right]=\mathbb{E}\left [W(h)F\right].
\end{equation}
\end{proposition}
\begin{proof}
By linearity it is sufficient to prove the result for a random variable $F \in  \mathcal{P}(\mathcal{H}_1)$ of the form $F=\prod_{i=1}^n W(h_i)$, $n \ge 1$ and $h_i \in \mathcal{H}_1$ for any $i=1,...,n$. Let $h_{n+1} \in \mathcal{H}$, using \eqref{sca_pro_preserved}, Definition \ref{defDpol} and the Isserlis formula (in equalities ($\star$) and ($\star \star$) below), we get

\begin{align*}
\mathbb{E}\left[\langle DF, h_{n+1}\rangle_{\mathcal{H}} \right]
&=
\mathbb{E}\left[\langle D\left( \prod_{i=1}^nW(h_i)\right), h_{n+1}\rangle_{\mathcal{H}} \right]
=\mathbb{E}\left[\langle \sum_{j=1}^n \left( h_j\prod_{i \ne j}W(h_i)\right), h_{n+1} \rangle_{\mathcal{H}} \right]
\\
&= \sum_{j=1}^n \left( \langle h_j, h_{n+1}\rangle_{\mathcal{H}} \mathbb{E}\left[\prod_{i \ne j}W(h_i) \right]\right)
= \sum_{j=1}^n \left( \mathbb{E}\left[ W(h_j)W(h_{n+1})\right] \mathbb{E}\left[\prod_{i \ne j}W(h_i) \right]\right)
\\
&\overset{(\star)}{=}  \sum_{j=1}^n \left( \mathbb{E}\left[ W(h_j)W(h_{n+1})\right] \sum \prod_k\mathbb{E}\left[W(h_{j_k}) W(h_{i_k})\right]\right)
\\
&=\sum \prod_k\mathbb{E}\left[W(h_{j_k}) W(h_{i_k})\right]
\overset{(\star \star)}{=} \mathbb{E}\left[ \prod_{i=1}^{n+1}W(h_i)\right]
= \mathbb{E}\left[FW(h_{n+1})\right],
\end{align*}
where the second sum in the second to last line of the above expression is taken over all the partitions of $\{1,...,j-1, j+1,...,n\}$ into disjoint pairs $\{j_k, i_k\}$ and the sum in the last line  is taken over all the partitions of $\{1,...,n+1\}$ into disjoint pairs $\{j_k, i_k\}$.
\end{proof}

We will need the following corollary of the i.b.p. formula.
\begin{corollary}
\label{coribp}
Let $F, G \in \mathcal{P}(\mathcal{H}_1)$ and let $h \in \mathcal{H}$, then 
\begin{equation}
\label{ibp_cor}
\mathbb{E}\left[\langle DF, h\rangle_{\mathcal{H}}G \right]= \mathbb{E}\left[ FGW(h)\right] -\mathbb{E}\left[F\langle DG, h\rangle_{\mathcal{H}}\right].
\end{equation}
\end{corollary}
\begin{proof}
The product $FG$ belongs to $\mathcal{P}(\mathcal{H}_1)$. The result is thus an immediate consequence of Proposition \ref{i.b.p._first} and Definition \ref{defDpol}.
\end{proof}

As a consequence of the above result we can show the closability (see Section \ref{closable_op_sec}) of the operator $D$.

\begin{proposition}
\label{closableProp}
For any $p \ge 1$ the operator $D: \mathcal{P}(\mathcal{H}_1) \rightarrow L^p(\Omega; \mathcal{H})$ is closable as an operator from $L^p(\Omega)$ to $L^p(\Omega; \mathcal{H})$ .
\end{proposition}
\begin{proof}
We only consider the case $p>1$; the proof for $p=1$ 
 requires a specific argument due to the duality $L^1(\Omega)/L^\infty(\Omega)$.
\\
Let $1<p<\infty$. Let $\{F_n\}_n$ be a sequence in $\mathcal{P}(\mathcal{H}_1)$ such that \begin{itemize}
\item [i)] $F_n \rightarrow 0$ in $L^p(\Omega)$,
\item [ii)] $DF_n \rightarrow \eta$ in $L^p(\Omega, \mathcal{H})$,
\end{itemize}
as $n \rightarrow \infty$. According to Proposition \ref{def_closable} we have to prove that $\eta=0$ $\mathbb{P}$-a.s.
\\
For any $G \in \mathcal{P}(\mathcal{H}_1)$ and $h \in \mathcal{H}$ we have
\begin{equation*}
\lim_{n \rightarrow \infty} \mathbb{E}\left[ G\langle DF_n, h\rangle_{\mathcal{H}}\right]=\mathbb{E}\left[ G\langle \eta, h\rangle_{\mathcal{H}}\right],
\end{equation*}
since the H\"older inequality yields 
\begin{equation*}
\mathbb{E}\left[\left|\langle DF_n-\eta,h\rangle_{\mathcal{H}}G \right| \right] \le \|h\|_{\mathcal{H}}\|DF_n-\eta\|_{L^p(\Omega;\mathcal{H})}\|G\|_{L^{\frac{p}{p-1}}(\Omega)}
\end{equation*}
and the r.h.s. in the above expression converges to zero as $n \rightarrow \infty$ in view of  (ii) and the fact that $\|G\|_{L^{\frac{p}{p-1}}(\Omega)}$ is bounded since $\mathcal{P}(\mathcal{H}_1) \subset L^q(\Omega)$ for any $0\le q <\infty$ thanks to Proposition \ref{PdenseLp}.
On the other hand, Corollary \ref{coribp} yields 
\begin{equation*}
\mathbb{E}\left[ G\langle DF_n, h\rangle_{\mathcal{H}}\right] = \mathbb{E}\left[ W(h)F_nG\right]-\mathbb{E}\left[F_n \langle DG,h\rangle_{\mathcal{H}}\right].
\end{equation*}
Reasoning similarly as above we have
\begin{equation*}
\lim_{n \rightarrow \infty}\mathbb{E}\left[ W(h)F_nG\right]=0,
\end{equation*}
since by the H\"older inequality 
\begin{equation*}
\mathbb{E}\left[ \left|F_nW(h)G\right|\right]\le \|F_n\|_{L^p(\Omega)}\|W(h)G\|_{L^{\frac{p}{p-1}}(\Omega)} 
\end{equation*}
and the r.h.s. converges to zero as $n \rightarrow \infty$ in view of (i) and the fact that $GW(h) \in \mathcal{P}(\mathcal{H}_1) \subset L^q(\Omega)$ for any $0\le q <\infty$. Moreover,
\begin{equation*}
\lim_{n \rightarrow \infty}\mathbb{E}\left[ F_n\langle DG, h\rangle_{\mathcal{H}}\right]=0,
\end{equation*}
since 
\begin{equation*}
\mathbb{E}\left[ \left|F_n\langle DG, h\rangle_{\mathcal{H}}\right|\right] 
\le  \|F_n\|_{L^p(\Omega)}\|DG\|_{L^{\frac{p}{p-1}}(\Omega; \mathcal{H})}\|h\|_{\mathcal{H}}
\end{equation*}
and the r.h.s. converges to zero as $n \rightarrow \infty$ in view of (i) and the fact that $DG\in L^q(\Omega;\mathcal{H})$ for any $0\le q <\infty$.
Collecting the above estimates we thus obtain 
\begin{align*}
\mathbb{E}\left[ G\langle \eta, h\rangle_{\mathcal{H}}\right]
=\lim_{n \rightarrow \infty} \mathbb{E}\left[ G\langle DF_n, h\rangle_{\mathcal{H}}\right]
&=\lim_{n \rightarrow \infty}\left( \mathbb{E}\left[ W(h)F_nG\right]-\mathbb{E}\left[F_n \langle DG,h\rangle_{\mathcal{H}}\right]\right)=0.
\end{align*}
Thus we get $\langle \mathbb{E}\left[G\eta\right], h\rangle_{\mathcal{H}}=0$. By the arbitrariness of $h \in \mathcal{H}$ it follows that $\mathbb{E}\left[G\eta\right]=0$. The arbitrariness of $G \in \mathcal{P}(\mathcal{H}_1)$ and the fact that $\mathcal{P}(\mathcal{H}_1)$ is a dense subset of $L^p(\Omega, \sigma(\mathcal{H}_1), \mathbb{P})$, in view of Proposition \ref{PdenseLp}, yields $\eta=0$ $\mathbb{P}$-a.s. which concludes the proof.
\end{proof}
For any $p \ge 1$ let $\mathbb{D}^{1,p}$ denote the closure of $\mathcal{P}(\mathcal{H}_1)$ with respect to the norm
\begin{equation}
\label{norm_D^1p}
\|F\|_{\mathbb{D}^{1,p}} =\left(\mathbb{E}\left[|F|^p\right]+\mathbb{E}\left[\|DF\|^p_{\mathcal{H}}\right]\right)^{\frac1p}.
\end{equation} 
In view of Proposition \ref{closableProp}, recalling Definition \ref{def_closure}, the operator $D$ can consistently be extended to the whole space $\mathbb{D}^{1,p}$.
This closed extension (still denoted by $D$) with domain $\mathbb{D}^{1,p}$ is called  
\textit{Malliavin derivative} and the space $\mathbb{D}^{1,p}$ is called \textit{domain of $D$} in $L^p(\Omega, \sigma(\mathcal{H}_1), \mathbb{P})$.
For any $p \ge 1$ the space $\mathbb{D}^{1,p}$ endowed with the norm \eqref{norm_D^1p} is a Banach space, for $p=2$ the space $\mathbb{D}^{1,2}$ is a Hilbert space with the scalar product 
\begin{equation*}
\langle F, G\rangle_{\mathbb{D}^{1,2}}= \mathbb{E}\left[FG\right] + \mathbb{E}\left[ \langle DF, DG\rangle_{\mathcal{H}}\right].
\end{equation*}


It is not difficult to prove, 
by a density argument, that the integration by parts formula \eqref{ibp_0} extends to elements in $\mathbb{D}^{1,2}$.
\begin{proposition}
\label{i.b.p.}
Let $F \in \mathbb{D}^{1,2}$ and let $h\in \mathcal{H}$, then
\begin{equation}
\label{ibp_1}
\mathbb{E}\left[ \langle DF, h \rangle_{\mathcal{H}}\right]=\mathbb{E}\left [W(h)F\right].
\end{equation}
\end{proposition}

The following result is a chain rule for the Malliavin derivative; it is useful in  applications.
\begin{proposition}
\label{chain_rule}
Let $\varphi: \mathbb{R}^m \rightarrow \mathbb{R}$ be a continuously differentiable function with bounded partial derivatives
and fix $p \ge 1$. Suppose that $F:(F^1,...,F^m)$ is a random vector whose components belong to the space $\mathbb{D}^{1,p}$. Then $\varphi(F) \in \mathbb{D}^{1,p}$ and 
\begin{equation*}
D(\varphi(F))=\sum_{i=1}^m \partial_i \varphi(F)DF^i.
\end{equation*}
\end{proposition}

One can also introduce the Malliavin derivative $D^k$ of order $k>1$. Proceeding similarly as what did above for the case $k=1$, one can define the iteration of the operator $D$ in such a way that for a random variable $F \in \mathcal{P}(\mathcal{H}_1)$, the iterated derivative $D^kF$ is a random variable with values in $\mathcal{H}^{\otimes k}$. Then one shows that the operator $D^k$ is closable from $\mathcal{P}(\mathcal{H}_1)\subset L^{p}(\Omega)$ into $L^p(\Omega; \mathcal{H}^{\otimes k})$ for all $p \ge 1$. The domain of $D^k$ in $L^p(\Omega)$ is denoted as $\mathbb{D}^{k,p}$: this is the completion of $\mathcal{P}(\mathcal{H}_1)$ with respect to the (full Sobolev) norm 
\begin{equation}
\label{Dkp_norm}
\|F\|_{\mathbb{D}^{k,p}} =\left(\mathbb{E}\left[|F|^p\right]+\sum_{j=1}^k\mathbb{E}\left[\|D^jF\|^p_{\mathcal{H}^{\otimes j}}\right]\right)^{\frac1p}.
\end{equation} 
The Malliavin derivative of order $k$ is the operator $D^k:\mathbb{D}^{k,p}\subset L^{p}(\Omega)\rightarrow L^p(\Omega; \mathcal{H}^{\otimes k})$.

We conclude this Section with the following result which characterizes the domain $\mathbb{D}^{1,2}$ of the Malliavin derivative operator in terms of the Wiener chaos expansion, for a proof see e.g. \cite[Proposition 1.2.2]{Nualart}.
\begin{proposition}
\label{charD12}
Let $F \in L^2(\Omega, \sigma(\mathcal{H}_1), \mathbb{P})$ with Wiener chaos expansion $F=\sum_{n=0}^\infty J_nF$. Then $F \in \mathbb{D}^{1,2}$ if and only if 
\begin{equation*}
\sum_{n=1}^\infty n\|J_nF\|_{L^2(\Omega)}^2 <\infty.
\end{equation*}
In this case,
\begin{equation*}
\mathbb{E}\left[ \|DF\|^2_{\mathcal{H}}\right] = \sum_{n=1}^\infty n\|J_nF\|_{L^2(\Omega)}^2.
\end{equation*}
\end{proposition}


By iteration we have the following result.
\begin{proposition}  
\label{charDk2}
Let $F \in L^2(\Omega, \sigma(\mathcal{H}_1), \mathbb{P})$ with Wiener chaos expansion $F=\sum_{n=0}^\infty J_nF$. Then, for $k \in \mathbb{N}$, $F \in \mathbb{D}^{k,2}$ if and only if 
\begin{equation*}
\sum_{n=1}^\infty n^k\|J_nF\|_{L^2(\Omega)}^2 <\infty.
\end{equation*}
In this case 
\begin{equation*}
\mathbb{E}\left[ \|D^kF\|^2_{\mathcal{H}^{\otimes k}}\right] = \sum_{n=k}^\infty \frac{n!}{(n-k)!}\|J_nF\|_{L^2(\Omega)}^2.
\end{equation*}
\end{proposition}

\vspace{1.2cm}

We emphasize that the definition of the Malliavin derivative $D$ depends on the choice of the Hilbert space $\mathcal{H}$ and the unitary operator $W$. In fact, once we have fixed the reference probability space $(\Omega, \mathcal{F}, \mathbb{P})$ and the Gaussian Hilbert spaces $\mathcal{H}_1$, there are infinitely many choices of the Hilbert space $\mathcal{H}$ and the unitary operator $W$ by which to characterize $\mathcal{H}_1$ as in Proposition \ref{Gaussian_isometries}.
Different choices of Hilbert spaces and unitary operators correspond to different (infinitely many!) Malliavin derivatives. This is clear from the way we have constructed the Malliavin derivative operator, starting from Definition \ref{def_Mal1}. Nevertheless, Proposition \ref{charD12} ensures that the domain of all these Malliavin derivatives is the same, since the characterization of $\mathbb{D}^{1,2}$ is given in terms of the Wiener chaos decomposition that relies only on the Gaussian Hilbert space $\mathcal{H}_1$ (and not on the choice of $\mathcal{H}$ and $W$). We can summarize the above discussion in the following diagram.

\bigskip
\begin{equation}
\label{abstract_diagram}
\xymatrix{
\mathcal{H} \ar^W[r] \ar_U[ddd] & \mathcal{H}_1
\ar@/ ^1pc/^{DW(h)=h}[drr] \ar@/ ^8pc/^{\widetilde{D}\widetilde{W}(\widetilde{h})=\widetilde{h}}[dddrr]  
\\
&\mathbb{D}^{1,2}\ar@{} [u] |{\cap} \ar^D[rr] \ar^{\widetilde D}[ddrr] && \mathcal{H} \ar^{U}[dd]
\\
 &L^2(\Omega, \sigma(\mathcal{H}_1),\mathbb{P})\ar@{} [u] |{\cap}
  \\
\widetilde{\mathcal{H}} \ar^{\widetilde{W}}[uuur] &&& \widetilde{\mathcal{H}} 
} 
\end{equation}

\bigskip
Here, $\mathcal{H}_1$ represents the reference Gaussian Hilbert space, $\mathcal{H}$ and $\widetilde{\mathcal{H}}$ are separable Hilbert spaces in correspondence with $\mathcal{H}_1$ through the unitary operators $W$ and $\widetilde{W}$, respectively; $U$ is a unitary operator between $\mathcal{H}$ and $\widetilde{\mathcal{H}}$. Notice that the only requirement on the Hilbert spaces $\mathcal{H}$ and $\tilde{\mathcal{H}}$ is that they need to have the same dimension as $\mathcal{H}_1$.
In this framework, we construct two Malliavin derivatives $D$ and $\widetilde{D}$ that are different but that have the same domain:
\begin{equation}
\label{abstract_rel}
\widetilde{D}=UD, \qquad \text{Dom}(D)=\text{Dom}(\widetilde{D})=\mathbb{D}^{1,2} \subset L^2(\Omega, \sigma(\mathcal{H}_1),\mathbb{P}).
\end{equation}
In particular, it holds 
\begin{equation}
\label{abstract_relation_bis}
DW(h)=h, \ \forall h \in \mathcal{H} \quad \text{and} \quad \widetilde{D}\widetilde{W}(\widetilde{h})=\widetilde{h}, \ \forall \widetilde h \in \mathcal{\widetilde H}.
\end{equation}

\subsection{The Malliavin derivative of Hilbert-valued random variables}
\label{sec_MDH}

The definition of Mallivin derivative can be extended to Hilbert-valued random variables. Let $\mathcal{V}$ be a separable Hilbert space and consider the family of $\mathcal{V}$-valued polynomial random variables in elements in $\mathcal{H}_1$ 
\begin{equation}
\label{P_V_H_1}
\mathcal{P}_{\mathcal{V}}(\mathcal{H}_1):=\left\{ F=\sum_{j=1}^n F_jv_j,\ \  v_j \in \mathcal{V}, \ \ F_j \in \mathcal{P}(\mathcal{H}_1)\right\}.
\end{equation}
For $F \in \mathcal{P}_{\mathcal{V}}(\mathcal{H}_1)$ we define
\begin{equation}
\label{def_DF_V}
DF:=\sum_{j=1}^n DF_j \otimes v_j.
\end{equation}
$DF$ is a square integrable random variable with values in the Hilbert space $\mathcal{H} \otimes \mathcal{V}$ (see Appendix \ref{app_ten_pro} for the definition of tensor product spaces) that we identify with the space of Hilbert-Schmidt operators from $\mathcal{H}$ to $\mathcal{V}$ (see Example \ref{ex_TP_HS}). If the real-valued random variables $F_j \in \mathcal{P}(\mathcal{H}_1)$ have the form $F_j=p_j(W(h_1),...,W(h_m))$, then Proposition \ref{defDpol} and Example \ref{ex_TP_HS} yields the explicit form for \eqref{def_DF_V}
\begin{equation}
\label{DF_HS}
DF
=\sum_{j=1}^n\sum_{i=1}^m \partial_i p_j(W(h_1),...,W(h_m)) h_i \otimes v_j
=\sum_{j=1}^n\sum_{i=1}^m \partial_i p_j(W(h_1),...,W(h_m))T_{h_i, v_j},
\end{equation}
where $T_{h_i,v_j} \in \mathcal{L}_2(\mathcal{H};\mathcal{V})$ is defined according to \eqref{def_TP_HS}, i.e. 
\begin{equation}
\label{def_TP_HS_bis}
T_{h_i,v_j}(k)=\langle h_i, k\rangle_{\mathcal{H}} v_j,
\end{equation}
for any $k \in \mathcal{H}$. 

It can be shown that $D$ is a closable operator from $\mathcal{P}_{\mathcal{V}}(\mathcal{H}_1) \subset L^2(\Omega;\mathcal{V})$ into $L^p(\Omega; \mathcal{H}\otimes \mathcal{V})$, for any $p \ge 1$. Its closure, still denoted by $D$, is the \textit{Malliavin derivative} of $\mathcal{V}$-valued random variables, $D : \mathbb{D}^{1,p}(\mathcal{V}) \subset L^p(\Omega;\mathcal{V}) \rightarrow L^p(\Omega; \mathcal{H}\otimes \mathcal{V})$, where the domain $\mathbb{D}^{1,p}(\mathcal{V})$ is the completion of $\mathcal{P}_{\mathcal{V}}(\mathcal{H}_1)$ with respect to the norm 
\begin{equation*}
\|F\|_{\mathbb{D}^{1,p}(\mathcal{V})} =\left(\mathbb{E}\left[\|F\|^p_{\mathcal{V}}\right]+\mathbb{E}\left[\|DF\|^p_{\mathcal{H} \otimes \mathcal{V}}\right]\right)^{\frac1p}.\end{equation*}

\begin{remark}
Similarly, one can introduce the Malliavin derivative $D^k:\mathbb{D}^{k,p}(\mathcal{V}) \subset L^p(\Omega;\mathcal{V}) \rightarrow L^p(\Omega; \mathcal{H}^{\otimes k}\otimes \mathcal{V})$ of order $k>1$. The domain $\mathbb{D}^{k,p}(\mathcal{V})$ is the completition of $\mathcal{P}_{\mathcal{V}}(\mathcal{H}_1)$ with respect to the seminorm
\begin{equation*}
\|F\|_{\mathbb{D}^{k,p}} =\left(\mathbb{E}\left[\|F\|_{\mathcal{V}}^p\right]+\sum_{j=1}^k\mathbb{E}\left[\|D^jF\|^p_{\mathcal{H}^{\otimes j}\otimes \mathcal{V}}\right]\right)^{\frac1p}.
\end{equation*} 
\end{remark}

\subsection{The directional Malliavin derivative}
\label{dir_der_sec}
It will be useful in the following to have an ad hoc notation to indicate the Malliavin derivative of a random variable in the direction $h \in \mathcal{H}$.
\\
Let us fix an element $h \in \mathcal{H}$. We define the operator $D^h$ on the set of real-valued polynomial random variables $\mathcal{P}(\mathcal{H}_1)$ as
\begin{equation*}
D^hF:= \langle DF, h\rangle_{\mathcal{H}}, \qquad F \in \mathcal{P}(\mathcal{H}_1).
\end{equation*} 
Reasoning as done for the operator $D$ one can prove that $D^h$ is a closable operator from $L^p(\Omega)$ into $L^p(\Omega)$, for any $p \ge 1$, and it has a domain that we will denote by $\mathbb{D}_h^{1,p}$.  
\\
The next lemma will be used in many instances in the following.
\begin{lemma}
\label{lem_tec_0}
Let $F \in \mathcal{P}(\mathcal{H}_1)$, for any $h, k \in \mathcal{H}$ it holds 
\begin{equation*}
D^hD^kF=D^kD^hF.
\end{equation*}
\end{lemma}
\begin{proof}
Let $h, k \in \mathcal{H}$ and $F \in \mathcal{P}(\mathcal{H}_1)$ of the form 
$F=p(W(h_1),...,W(h_m))$, with $p$ a polynomial function. For any $j=1,...,m$, let us set $G_j:= \partial_j p(W(h_1),...,W(h_m))$ and for any $\ell \in 1,...,m$ let us set $G_{j, \ell}:=  \frac{\partial G_j}{\partial x_\ell}$.
 We have
 \begin{equation*}
 DF=\sum_{j=1}^m G_j \,h_j \qquad \text{and} \quad 
DG_j=\sum_{\ell=1}^m G_{j,\ell}\,h_\ell, \ \forall j=1,...,m
\end{equation*}
and thus
\begin{align*}
D^k(D^hF)
&=
\sum_{j=1}^m  \langle D(G_j ),k\rangle_{\mathcal{H}}\,\langle h_j,h\rangle_{\mathcal{H}}=\sum_{j=1}^m \sum_{\ell=1}^mG_{j,\ell} \langle h_\ell,k\rangle_{\mathcal{H}}\, \langle h_j, h\rangle_{\mathcal{H}}
\\
&=\sum_{j=1}^m \sum_{\ell=1}^mG_{j,\ell} \langle h_\ell,h\rangle_{\mathcal{H}}\, \langle h_j, k\rangle_{\mathcal{H}}=
\sum_{j=1}^m\langle D(G_j),h\rangle_{\mathcal{H}}\, \langle h_j, k\rangle_{\mathcal{H}}=D^h(D^kF).
\end{align*}
\end{proof}

For $\mathcal{V}$-valued random variables, the notation $D^h$ should be understood as explained below. The operator $D^h$ acts on the set of polynomial $\mathcal{V}$-valued random variables $\mathcal{P}_{\mathcal{V}}(\mathcal{H}_1)$ as follows: given an element $F \in \mathcal{P}_{\mathcal{V}}(\mathcal{H}_1)$ of the form $F=\sum_{j=1}^n F_jv_j$ where $v_j \in \mathcal{V}$ and $F_j \in \mathcal{P}(\mathcal{H}_1)$, 
\begin{equation*}
D^hF:= \sum_{j=1}^n v_j D^hF_j=\sum_{j=1}^nv_j\langle DF_j, h\rangle_{\mathcal{H}}.
\end{equation*} 
In this case one can prove that $D^h$ is a closable operator from $L^p(\Omega; \mathcal{V})$ into $L^p(\Omega;\mathcal{V})$, for any $p \ge 1$, and it has a domain that we will denote by $\mathbb{D}_h^{1,p}(\mathcal{V})$.
\\
The next result will be useful in the following.
\begin{lemma}
\label{tech_lem_der}
Let $F \in \mathcal{P}_{\mathcal{V}}(\mathcal{H}_1)$, then for any $h \in \mathcal{H}$, $k \in \mathcal{V}$ it holds
\begin{equation*}
\langle D^hF, k\rangle_{\mathcal{V}}=D^h\langle F, k\rangle_{\mathcal{V}}.
\end{equation*} 
\end{lemma}
\begin{proof}
Let $F \in \mathcal{P}_{\mathcal{V}}(\mathcal{H}_1)$ of the form $F=\sum_{j=1}^n F_jv_j$ where $v_j \in \mathcal{V}$ and $F_j \in \mathcal{P}(\mathcal{H}_1)$, we have 
\begin{align*}
\langle D^hF, k\rangle_{\mathcal{V}}
&:= \langle \sum_{j=1}^n v_j\langle DF_j, h \rangle_{\mathcal{H}}, k\rangle_{\mathcal{V}}
= \sum_{j=1}^n\langle DF_j, h \rangle_{\mathcal{H}} \langle v_j,k\rangle_{\mathcal{V}}
\\
&=\sum_{j=1}^n\langle D\left( F_j \langle v_j,k\rangle_{\mathcal{V}}\right), h \rangle_{\mathcal{H}}
=\langle D\left( \langle \sum_{j=1}^n F_jv_j,k\rangle_{\mathcal{V}}\right), h \rangle_{\mathcal{H}}
\\
&=\langle D\left( \langle F,k\rangle_{\mathcal{V}}\right), h \rangle_{\mathcal{H}}
=D^h\langle F, k\rangle_{\mathcal{V}}.
\end{align*}
\end{proof}

\section{The divergence operator}
\label{div_sec}

In this Section we introduce the divergence operator, defined as the adjoint operator (see Section \ref{HS_sec}) of the Malliavin derivative, and we derive some of its properties. 

The framework is the same as in Section \ref{Malliavin_sec} that is: we are given a complete probability space  $(\Omega, \mathcal{F}, \mathbb{P})$,  a Gaussian Hilbert space $\mathcal{H}_1$ 
and we assume $\mathcal{F}$ to be the $\sigma$-field generated by $\mathcal{H}_1$. Moreover, we fix a separable Hilbert space $\mathcal{H}$ and a unitary operator $W:\mathcal{H} \rightarrow \mathcal{H}_1$ so that elements in $\mathcal{H}_1$ are characterized as in Proposition \ref{Gaussian_isometries}.

We recall that, by construction, the Malliavin derivative operator $D$ is a closed an unbounded operator with values in $L^2(\Omega, \mathcal{H})$. Its domain $\mathbb{D}^{1,2}$ is a dense subset of $L^2(\Omega)$, being $\mathcal{P}(\mathcal{H}_1) \subset \mathbb{D}^{1,2}$ and $\mathcal{P}(\mathcal{H}_1)$ dense in $L^2(\Omega)$ in virtue of Corollary \ref{cor_den_pol}.

\begin{definition}
We call \emph{divergence operator}, denoted by $\delta$, the adjoint of the Malliavin derivative operator $D$. That is, $\delta$
 is an unbounded operator on $L^2(\Omega;\mathcal{H})$ with values in $L^2(\Omega)$ such that:
\begin{itemize}
\item the domain of $\delta$, denoted by $\emph{Dom}(\delta)$, is the set of $\mathcal{H}$-valued square-integrable random variables $u \in L^2(\Omega;\mathcal{H})$ such that 
\begin{equation}
\label{cond_u_domd}
\left|\mathbb{E}\left[\langle DF, u\rangle_{\mathcal{H}} \right] \right| \le C\|F\|_{L^2(\Omega)}, \quad \forall F \in \mathbb{D}^{1,2},
\end{equation}
where $C$ is some positive constant depending on $u$.
\item If $u \in \emph{Dom}(\delta)$, then $\delta(u)$ is characterized as the element in $L^2(\Omega)$ such that 
\begin{equation}
\label{duality_rel}
\mathbb{E}\left[ F\delta(u)\right]=\mathbb{E}\left[\langle DF, u \rangle_{\mathcal{H}} \right], \quad \forall F \in \mathbb{D}^{1,2}.
\end{equation}
\end{itemize}
\end{definition}

\begin{remark}
Notice that the duality relation \eqref{duality_rel} is an extension of the integration by parts formula \eqref{ibp_1} to random elements $u\in \emph{Dom}(\delta) \subset L^2(\Omega;\mathcal{H})$.
Equivalently said, $\mathcal{H}\subset \emph{Dom}(\delta)$ and $\delta \equiv W$ on $\mathcal{H}$.
\end{remark}

\begin{remark}
Let us consider $(\mathbb{R}^n, \mathcal{B}(\mathbb{R}^n))$ and on it the Lebesgue measure $\lambda_n$. Given the vector field $u:\mathbb{R}^n \rightarrow \mathbb{R}^n$ and the scalar field $F:\mathbb{R}^n \rightarrow \mathbb{R}$, it holds 
\begin{equation*}
\int_{\mathbb{R}^n}\langle \nabla F, u\rangle_{\mathbb{R}^n}\, {\rm d}\lambda_n=\int_{\mathbb{R}^n} F (-\nabla \cdot u)\, {\rm d}\lambda_n.
\end{equation*}
This explains (up to a minus-sign) why the operator $\delta$ is called \emph{divergence}.
\end{remark}

\subsection{Properties of the divergence operator}
We derive here some properties of the divergence operator.

\begin{proposition}
\begin{itemize}
\item [(i)] If $u \in \emph{Dom}(\delta)$, then $\mathbb{E}\left[\delta(u)\right]=0$.
\item [(ii)] $\delta$ is a linear and closed operator.
\end{itemize}
\end{proposition}
\begin{proof}
\begin{itemize}
Item (i) is an immediate consequence of \eqref{duality_rel} by taking $F=1$; item (ii) steams from the fact that the adjoint of a linear densely defined operator is closed (see Section \ref{HS_sec}).
\end{itemize}
\end{proof}
Let us recall that by $\mathcal{P}_{\mathcal{H}}(\mathcal{H}_1)$ we denote the family of $\mathcal{H}$-valued polynomial random variables in elements in $\mathcal{H}_1$ (see \eqref{P_V_H_1}). The next result shows that elements $u$ in $\mathcal{P}_{\mathcal{H}}(\mathcal{H}_1)$ belong to the domain of $\delta$ and it provides an explicit expression for $\delta(u)$.
\begin{lemma}
\label{preli_lem_0}
Let $u \in \mathcal{P}_{\mathcal{H}}(\mathcal{H}_1)$ be of the form 
\begin{equation}
\label{u_lem}
u= \sum_{j=1}^n F_j h_j
\quad \text{with} \ \ F_j\in\mathcal{P}(\mathcal{H}_1), \ h_j\in \mathcal{H}, \ \forall j=1..n.
\end{equation}
Then $u \in \emph{Dom}(\delta)$ and 
\begin{equation}
\label{deltau_lem}
\delta(u)= \sum_{j=1}^nF_jW(h_j)-\sum_{j=1}^n \langle DF_j, h_j\rangle_{\mathcal{H}}.
\end{equation}
\end{lemma}
\begin{proof}
Let us start by verifying  \eqref{cond_u_domd}. Let $G \in \mathbb{D}^{1,2}$ and let $u \in \mathcal{P}_\mathcal{H}(\mathcal{H}_1)$ be of the form \eqref{u_lem}. By means of the integration by parts formula \eqref{ibp_cor} and the H\"older inequality we infer
\begin{align*}
\left|\mathbb{E} \left[ \langle DG, u\rangle_{\mathcal{H}}\right] \right|
&= \left|\sum_{j=1}^n \mathbb{E}\left[F_j \langle DG, h_j\rangle_{\mathcal{H}}\right] \right|
\le \sum_{j=1}^n \left|\mathbb{E}\left[G \langle DF_j, h_j \rangle_{\mathcal{H}} \right]\right|+ \sum_{j=1}^n \left|\mathbb{E}\left[G F_j W(h_j) \right]\right|
\le C \|G\|_{L^2(\Omega)},
\end{align*}
with $C$ a positive constant depending on $u$. This proves \eqref{cond_u_domd} and shows that $u \in \text{Dom}(\delta)$. 
\\
Exploiting \eqref{duality_rel} and \eqref{ibp_cor}, for any $G \in \mathbb{D}^{1,2}$ we get 
\begin{align*}
\mathbb{E}\left[ G \delta(u)\right] 
&=\mathbb{E}\left[ \langle DG ,u\rangle_{\mathcal{H}}\right] 
= \mathbb{E} \left[\sum_{j=1}^n F_j\langle DG,  h_j \rangle_{\mathcal{H}}\right]
\\
&= \mathbb{E} \left[\sum_{j=1}^n F_jG W(h_j) \right]- \mathbb{E} \left[\sum_{j=1}^n G\langle DF_j,  h_j \rangle_{\mathcal{H}}\right]
= \mathbb{E} \left[G\sum_{j=1}^n \left(F_j W(h_j)-\langle DF_j,  h_j \rangle_{\mathcal{H}}\right)\right]
\end{align*}
and \eqref{deltau_lem} immediately follows from the arbitrariness of $G\in \mathbb{D}^{1,2}$.
\end{proof}
 The following result provides a larger class of $\mathcal{H}$-valued random variables in the domain of the divergence operator.
 \begin{proposition}
 \label{D^12Hindomd}
 The space $\mathbb{D}^{1,2}(\mathcal{H})$ is included in $\emph{Dom}(\delta)$. If $u, v \in \mathbb{D}^{1,2}(\mathcal{H})$, then 
\begin{equation}
\label{dudv-for}
\mathbb{E}\left[ \delta (u)\delta(v)\right]= \mathbb{E}\left[ \langle u,v\rangle_{\mathcal{H}}\right] + \mathbb{E}\left[ \emph{Tr}(DuDv)\right].
\end{equation}
\end{proposition}
We recall that (see Section \ref{sec_MDH} for the details) if $u \in \mathbb{D}^{1,2}(\mathcal{H})$, then its derivative $Du$ is a random variable with values in the tensor product Hilbert space $\mathcal{H}\otimes \mathcal{H}$ that we identify with the space of Hilbert-Schmidt operators form $\mathcal{H}$ into itself. The term  $\mathbb{E}\left[\text{Tr}(DuDv)\right]$ that appears in \eqref{dudv-for} is thus well defined in virtue of Definition \ref{trace_def} and Proposition \ref{proposition-HS}(iii).
\\
The proof of Proposition \ref{D^12Hindomd} relies on the following two lemmata. Lemma \ref{lemma_duality} provides a commutativity relation between the derivative operator $D^h$ and the divergence operator. Lemma \ref{lem6.7} proves equality \eqref{dudv-for} for elements in $\mathcal{P}_{\mathcal{H}}(\mathcal{H}_1)$.

\begin{lemma}
\label{lemma_duality}
Let $u \in \mathcal{P}_{\mathcal{H}}(\mathcal{H}_1)$ and $h \in \mathcal{H}$, then 
\begin{equation}
\label{com_rel}
D^h\delta(u)=\langle u, h\rangle_{\mathcal{H}}+\delta(D^hu).
\end{equation}
\end{lemma}
\begin{proof}
Let $u \in \mathcal{P}_{\mathcal{H}}(\mathcal{H}_1)$ be of the form 
\begin{equation*}
u= \sum_{j=1}^n F_j h_j
\quad \text{with} \ \ F_j\in\mathcal{P}(\mathcal{H}_1), \ h_j\in \mathcal{H}, \ \forall j=1..n.
\end{equation*}
From Lemma \ref{preli_lem_0} we know that $u \in \text{Dom}(\delta)$ and \eqref{deltau_lem} holds. Keeping in mind the notation introduced in Section \ref{dir_der_sec} we compute
\begin{align}
\label{pezzo1}
D^h(\delta(u))
&=D^h \left(\sum_{j=1}^n F_jW(h_j)\right)-D^h \left(\sum_{j=1}^n \langle DF_j,h_j\rangle_{\mathcal{H}}\right)
\notag\\
&=\sum_{j=1}^n F_j\langle DW(h_j),h\rangle_{\mathcal{H}}+ \sum_{j=1}^nW(h_j) \langle DF_j, h\rangle_{\mathcal{H}} - \sum_{j=1}^n \langle D\left(\langle DF_j,h_j\rangle_{\mathcal{H}}\right),h\rangle_{\mathcal{H}}
\notag \\
&=\sum_{j=1}^n F_j\langle h_j,h\rangle_{\mathcal{H}}+ \sum_{j=1}^n D^hF_jW(h_j)-\sum_{j=1}^n \langle D(D^{h_j}F_j),h\rangle_{\mathcal{H}}
\notag \\
&=\langle u,h\rangle_{\mathcal{H}} + \sum_{j=1}^n D^hF_jW(h_j) 
-\sum_{j=1}^n D^h(D^{h_j}F_j).
\end{align}
Let us now compute
\begin{align*}
\delta(D^hu)
&=\delta \left(\sum_{j=1}^n \langle DF_j, h\rangle_{\mathcal{H}}h_j\right)
=\sum_{j=1}^n \langle DF_j, h\rangle_{\mathcal{H}}W(h_j)-\sum_{j=1}^n \langle D\left(\langle DF_j,h\rangle_{\mathcal{H}}\right),h_j \rangle_{\mathcal{H}}
\\
&=\sum_{j=1}^n D^hF_jW(h_j)-\sum_{j=1}^n \langle D(D^hF_j),h_j \rangle_{\mathcal{H}}
= \sum_{j=1}^n D^hF_jW(h_j)-\sum_{j=1}^n D^{h_j}(D^hF_j),
\end{align*}
where in the second equality we used the formula \eqref{deltau_lem} since $\sum_{j=1}^n \langle DF_j, h\rangle_{\mathcal{H}}h_j \in \mathcal{P}_{\mathcal{H}}(\mathcal{H}_1)$.
From the chain of equalities above we thus obtain 
\begin{equation}
\label{pezzo2}
\sum_{j=1}^n D^hF_jW(h_j)=\delta(D^hu)+ \sum_{j=1}^n D^{h_j}(D^hF_j).
\end{equation}
Substituting \eqref{pezzo2} in \eqref{pezzo1} we get
\begin{equation*}
D^h(\delta(u))
=\langle u,h\rangle_{\mathcal{H}} + \delta(D^hu)+ \sum_{j=1}^n D^{h_j}(D^hF_j)
-\sum_{j=1}^n D^h(D^{h_j}F)=\langle u,h\rangle_{\mathcal{H}} + \delta(D^hu),
\end{equation*}
where the last equality follows from Lemma \ref{lem_tec_0}.
This concludes the proof.
\end{proof}

\begin{remark}
\label{remark_duality}
The above result can be extended as follows (for a proof see e.g \cite[Proposition 1.3.2]{Nualart}). Suppose that $u \in \mathbb{D}^{1,2}(\mathcal{H})$, and $D^hu$ belongs to the domain of the divergence. Then $\delta(u) \in \mathbb{D}^{1,2}_{h}$,and the commutation relation \eqref{com_rel} holds.
\end{remark}

 \begin{lemma}
 \label{lem6.7}
If $u, v \in \mathcal{P}_{\mathcal{H}}(\mathcal{H}_1)$, then 
\begin{equation}
\label{dudv_for2}
\mathbb{E}\left[ \delta (u)\delta(v)\right]= \mathbb{E}\left[ \langle u,v\rangle_{\mathcal{H}}\right] + \mathbb{E}\left[ \emph{Tr}(DuDv)\right].
\end{equation}
\end{lemma}
\begin{proof}
Let $u, v \in \mathcal{P}_{\mathcal{H}}(\mathcal{H}_1)$ and let $\{e_i\}_{i=1}^\infty$ be a complete orthonormal system in $\mathcal{H}$.
Using the duality relationship \eqref{duality_rel} and Lemma \ref{lemma_duality} we get
\begin{align*}
\mathbb{E}\left[\delta(u)\delta(v) \right]
&=\mathbb{E} \left[ \langle D\delta(u), v \rangle_{\mathcal{H}} \right]
= \mathbb{E} \left[\sum_{i=1}^\infty  D^{e_i}\delta(u)\langle v, e_i\rangle_{\mathcal{H}} \right]
= \mathbb{E}\left[\sum_{i=1}^\infty\langle v, e_i\rangle_{\mathcal{H}}\left(\langle u, e_i\rangle_{\mathcal{H}}+ \delta(D^{e_i}u)\right) \right]
\\
&= \mathbb{E} \left[ \langle u,v\rangle_{\mathcal{H}}\right] +\sum_{i=1}^\infty \mathbb{E}\left[\langle v, e_i\rangle_{\mathcal{H}}\delta(D^{e_i}u) \right].
\end{align*}
Exploiting again \eqref{duality_rel} and Lemma \ref{tech_lem_der} we now rewrite the term $\mathbb{E}\left[\langle v, e_i\rangle_{\mathcal{H}}\delta(D^{e_i}u) \right]
$, that appears in the last line of the above equality, as follows
\begin{align*}
\mathbb{E}\left[\langle v, e_i\rangle_{\mathcal{H}}\delta(D^{e_i}u) \right]
&=\mathbb{E}\left[\langle D\left(\langle v, e_i\rangle_{\mathcal{H}}\right), D^{e_i}u\rangle \right]
=\sum_{j=1}^{\infty} \mathbb{E}\left[\langle D\left(\langle v, e_i\rangle_{\mathcal{H}}\right), e_j\rangle_{\mathcal{H}}\langle D^{e_i}u, e_j\rangle \right]
\\
&=\sum_{j=1}^{\infty} \mathbb{E}\left[ D^{e_j}\langle v, e_i\rangle_{\mathcal{H}}D^{e_i}\langle u, e_j\rangle \right].
\end{align*}
Thus we get 
\begin{equation*}
\mathbb{E}\left[\delta(u)\delta(v) \right]=  \mathbb{E} \left[ \langle u,v\rangle_{\mathcal{H}}\right] + 
\sum_{i=1}^\infty\sum_{j=1}^{\infty} \mathbb{E}\left[D^{e_j}\langle v, e_i\rangle_{\mathcal{H}}D^{e_i}\langle u, e_j\rangle_{\mathcal{H}} \right].
 \end{equation*}
Let us now show that 
$\mathbb{E}\left[ \text{Tr}(DuDv)\right] = \sum_{i=1}^\infty\sum_{j=1}^{\infty} \mathbb{E}\left[ D^{e_j}\langle v, e_i\rangle_{\mathcal{H}}D^{e_i}\langle u, e_j\rangle \right]$
which will conclude the proof. Without loss of generality and for the sake of simplicity we suppose $u$ and $v$ to be of the form $u=Fh$ and $v=Gk$ with $h, k  \in \mathcal{H}$, $F=\varphi(W(h_1),...,W(h_n))$, $G=\psi(W(k_1),...,W(k_m))$ with $\varphi, \psi \in \mathcal{P}(\mathcal{H}_1)$. 
Recalling \eqref{DF_HS} we have 
\begin{equation*}
Du=\sum_{j=1}^n\partial_j \varphi(W(h_1),...,W(h_n))T_{h_j,h}, \quad Dv=\sum_{\ell=1}^m\partial_\ell \psi(W(k_1),...,W(k_m))T_{k_\ell,k}.\end{equation*}
Therefore we compute (using a more compact notation)
\begin{align}
\label{trace_compu}
\text{Tr}(DuDv)
&=\sum_{j=1}^n\sum_{\ell=1}^m\partial_j \varphi \partial_\ell \psi\text{Tr}\left(T_{h_j,h}T_{k_\ell,k}\right)
= \sum_{j=1}^n\sum_{\ell=1}^m\partial_j \varphi \partial_\ell \psi\langle k, h_j \rangle_{\mathcal{H}}\langle h, k_\ell\rangle_{\mathcal{H}}
\notag\\
&=\sum_{j=1}^n\partial_j\varphi \langle k, h_j\rangle_{\mathcal{H}}\sum_{\ell=1}^m\partial_\ell\psi \langle h, k_\ell\rangle_{\mathcal{H}}
=D^kFD^hG
\notag\\
&=\sum_{i=1}^\infty \sum_{j=1}^\infty\langle DF, e_i\rangle_{\mathcal{H}}\langle k, e_i\rangle_{\mathcal{H}}\langle DG, e_j\rangle_{\mathcal{H}}\langle h, e_j\rangle_{\mathcal{H}}
\notag\\
&=\sum_{i=1}^\infty\sum_{j=1}^{\infty}  D^{e_j}\langle v, e_i\rangle_{\mathcal{H}}D^{e_i}\langle u, e_j\rangle,
\end{align}
where in the second equality in \eqref{trace_compu} we have used the definition of trace \ref{trace_def} and \eqref{def_TP_HS_bis} to compute 
\begin{align*}
\text{Tr}\left(T_{h_j,h}T_{k_\ell,k}\right)
&:=\sum_{i=1}^\infty\langle T_{h_j,h}T_{k_\ell,k}e_i,e_i\rangle_{\mathcal{H}}
= \sum_{i=1}^\infty \langle T_{h_j,h}\left(k_\ell\langle k,e_i\rangle_{\mathcal{H}} \right), e_i\rangle_{\mathcal{H}}
\\
& = \sum_{i=1}^\infty \langle h_j\langle h, k_\ell\rangle_{\mathcal{H}}\langle k, e_i\rangle_{\mathcal{H}},e_i\rangle_{\mathcal{H}}
= \sum_{i=1}^\infty \langle h, k_\ell\rangle_{\mathcal{H}}\langle k, e_i\rangle_{\mathcal{H}}\langle h_j,e_i\rangle_{\mathcal{H}}
= \langle k, h_j\rangle_{\mathcal{H}} \langle h, k_\ell\rangle_{\mathcal{H}},\end{align*}
and in the last equality in \eqref{trace_compu} we have used the fact that 
\begin{equation*}
\langle DG, e_j\rangle_{\mathcal{H}}\langle k, e_i\rangle_{\mathcal{H}}
= D^{e_j}G\langle k, e_i\rangle_{\mathcal{H}}=D^{e_j}\langle Gk, e_i\rangle_{\mathcal{H}}=D^{e_j}\langle v, e_i\rangle_{\mathcal{H}}
\end{equation*}
and analogsly
\begin{equation*}
\langle DF, e_i\rangle_{\mathcal{H}}\langle h, e_j \rangle_{\mathcal{H}}=D^{e_i}\langle u, e_j\rangle_{\mathcal{H}}.
\end{equation*}
This concludes the proof.
\end{proof}

We have now all the ingredients to prove Proposition \ref{D^12Hindomd}.
\begin{proof}[Proof of Proposition \ref{D^12Hindomd}.]
Let us start by proving that $\mathbb{D}^{1,2}(\mathcal{H}) \subset \text{Dom}(\delta)$. Let $u \in \mathbb{D}^{1,2}(\mathcal{H})$, then there exists a sequence $u_n \in \mathcal{P}_{\mathcal{H}}(\mathcal{H}_1)$ such that $u_n \rightarrow u$ in $L^2(\Omega, \mathcal{H})$ and $Du_n \rightarrow Du$  in $L^2(\Omega, \mathcal{H}\otimes \mathcal{H})=L^2(\Omega, \mathcal{L}_2(\mathcal{H}))$. On the other hand, Lemma \ref{lem6.7} and Proposition \ref{proposition-HS}(iii) yield 
\begin{equation}
\label{dudv_for_0}
\mathbb{E}\left[ \delta (u_n)^2\right]= \mathbb{E}\left[\|u_n\|^2_{\mathcal{H}}\right] + \mathbb{E}\left[ \text{Tr}(Du_nDu_n)\right]
\le \mathbb{E}\left[\|u_n\|^2_{\mathcal{H}}\right] + \mathbb{E}\left[ \|Du_n\|^2_{\mathcal{H} \otimes \mathcal{H}}\right].
\end{equation}
From \eqref{dudv_for_0} we infer that $\{\delta(u_n)\}_{n \in \mathbb{N}}$ is a Cauchy sequence in $L^2(\Omega)$, thus convergent in $L^2(\Omega)$. Therefore, there exists $w \in L^2(\Omega)$ such that $\delta(u_n) \rightarrow w$ in  $L^2(\Omega)$ and, since $\delta$ is a closed operator, $u \in \text{Dom}(\delta)$ and $w=\delta(u)$. Moreover, it is immediate to see that \eqref{dudv-for} holds true.
\end{proof}

The following result (compare with \cite[Theorem 10.2.7]{Lunardi})  reminds us that the divergence operator has to be understood as a \emph{differential} operator: we can explicitly write the divergence of an element belonging to $\mathbb{D}^{1,2}(\mathcal{H})$ in terms of the directional Malliavin derivatives.
\begin{proposition}
\label{inf_dim_div}
Let $u \in \mathbb{D}^{1,2}(\mathcal{H})$. Fixing an orthonormal basis $\{h_n\}_{n \in \mathbb{N}}$ of $\mathcal{H}$, it holds 
\begin{equation}
\label{delta_inf}
\delta (u)=\sum_{j=1}^\infty \langle u, h_j\rangle_{\mathcal{H}}W(h_j)-\sum_{j=1}^\infty \langle D\langle u, h_j\rangle_{\mathcal{H}},h_j  \rangle_{\mathcal{H}}.
\end{equation}
\end{proposition}
\begin{proof}
Let $u \in \mathbb{D}^{1,2}(\mathcal{H})$. We approximate it by the sequence $\{u_n\}_{n \in \mathbb{N}}$ in $\mathbb{D}^{1,2}(\mathcal{H})$ of the form 
\begin{equation*}
u_n=\sum_{j=1}^n \langle u, h_j\rangle_{\mathcal{H}}h_j, \qquad n \in \mathbb{N}.
\end{equation*}
Exploiting the density of $\mathcal{P}_{\mathcal{H}}(\mathcal{H}_1)$ in $\mathbb{D}^{1,2}(\mathcal{H})$ and keeping in mind  \eqref{deltau_lem}, we infer that, for any $n \in \mathbb{N}$, 
\begin{equation}
\label{intermed}
\delta(u_n)= \sum_{j=1}^n\langle u, h_j\rangle_{\mathcal{H}}W(h_j)-\sum_{j=1}^n \langle D\langle u, h_j\rangle_{\mathcal{H}}, h_j\rangle_{\mathcal{H}}.
\end{equation}
Thanks to Propositions \ref{D^12Hindomd} and \ref{proposition-HS}(iii) it holds
\begin{equation}
\label{dudv_for}
\mathbb{E}\left[ \delta (u_n)^2\right]= \mathbb{E}\left[\|u_n\|^2_{\mathcal{H}}\right] + \mathbb{E}\left[ \text{Tr}(Du_nDu_n)\right]
\le \mathbb{E}\left[\|u_n\|^2_{\mathcal{H}}\right] + \mathbb{E}\left[ \|Du_n\|^2_{\mathcal{H} \otimes \mathcal{H}}\right].
\end{equation}
Bearing in mind that $u_n \rightarrow u$ in $\mathbb{D}^{1,2}(\mathcal{H})$, as $n \rightarrow \infty$, from \eqref{dudv_for} we infer that $\{\delta(u_n)\}_{n \in \mathbb{N}}$ is a Cauchy sequence in $L^2(\Omega)$, thus convergent in $L^2(\Omega)$. If follows that there exists $w \in L^2(\Omega)$ such that $\delta(u_n) \rightarrow w$ in  $L^2(\Omega)$ and, since $\delta$ is a closed operator, $u \in \text{Dom}(\delta)$ and $w=\delta(u)$. 
Therefore, passing to the limit in equality \eqref{intermed}, as $n \rightarrow \infty$,  we get the thesis.
\end{proof}
\begin{remark}
Let $(\mathbb{R}^n,\mathcal{B}(\mathbb{R}^n))$ and let us consider the Lebesgue measure on it. Given the vector field $F:\mathbb{R}^n \rightarrow \mathbb{R}^n$, and denoted by $\{e_j\}_{j=1}^n$ the canonical 
 basis of $\mathbb{R}^n$, one has the following expression for the divergence of $F$:
\begin{equation}
\label{delta_fin}
\nabla \cdot F(x)= \sum_{j=1}^n \frac{\partial F_j(x)}{\partial x_j} :=
\sum_{j=1}^n \langle \nabla \langle F(x), e_j\rangle_{\mathbb{R}^n}, e_j\rangle_{\mathbb{R}^n}.
\end{equation}
Comparing  \eqref{delta_inf} with \eqref{delta_fin}, we can interpret the term $\sum_{j=1}^\infty \langle D\langle v, h_j\rangle_{\mathcal{H}},h_j  \rangle_{\mathcal{H}}$ as an infinite-dimensional generalization of $\sum_{i=j}^n \langle \nabla \langle F, e_j\rangle_{\mathbb{R}^n}, e_j\rangle_{\mathbb{R}^n}$. The additional term $\sum_{j=1}^\infty \langle v, h_j\rangle_{\mathcal{H}}W(h_j)$ that appears in \eqref{delta_inf} is due to the Gaussian framework: we have a similar term also in the finite dimensional case. In fact, let $(\mathbb{R}^n,\mathcal{B}(\mathbb{R}^n))$ and  consider on it  the standard Gaussian measure. In this case it is easy to verify that the divergence of $F$ is equal to
\begin{equation*}
\delta(F)=\sum_{j=1}^n F_j(x)x_j- \sum_{j=1}^n \frac{\partial F_j(x)}{\partial x_j}= \langle F, x\rangle_{\mathbb{R}^n}- \nabla \cdot F(x).
\end{equation*}
\end{remark}

\subsection{Skorokhod integral}
In \cite{GT} is proved that, in a particular setting,  the adjoint of the Malliavin derivative
is equivalent to the definition of the stochastic integral with respect to a generalized Gaussian process, as introduced in \cite{Sko}.

Let $\{B_t\}_{t \in [0, 1]}$ be a standard Brownian motion on the probability space $(\Omega,\mathcal{F}, \mathbb{P})$. As in Example \ref{BM_ex} (i), we now consider the Gaussian Hilbert space $\mathcal{H}_1$ given by the closure of $Span \{B_t\}_{t\in [0,1]}$ in $L^2(\Omega)$.  We assume $\mathcal{F}$ to be the $\sigma$-field generated by $\mathcal{H}_1$.  As separable Hilbert space $\mathcal{H}$ we take $L^2(0,1)$ and we consider the unitary operator $W:L^2(0,1) \rightarrow \mathcal{H}_1$ given in Example \ref{ex_cov}-1, that is the Wiener integral. Thus elements in $\mathcal{H}_1$ are characterized as 
\begin{equation*}
W(h)=\int_0^1 h(s)\, {\rm d}B_s, \quad h \in L^2(0,1).
\end{equation*}
In particular, $B_t=W(\pmb{1}_{[0,t]})$, for all $t \in [0,1]$. We denote by $\{\mathcal{F}_t\}_{t \in [0,1]}$ the $\sigma$-algebra generated by the Brownian motion $\{B_t\}_{t \in [0, 1]}$.
\begin{footnote}{Equivalently, $\{\mathcal{F}_t\}_{t \in [0,1]}$ is the $\sigma$-algebra generated by the elements $W(h)$ with supp$(h) \subset [0,t]$.}
\end{footnote}

In this case the elements of Dom$(\delta)\subset L^2(\Omega; L^2(0,1))\simeq L^2(\Omega \times (0,1))$ are square integrable processes, and the divergence operator $\delta(u)$ is called \emph{Skorokhod stochastic integral} of the process $u$
. Usually the following notation is used
\begin{equation*}
\delta(u)=\int_0^1 u(s)\,{\rm d}B_s, \quad u \in \text{Dom}(\delta).
\end{equation*}
In this framework, Proposition \ref{D^12Hindomd} reads as follows.
\begin{proposition}
\label{Sko_pro}
The space $\mathbb{D}^{1,2}(L^2(0,1))$ is included in \emph{Dom}$(\delta)$. If $u,v \in \mathbb{D}^{1,2}(L^2(0,1))$, then 
\begin{equation}
\label{extension_Ito_isometry}
\mathbb{E} \left[\left(\int_0^1 u(s)\, {\rm d}B_s\right) \left(\int_0^1 v(s)\, {\rm d}B_s \right)\right]
=\mathbb{E}\left[\int_0^1 u(s)v(s)\, {\rm d}s \right]
+ \mathbb{E} \left[\int_0^1\int_0^1 D_su(t)D_tv(s)\, {\rm d}s\, {\rm d}t \right].
\end{equation}
\end{proposition}
\begin{proof}
The proof follows the lines of the proof of Proposition \ref{D^12Hindomd}. Let us first suppose $u, v \in \mathcal{P}_{L^2(0,1)}(\mathcal{H}_1)$ to be of the form $u=Fh$ and $v=Gk$ with $h, k  \in L^2(0,1)$, $F=\varphi(W(h_1),...,W(h_n))$, $G=\psi(W(k_1),...,W(k_m))$ with $\varphi, \psi \in \mathcal{P}(\mathcal{H}_1)$. Denoted by $\{e_i\}_{i=1}^\infty$ a complete orthonormal system in $L^2(0,1)$,  reasoning as in the proof of Proposition \ref{D^12Hindomd}, we obtain 
\begin{align*}
\mathbb{E}\left[\delta(u)\delta(v) \right]
&=\mathbb{E} \left[ \langle u,v\rangle_{L^2(0,1)}\right] + 
\sum_{i=1}^\infty\sum_{j=1}^{\infty} \mathbb{E}\left[D^{e_j}\langle v, e_i\rangle_{L^2(0,1)}D^{e_i}\langle u, e_j\rangle_{L^2(0,1)} \right]
\\
&=\mathbb{E} \left[\int_0^1u(s)v(s)\, {\rm d}s\right]+ \mathbb{E}\left[D^kFD^hG\right].
\end{align*}
Using the pointwise notation $D_tG=(DG)(t)$, we explicitly write 
\begin{align*}
D^kFD^hG
&= \langle DF,k\rangle_{L^2(0,1)} \langle DG,h\rangle_{L^2(0,1)}
=\int_0^1 (D_sF)k(s)\, {\rm d}s \int_0^1 (D_tG)h(t)\, {\rm d}t
\\
&=\int_0^1 \int_0^1 (D_s Fh(t))(D_t Gk(s))\, {\rm d}s \, {\rm d}t
=\int_0^1 \int_0^1 D_s u(t) D_tv(s)\, {\rm d}s \, {\rm d}t.
\end{align*}
Using the fact that $\mathcal{P}_{L^2(0,1)}(\mathcal{H}_1)$ is dense in $\mathbb{D}^{1,2}(L^2(0,1))$, we then conclude as in the proof of Proposition \ref{D^12Hindomd}.
\end{proof}

\begin{remark}
Formula \eqref{extension_Ito_isometry} can be considered an extension of the It\^o isometry. If, in addition to the Assumptions of Proposition \ref{Sko_pro}, we also assume the processes $u$ and $v$ to be adapted to the filtration $\{\mathcal{F}_t\}_{t \in [0,1]}$, generated by the Brownian motion, we have that $D_s u(t)=0$ for a.e. $s, t \in [0,1]$ with $s>t$ (see e.g. \cite[Corollary 1.2.1]{Nualart}). Consequently, the second summand in the r.h.s of \eqref{extension_Ito_isometry} equals zero and we recover the usual isometry property of the It\^o integral.
 \end{remark}

The commutativity relation between $D^h$ and $\delta$ (see Lemma \ref{lemma_duality} and Remark \ref{remark_duality}) reads as follows.
\begin{proposition}
\label{pro_der_int}
Let $u \in \mathbb{D}^{1,2}(L^2(0,1))$. Assume that for a.e. $t \in [0,1]$ the process $\{D_t u(s)\}_{s \in [0,1]}$ is Skorokhod-integrable, and there exists a version of the process $\left\{\int_0^1 D_tu(s)\, {\rm d}B_s\right\}_{t \in [0,1]}$ which is in $L^2(\Omega \times (0,1))$. Then $\delta (u) \in \mathbb{D}^{1,2}$ and 
\begin{equation}
\label{der_int}
D_t(\delta (u))=u(t)+\int_0^1 D_t u(s)\, {\rm d}B_s.
\end{equation}
\end{proposition}
The above formula is useful in applications: it allows to compute the Malliavin derivative of random variables that are given as a stochastic integral.

\bigskip
To conclude this part, we show that the Skorokhod integral is an extension of the It\^o integral.
Let us denote by $L^2_a$ the closed subspace of $L^2(\Omega \times (0,1))$ formed by adapted processes.
\begin{proposition}
The space $L^2_a$ is contained in \emph{Dom}$(\delta)$ and the operator $\delta$ restricted to $L^2_a$ coincides with the It\^o integral.
\end{proposition}
\begin{proof}
Let $u$ be an elementary adapted process of the form 
\begin{equation}
\label{u_proof}
u=\sum_{i=1}^N Y_k \pmb{1}_{[s_k, t_k)}, 
\end{equation}
for some $N \ge 1$, some times $0\le s_k <t_k \le 1$ and some square-integrable $\mathcal{F}_{s_k}$-adapted random variables $Y_k \in \mathcal{P}(\mathcal{H}_1)$. Without loss of generality we suppose any $Y_k$ to be of the form 
\begin{equation*}
Y_k=\varphi(W(h_1),...,W(h_m)),
\end{equation*}
for some polynomial function $\varphi :\mathbb{R}^n \rightarrow \mathbb{R}$ and each $h_i \in L^2(0,1)$ satisfying supp$(h_i) \subset [0, s_k)$.
Exploiting \eqref{ibp_cor}, for a given $X \in \mathcal{P}(\mathcal{H}_1)$, we infer
\begin{align}
\label{sko_ito_ext}
\mathbb{E}\left[\langle u, DX\rangle\right] 
=\sum_{k=1}^N\mathbb{E}\left[Y_k \langle \pmb{1}_{[s_k,t_k)},DX\rangle\right]
= \sum_{k=1}^N\left(\mathbb{E}\left[Y_k X W(\pmb{1}_{[s_k,t_k)})\right]- \mathbb{E}\left[X \langle DY_k, \pmb{1}_{[s_k,t_k)}\rangle\right]\right).
\end{align}
Now, since $DY_k$ 
has the form 
\begin{equation*}
DY_k= \sum_{i=1}^m \partial_i \varphi(W(h_1),...,W(h_m))h_i,
\end{equation*}
for any $k$ one has  
\begin{equation*}
\langle DY_k, \pmb{1}_{[s_k,t_k)}\rangle=  \sum_{i=1}^m \partial_i \varphi(W(h_1),...,W(h_m))\langle h_i, \pmb{1}_{[s_k,t_k)}\rangle=0.
\end{equation*}
In fact $\langle h_i, \pmb{1}_{[s_k,t_k)}\rangle_{L^2(0,1)}=0$, since each $h_i \in L^2(0,1)$ is such that supp$(h_i) \subset [0, s_k)$.
Combining this with identity \eqref{sko_ito_ext}, and recalling that for an elementary adapted process of the type \eqref{u_proof}, its It\^o integral is given by
\begin{equation*}
\int_0^1 u(t)\, {\rm d}B_t
:= \sum_{k=1}^N Y_k\left(B_{t_k}-B_{s_k}\right)
=\sum_{k=1}^N Y_kW(\pmb{1}_{[s_k,t_k)}),
\end{equation*}
we conclude that 
\begin{equation*}
\mathbb{E}\left[\langle u, DX\rangle\right] = \mathbb{E}\left[ X \int_0^1 u(t)\, {\rm d}B_t\right].
\end{equation*}
Since the set of elementary adapted processes is dense in $L^2_a$, taking limits on both sides of the above identity, we conclude that it holds for every $u \in L^2_a$ and this concludes the proof.
\end{proof}

\begin{remark}
Notice that formula \eqref{der_int} still applies to processes in $L_a^2$, provided they also satisfy the Assumptions of Proposition \ref{pro_der_int}. In this case, 
\begin{equation*}
D_t\left(\int_0^1u(s)\, {\rm d}B_s\right)=u(t)+\int_t^1 D_t u(s)\, {\rm d}B_s,
\end{equation*}
since $D_t u(s)=0$ for a.e. $s, t \in [0,1]$ with $t>s$ (see \cite[Corollary 1.2.1]{Nualart}).
\end{remark}

\subsection{Existence and smoothness of densities}

\subsubsection{An application of the duality relation. }
Using the duality \eqref{duality_rel} between the operators $D$ and $\delta$ it is possible to prove that the image law of a random variable $F: (\Omega, \sigma(\mathcal{H}_1), \mathbb{P})\rightarrow \mathbb{R}$
is absolutely continuous with respect to the Lebesgue measure on $\mathbb{R}$ and to  derive an explicit formula for its density. 
\begin{proposition}
Let $F$ be a random variable in the space $\mathbb{D}^{1,2}$ such that $\|DF\|^2_{\mathcal{H}}>0$ a.s.. Suppose that $\frac{DF}{\|DF\|^2_{\mathcal{H}}}$ belongs to the domain of the divergence operator $\delta$ in $L^2(\Omega)$. Then the law of $F$ has a continuous and bounded density given by
\begin{equation*}
p(x)=\mathbb{E}\left[\pmb{1}_{(F>x)}\delta \left(\frac{DF}{\|DF\|^2_{\mathcal{H}}}\right) \right].
\end{equation*}
\end{proposition}
\begin{proof}
Let $\psi$ be a nonnegative smooth function with compact support and set $\varphi(y)_= \int_{-\infty}^y\psi(z)\, {\rm d}z$. Proposition \ref{chain_rule} ensures that $\varphi(F) \in \mathbb{D}^{1,2}$ and $D\varphi(F)=\psi(F)DF$. Thus we write
\begin{equation*}
\langle D\varphi(F), DF\rangle_{\mathcal{H}}=\langle \psi(F)DF,DF\rangle_{\mathcal{H}}=\psi(F)\|DF\|^2_{\mathcal{H}}.
\end{equation*}
Using the duality relation \eqref{duality_rel} between $D$ and $\delta$ we obtain
\begin{equation}
\label{density_1}
\mathbb{E}\left[\psi(F)\right] 
= \mathbb{E}\left[\biggl\langle D\varphi(F), \frac{DF}{\|DF\|^2_{\mathcal{H}}} \biggr\rangle_{\mathcal{H}}\right] 
= \mathbb{E} \left[ \varphi(F)\delta \left( \frac{DF}{\|DF\|^2_{\mathcal{H}}}\right)\right].
\end{equation}
By an approximation argument, \eqref{density_1} holds for $\psi(y)=\pmb{1}_{[a,b]}(y)$ with $a<b$. Consequently, Fubini Theorem yields
\begin{align*}
\mathbb{P}(a \le F \le b)
&=\mathbb{E}\left[\left(\int_{-\infty}^F \pmb{1}_{[a,b]}(z)\, {\rm d}z \right)\ \delta\left(\frac{DF}{\|DF\|^2_{\mathcal{H}}} \right) \right]
\\
&=\mathbb{E}\left[\left(\int_{\mathbb{R}} \pmb{1}_{(-\infty,F)}(z)\pmb{1}_{[a,b]}(z)\, {\rm d}z \right)\ \delta\left(\frac{DF}{\|DF\|^2_{\mathcal{H}}} \right) \right]
=\int_a^b \mathbb{E}\left[ \pmb{1}_{(F>z)} \delta\left(\frac{DF}{\|DF\|^2_{\mathcal{H}}}\right)\right]\, {\rm d}z,
\end{align*}
which concludes the proof.
\end{proof}

\subsubsection{General criteria for the existence and smoothness of densities}
In general, the techniques of the Malliavin calculus can be applied to study
the regularity of the probability law of random vectors defined on the probability space $(\Omega, \sigma(\mathcal{H}_1), \mathbb{P})$. General criteria for the absolute continuity and regularity of the density of random vectors are analyzed in \cite[Chapter 2]{Nualart}. These general
criteria can be applied to infer the existence and regularity of the density for the solutions of stochastic differential equations
and stochastic partial differential equations driven by a Gaussian noise, see e.g.  \cite{BoZa2,BoZa1,FerZan, Bally1998, Cardon-Weber2001, Marinelli2013, Marquez-Carreras2001, Mueller2008, Morien1999, Sard_Nualart, Pardoux1993, Sanz_Sardanyons, NualartZaidi1997}
 and the therein references.
We briefly recall some of these criteria.
\begin{definition}
Let $F: (\Omega, \sigma(\mathcal{H}_1), \mathbb{P})\rightarrow \mathbb{R}$ be a random vector with components $F^i \in \mathbb{D}^{1,2}$, $i=1,...,n$. The \emph{Malliavin matrix} of $F$ is the $n \times n$ matrix, denoted by $\Gamma$, whose entries are the random variables $\Gamma_{i,j}=\langle DF^i, DF^j\rangle_{\mathcal{H}}$, $i,j=1,...,n$.
\end{definition}

Suppose that the random vector $F: (\Omega, \sigma(\mathcal{H}_1), \mathbb{P})\rightarrow \mathbb{R}$ is of the form $F=(W(h_1),...,W(h_n))$, for $h_i \in \mathcal{H}$, $i=1,...,n$. The entries of its Malliavin matrix are given by $\Gamma_{i,j}=\langle h_i, h_j\rangle_{\mathcal{H}}$, that is, the Malliavin matrix coincides with the covariance matrix of the centered Gaussian vector $F$. Recall now that we want to obtain an absolute continuity criterium (w.r.t. the Lebesgue measure on $\mathbb{R}^n$) for the law of $F$. If $\Gamma$ is not invertible then $F$ lives, in the sense of the support of the law, in some subspace of $\mathbb{R}^n$ of dimension $m<n$ and thus may not be absolutely continuous w.r.t. the Lebesgue measure on $\mathbb{R}^n$. In fact, the law of a $\mathbb{R}^n$-valued Gaussian random variable is absolutely continuous w.r.t. the Lebesgue measure on $\mathbb{R}^n$ iff its covariance matrix is non degenerate.
\newline
The following criterium, due to Bouleau and Hirsch, see \cite{BH}, appears as a quite natural extension of the particular example described above.
\begin{theorem}
Let $F: (\Omega, \sigma(\mathcal{H}_1), \mathbb{P})\rightarrow \mathbb{R}$ be a random vector satisfying the following conditions:
\begin{itemize}
\item [i)] $F^i \in \mathbb{D}^{1,2}$, for any $i=1,...,n$,
\item [ii)] the Malliavin matrix is invertible $\mathbb{P}$-a.s.
\end{itemize}
Then the law of $F$ has a density w.r.t. the Lebesgue measure on $\mathbb{R}^n$.
\end{theorem}
Under stronger regularity assumptions one can also infer the \emph{smoothness} of the density of the random vector. For the proof of the following result see e.g. \cite[Theorem 4.2]{Sanz}. 

\begin{theorem}
Let $F: (\Omega, \sigma(\mathcal{H}_1), \mathbb{P})\rightarrow \mathbb{R}$ be a random vector satisfying the following conditions:
\begin{itemize}
\item [i)] $F^i \in \mathbb{D}^\infty:= \bigcap_{p \ge 1}\bigcap_{k \ge 1}\mathbb{D}^{k,p}$, for any $i=1,...,n$,
\item [ii)] the Malliavin matrix is invertible $\mathbb{P}$-a.s. and 
\begin{equation*}
\emph{det}( \Gamma^{-1} ) \in \bigcap_{p \ge 1}L^p(\Omega).
\end{equation*}
\end{itemize}
Then the law of $F$ has an infinitely differentiable density w.r.t. the Lebesgue measure on $\mathbb{R}^n$.
\end{theorem}


\section{The relation between $D$, $\delta$ and $L$}
\label{relation_sec}


In this Section we clarify the relation between the Malliavin derivative $D$, the divergence operator $\delta$ and the Ornstein-Uhlenbeck operator $L$.

\begin{proposition}
\label{rela_prop}
Let $F\in L^2(\Omega)$. $F \in \text{\emph{Dom}}(L)$ if and only if $F \in \mathbb{D}^{1,2}$ and $DF \in \text{\emph{Dom}}(\delta)$. In this case $\delta DF=-LF$.
\end{proposition}

\begin{proof}
Let us assume first that $F \in$ Dom$(L)$. We know that Dom$(L) = \mathbb{D}^{2,2}$ which implies $F \in \mathbb{D}^{1,2}$ and $DF \in \mathbb{D}^{1,2}(\mathcal{H}) \subset$ Dom$(\delta)$. Let $G \in \mathcal{P}(\mathcal{H}_1)$, using the duality relation and Proposition \ref{charD12}, we have
\begin{align}
\label{stima_rela}
\mathbb{E}\left[G\delta(DF) \right]
& =\mathbb{E}\left[\langle DG, DF\rangle_{\mathcal{H}} \right] = \sum_{n=0}^\infty n \mathbb{E}\left[J_nFJ_nG \right]
=\mathbb{E}\left[\sum_{n=0}^\infty nJ_nF\sum_{m=0}^\infty J_mG \right]=\mathbb{E}\left[G(-LF)\right]
\end{align}
and by density we infer that $\delta (DF)=-LF$. For the converse, since $F\in \mathbb{D}^{1,2}$ and $DF \in$ Dom$(\delta)$, we can apply directly \eqref{stima_rela} and hence $F \in$ Dom$(L)$ and $\delta DF=-LF$.

\end{proof}

From \eqref{Dom_L} and Proposition \ref{charDk2} we immediately see that Dom$(L)=\mathbb{D}^{2,2}$. 
On Dom$(L)$ we can introduce the norm
\begin{equation*}
\|F\|_L:=\left(\mathbb{E}\left[|F|^2\right]+\mathbb{E}\left[|LF|^2\right] \right)^{\frac 12}.
\end{equation*}
The norms $\|\cdot\|_L$ and $\|\cdot\|_{\mathbb{D}^{2,2}}$ coincide; in fact, bearing in mind \eqref{L_op} and Proposition \ref{charDk2}, we infer
\begin{align*}
\mathbb{E}\left[|F|^2\right]+\mathbb{E}\left[|LF|^2\right]
= \sum_{n=0}^\infty (1+n^2)\|J_nF\|^2_{L^2(\Omega)}
=\mathbb{E}\left[|F|^2\right]+ \mathbb{E}\left[\|DF\|^2_{\mathcal{H}}\right]+ \mathbb{E}\left[ \|D^2F\|^2_{\mathcal{H} \otimes \mathcal{H}}\right].
\end{align*}
Let us now introduce the operator $C:=-\sqrt{-L}$ as
\begin{equation*}
CF=\sum_{n=0}^\infty -\sqrt{n}J_nF,
\end{equation*}
\begin{equation*}
\text{{Dom}}(C):=\{F \in L^2(\Omega): \ \sum_{n=0}^\infty n \mathbb{E}\left[ |J_nF|^2\right] < \infty\}.
\end{equation*}
As in the case of the operator $L$, one can show that $C$ is the infinitesimal generator of a semigroup of operators (Cauchy semigroup) given by 
\begin{equation*}
P_t F=\sum_{n=0}^\infty e^{-\sqrt{n}t}J_nF.
\end{equation*}
Reasoning as above, we see that Dom$(C)=\mathbb{D}^{1,2}$. On Dom$(C)$ we  introduce the norm
\begin{equation*}
\|F\|_C:=\left(\mathbb{E}\left[|F|^2\right]+\mathbb{E}\left[|CF|^2\right] \right)^{\frac 12},
\end{equation*}
and one has that the norms $\|\cdot\|_C$ and $\|\cdot\|_{\mathbb{D}^{1,2}}$ coincide. In fact 
\begin{align*}
\mathbb{E}\left[|F|^2\right]+\mathbb{E}\left[|CF|^2\right]
= \sum_{n=0}^\infty (1+n)\|J_nF\|^2_{L^2(\Omega)}
=\mathbb{E}\left[|F|^2\right]+ \mathbb{E}\left[\|DF\|^2_{\mathcal{H}}\right].
\end{align*}
\begin{remark}
For $1<p<\infty$, $p\ne 2$, there is no equality of the norms $\|CF\|_{L^p(\Omega)}$ and $\|DF\|_{L^p(\Omega;\mathcal{H})}$, but one can prove that the norms are equivalent. The result is a consequences of the \emph{Meyer inequalities}, see e.g. \cite[Section 1.5]{Nualart}.
\end{remark}


We conclude this Section with the following result that shows the behavior of $L$ as a second order differential operator.

\begin{proposition}
Let $F \in\emph{Dom}(L)$. Fixing an orthonormal basis $\{h_n\}_{n \in \mathbb{N}}$ of $\mathcal{H}$, it holds 
\begin{equation}
\label{OU_inf}
LF=\sum_{j=1}^\infty \langle D\langle DF, h_j\rangle_{\mathcal{H}},h_j  \rangle_{\mathcal{H}}-\sum_{j=1}^\infty \langle DF, h_j\rangle_{\mathcal{H}}W(h_j).
\end{equation}
\end{proposition}
\begin{proof}
Let $F \in \text{Dom}(L)$. Since Dom$(L)=\mathbb{D}^{2,2}$, $DF \in \mathbb{D}^{1,2}(\mathcal{H})$. 
Applying Proposition \ref{inf_dim_div} to $DF$, we thus infer 
\begin{equation*}
\delta (DF)=\sum_{j=1}^\infty \langle DF, h_j\rangle_{\mathcal{H}}W(h_j)-\sum_{j=1}^\infty \langle D\langle DF, h_j\rangle_{\mathcal{H}},h_j  \rangle_{\mathcal{H}}.
\end{equation*}
Bering in mind the ralation $\delta DF=-LF$, see Proposition \ref{rela_prop}, we immediately infer the thesis.
\end{proof}

\begin{remark}
In the finite dimensional case, the Ornstein-Uhlenbeck operator and the Ornstein-Uhlenbeck semigroup play the role of the Laplacian and the heat semigroup, respectively, if the Lebesgue measure is replaced by the standard Gaussian measure.
\\
Let $(\mathbb{R}^n,\mathcal{B}(\mathbb{R}^n))$ and  consider on it  the standard Gaussian measure. Given the scalar field $F:\mathbb{R}^n \rightarrow \mathbb{R}$, and denoted by $\{e_j\}_{j=1}^n$ the canonical  basis of $\mathbb{R}^n$, the Ornstein-Uhlenbeck operator is related to the Laplace operator by the relation
 \begin{equation}
 \label{Lap_fin}
L F(x)=
\Delta F(x)- \langle x, \nabla F(x)\rangle_{\mathbb{R}^n}
=\sum_{j=1}^n \langle \nabla \langle \nabla \langle F(x), e_j\rangle_{\mathbb{R}^n}, e_j\rangle_{\mathbb{R}^n}- \sum_{j=1}^n \langle \nabla F(x), e_j\rangle_{\mathbb{R}^n}\langle x, e_j\rangle_{\mathbb{R}^n}.
\end{equation}
Comparing the infinite dimensional case, that is \eqref{OU_inf}, with the final dimensional case, that is \eqref{Lap_fin}, we can interpret the term $\sum_{j=1}^\infty \langle D\langle DF, h_j\rangle_{\mathcal{H}},h_j  \rangle_{\mathcal{H}}$ as an infinite-dimensional generalization of $\Delta F
$. The additional term $\sum_{j=1}^\infty \langle DF, h_j\rangle_{\mathcal{H}}W(h_j)$ that appears in \eqref{OU_inf} is due to the Gaussian framework and it has to be compared with $\langle x, \nabla F(x)\rangle_{\mathbb{R}^n}$.
\end{remark}

\section{Malliavin calculus on probability spaces with a linear topological structure: some results in the literature}
\label{top_sec}

The problem of differentiability along subspaces arises in a natural way in the study of differential equations with values in linear topological spaces such as Hilbert or Banach spaces. Many authors have introduced many definitions for differentiability along subspaces, the most widely used are two. One introduced by L. Gross in \cite{GRO1} and systematically presented by V. I. Bogachev in \cite{Bogachev}, the other one introduced by P. Cannarsa and G. Da Prato in \cite{CDP96}. Exploiting the topological linear structure of the space (usually a Hilbert or a Banach space), the authors introduce the notion of gradient along subspaces for sufficiently smooth functions. These gradient operators are introduced as limits of suitable incremental ratios; for more details see \cite{BFFZ}.
Malliavin derivative operators naturally appear in the setting considered by Gross and Cannarsa--Da Prato when a Gaussian framework comes into play. 
In fact, in the specific case where the subspace along which to differentiate is the Cameron--Martin space associated to a given Gaussian measure, it is possible to extend such gradient operators to spaces of less regular functions, i.e., Sobolev spaces with respect to the reference Gaussian measure. The Malliavin derivative that appears in the framework considered by Gross is systematically studied in the book \cite{Bogachev}, whereas the Malliavin derivative that appears in the framework considered by Cannarsa and Da Prato is analyzed in the book \cite{DaPrato}. In this Section we briefly recall the construction of these two Malliavin derivatives, referred from here on as the Malliavin derivatives in the sense of Bogachev and in the sense of Da Prato, respectively, and we  show that these Malliavin derivatives are two \textit{different} operators (although linked by a relationship that we will clarify) that still have the \textit{same} Sobolev space as their  domain.

In order to make a rigorous comparison between the Malliavin derivatives in the sense of Bogachev and Da Prato, we are helped by the abstract approach to Malliavin calculus of the previous sections, see in particular Section \ref{Malliavin_sec}. Referring back to Section \ref{Malliavin_sec}, in order to make a comparison, we will need to identify the Gaussian Hilbert space $\mathcal{H}_1$ and to characterize it in terms of a suitable choice of a separable Hilbert space $\mathcal{H}$ and a unitary operator $W$. What we will show is that the Malliavin derivatives in the sense of Bogachev and Da Prato can be interpreted as two particular examples of the general notion of Malliavin derivative given in Section \ref{Malliavin_sec}. This will make it possible, in a rather simple way, to also understand the relationship between the two derivatives. 

At first glance it might seem strange to refer to Malliavin derivatives that are different, since one usually speaks of \textit{the} Malliavin derivative. We point out that, in fact, it would be more appropriate to speak of \textit{a} (choice of) Malliavin derivative rather than \textit{the} Malliavin derivative. In fact, as explained in details in Section \ref{Malliavin_sec}, one can construct infinitely many different Malliavin derivative operators. 
On the other hand, it turns out that all these Malliavin derivatives have the same domain when the Gaussian Hilbert space $\mathcal{H}_1$ is the same. 
In a sense, the results of the present Section can be considered as an example of this general fact in a concrete situation: we deal with two particular Malliavin derivatives that naturally appear in the literature for the study of different problems. However, these Malliavin derivatives are just two possible choices among the infinite possible ones that would be possible to consider in that specific Gaussian framework.


In details, in Subsections \ref{Gaussian_Banach_sec} and \ref{Gaussian_Hilbert_sec} we recall the construction of  the Malliavin derivatives in the sense of Bogachev (see \cite{Bogachev}) and in the sense of Da Prato (see \cite{DaPrato}), respectively. In Subsection \ref{comparison_sec} we make a comparison of the two operators.

\bigskip

\subsection{Malliavin derivative in the sense of Bogachev}
\label{Gaussian_Banach_sec}
For this part we mainly refer to \cite{Bogachev} and \cite{Lunardi}, see also \cite{GRO1}, \cite{Bax} and \cite{vtc}.
\\
Throughout this Subsection we denote by $X$ an infinite dimensional separable real Banach space. By $\|\cdot\|_X$ we denote the corresponding norm. Recall that a space is separable if it contains a countable dense subset.  We denote by $X^*$, with norm $\|\cdot\|_{X^*}$, the topological dual of $X$ consisting of all linear continuous functions $f:X \rightarrow \mathbb{R}$; we denote by $X'$ the algebraic dual of $X$ consisting of all linear functions $f:X \rightarrow \mathbb{R}$. The open (resp. closed) ball with centre $x \in X$ and radius $r>0$ are denoted by $B(x,r)$ (resp. $\bar B(r, x)$). By $\mathcal{B}(X)$ we denote the Borel $\sigma$-algebra on $X$. We recall that in this setting the $\sigma$-algebra $\mathcal{E}(X)$ generated by cylindrical sets and the Borel $\sigma$-algebra $\mathcal{B}(X)$ coincides (see e.g. \cite[Theorem 2.1.1]{Lunardi}).

We start by emphasizing that measure theory in infinite dimensional spaces is not a trivial issue since there is no equivalent of the Lebesgue measure, that is, there is not a nontrivial measure which is invariant by translation, not even in a Hilbert space.

\begin{proposition}
Let $X$ be an infinite dimensional separable Hilbert space. If $\mu:\mathcal{B}(X) \rightarrow [0, \infty]$ is a $\sigma$-additive set function such that:
\begin{itemize}
\item [(i)] $\mu(x+B)=\mu(B)$, for every $x \in X$, $B \in \mathcal{B}(X)$,
\item [(ii)] $\mu(B(0,r))>0$ for every $r>0$,
\end{itemize}
then $\mu(A)=\infty$ for every open set $A$.
\end{proposition}
\begin{proof}
Assume that $\mu$ satisfies (i) and (ii). Let $\{e_k\}_{k \in \mathbb{N}}$ be an orthonormal system in $X$. For any $k \in \mathbb{N}$ consider the balls $B_k$ with center $\frac34 r e_k$ and radius $\frac r4$: they are pairwise disjoint and, by assumption, have the same measure, say $\mu(B_k)=m>0$ for all $k \in \mathbb{N}$. Thus 
\begin{equation*}
B(0,r)\supset \bigcup_{k \in \mathbb{N}}B_k \quad \Longrightarrow  \mu(B(0,r)) \ge \sum_{k \in \mathbb{N}}\mu(B_k)= \sum_{k \in \mathbb{N}}m=\infty,
\end{equation*}
which in turns implies $\mu(A)= \infty$ for every open set $A$.
\end{proof}
In infinite dimensional frameworks Gaussian probability measures usually play the role of the main reference measures. As we will see, this is related to the fact that Gaussian measures are ‘quasi-invariant’ by translations along the directions of the Cameron-Martin space.

\subsubsection{Gaussian measures in Banach spaces}
 \label{Gaus_meas_Ban}

\begin{definition}
\emph{(Gaussian measures on $X$).}  A probability measure $\gamma$ on $(X, \mathcal{B}(X))$ is said to be \emph{Gaussian} if $\gamma \circ f^{-1}$ is a Gaussian measure in $\mathbb{R}$ for every $f \in X^*$. The Gaussian measure is called \emph{centered} if all the measures $\gamma \circ f^{-1}$ are centered and it is called \emph{nondegenerate} if for any $f \in X^*\setminus\{0\}$, the measure $\gamma \circ f^{-1}$ is nondegenerate.
\end{definition}
If $f \in X^*$ then $f \in L^p(X,\gamma)$ for any $p \ge 1$. In fact, 
\begin{equation*}
\int_X |f(x)|^p\,\gamma({\rm d}x) = \int_{\mathbb{R}} |s|^p(\gamma \circ f^{-1})({\rm d}s)
\end{equation*}
and the integral in the r.h.s. of the above equality is finite since $\gamma \circ f^{-1}$ is a Gaussian measure in $\mathbb{R}$. This ensures that the following definition is well posed.
\begin{definition}
We define the mean $a_\gamma$ and the covariance $K_\gamma$ of $\gamma$ as 
\begin{equation*}
a_\gamma(f):=\int_Xf(x)\, \gamma({\rm d}x),  \quad f \in X^*
\end{equation*}
\begin{equation*}
K_\gamma(f,g):=\int_X [f(x)-a_\gamma(f)][g(x)-a_\gamma(g)]\, \gamma({\rm d}x),\quad f, g \in X^*.
\end{equation*}
\end{definition}

\begin{remark}
When $X = \mathbb{R}^d$, $K$ is the covariance matrix, provided that we perform the usual identification of $\mathbb{R}^d$ with its dual.
\end{remark}
Notice that $a_\gamma:X^*\rightarrow \mathbb{R}$, $\mapsto K_\gamma :X^*\times X^* \rightarrow \mathbb{R}$ are a linear and a bilinear map respectively and $K_\gamma(f,f)=\|f-a_\gamma(f)\|^2_{L^2(X, \gamma)}\ge 0$ for every $f \in X^*$. Moreover, it can be proved that $a_\gamma$, $K_\gamma$ are continuous (see \cite[Proposition 2.3.5]{Lunardi}) as a consequence of the Fernique Theorem (see \cite[Theorem 2.3.1]{Lunardi}).

A pair $(a, K)$ determines the measure $\gamma$. In fact we have a characterization of $\gamma$ in terms of characteristic functions analogous to the one that one has in $\mathbb{R}^d$.
We recall that given a measure $\mu$ on the space $(X, \mathcal{B}(X))$, its characteristic function (or Fourier transform) is defined by
\begin{equation}
\label{car_fun_def}
\hat{\mu}(f):= \int_X e^{ i f(x)}\, \mu({\rm d} x), \quad f \in X^*.
\end{equation}
Moreover, we recall that characteristic functions determine measures, that is given two probability measures $\mu_1, \mu_2$ on $(X, \mathcal{B}(X))$, if $\widehat{\mu_1}=\widehat{\mu_2}$ then $\mu_1=\mu_2$ (see \cite[Proposition 2.12]{Lunardi}).
\begin{proposition}
\label{CF_gamma_prop}
If $\gamma$ is a Gaussian measure on $X$, then
\begin{equation}
\label{CF_gamma}
\hat{\gamma}(f)=e^{ia_\gamma(f)-\frac 12 K_\gamma(f,f)}, \quad f \in X^*.
\end{equation}
Conversely,  if a Borel probability measure $\gamma$ is such that 
\begin{equation}
\label{CF_gamma_bis}
\hat{\gamma}(f)=e^{ia(f)-\frac 12 K(f,f)}, \quad f \in X^*,
\end{equation}
for some linear $a: X^*\rightarrow \mathbb{R}$ and for some bilinear symmetric nonnegative $K:X^*\times X^* \rightarrow \mathbb{R}$, then $\gamma$ is a Gaussian measure with mean $a$ and covariance $K$.
\end{proposition}
\begin{proof}
Assume that $\gamma$ is a Gaussian measure. In view of \eqref{car_fun_def} and \eqref{car_fun_Gau_1D} we get 
\begin{align*}
\hat \gamma(f)=\int_X e^{ i f(x)}\, \gamma({\rm d} x)= \int_{\mathbb{R}}e^{ i \xi}\, (\gamma \circ f^{-1})({\rm d} \xi)=e^{i\mu-\frac {\sigma^2}{2}},
\end{align*}
where $\mu$ and $\sigma^2$ are the mean and the variance of $\gamma \circ f^{-1}$, given by
\begin{equation*}
\mu= \int_\mathbb{R}\xi\, (\gamma \circ f^{-1})({\rm d}\xi)=\int_X f(x)\, \gamma({\rm d}x)=a_\gamma (f), 
\end{equation*}
and 
\begin{equation*}
\sigma^2= \int_{\mathbb{R}} (\xi - \mu)^2\,(\gamma \circ f^{-1})({\rm d}\xi)=\int_X \left(f(x)-a_\gamma(f)\right)^2\, \gamma({\rm d}x)=K_\gamma (f,f).
\end{equation*}
This proves \eqref{CF_gamma}.
\\
Conversely, let $\gamma$ be a Borel probability measure on $X$ such that \eqref{CF_gamma_bis} holds for some linear $a$ and bilinear $K$ and for every $f \in X^*$. For $f \in X^*$ we compute 
\begin{align*}
\widehat{\gamma \circ f^{-1}}(\tau)= \int_{\mathbb{R}}e^{i\tau \xi} \, (\gamma \circ f^{-1})({\rm d}\xi) = \int_X e^{i\tau f(x)}\, \gamma({\rm d}x)= e^{i \tau a(f)-\frac 12 \tau^2 K(f,f)}.
\end{align*}
 This proves that, for any $f \in X^*$, $\gamma \circ f^{-1} = \mathcal{N}(a(f), K(f,f))$ is Gaussian which proves the statement.
\end{proof}

\subsubsection{The Gaussian Hilbert space $X_\gamma^*$}

The space $X^*$ is included in $L^2(X, \gamma)$. The inclusion map $j:X^* \rightarrow L^2(X, \gamma)$,
\begin{equation*}
j(f)=f-a_\gamma(f), \quad f \in X^*,
\end{equation*}
is continuous since $\|j(f)\|_{L^2(X, \gamma)} \le (\sqrt{c_1}+c_0)\|f\|_{X^*}$, with $c_0=\int_X \|x\|\, \gamma({\rm d}x)$ and $c_1=\int_X\|x\|^2\, \gamma({\rm d}x)$ finite constants, as a consequence of the Fernique Theorem.
Moreover, if $\gamma$ is nondegenerate, $j$ is one to one. 
In fact, let $f\in X^*$ be such that $j(f)=0$. Then $\hat \gamma(f)=\widehat{\gamma \circ f^{-1}}=e^{i a_\gamma(f)}$, which implies by Theorem \ref{CF_gamma_prop}, that $\gamma \circ f^{-1}=\delta_{a_\gamma(f)}$ is degenerate. Since $\gamma$ is nondegenerate this yields $f=0$.

\begin{definition}
\label{Xstargammadef}
We define the space 
\begin{equation*}
X_\gamma^*:= \text{the closure of $j(X^*)$ in $L^2(X, \gamma)$},
\end{equation*}
i.e. $X_\gamma^*$ consists of all limits in $L^2(X, \gamma)$ of sequence of functions $j(f_h)=f_h-a_\gamma(f_h)$ with $(f_h) \subset X^*$.
\end{definition}
So far we have defined the functions $a_\gamma$, $\hat{\gamma}$ in $X^*$ and the function $K_\gamma$ in $X^* \times X^*$.
The extension of $a_\gamma$ to $X^*_\gamma$ is trivial: since the mean value of every element of $X^*_\gamma$ is zero, we define $a_\gamma$ on $X^*_\gamma$ by setting $a_\gamma=0$. The extension of $K_\gamma$ to $X^*_\gamma \times X^*_\gamma$ is obviously continuous ($X^*_\gamma \times X^*_\gamma$ is endowed with the $L^2(X, \gamma) \times L^2(X, \gamma)$ norm) and since $a_\gamma=0$ on $X^*_\gamma$, we define the bilinear form $K_\gamma$ on $X^*_\gamma$ by setting
\begin{equation*}
K_\gamma(f,g)=\int_Xf(x)g(x)\, \gamma({\rm d} x)=\langle f, g\rangle_{L^2(X, \gamma)}, \quad f, g \in X^*_\gamma.
\end{equation*}
The extension of $\widehat \gamma$ to $X_\gamma^*$ is given by setting 
\begin{equation*}
\widehat \gamma(f)=\int_X e^{if(x)}\, \gamma({\rm d}x), \quad f \in X_\gamma^*
\end{equation*}
and 
we have the following result.
\begin{proposition}
\label{lem_Gaus_X^*_gamma}
Let $\gamma$ be a Gaussian measure on a separable Banach space $(X, \mathcal{B}(X))$. Then \begin{equation}
\label{CF_gamma_X^*_gamma}
\widehat{\gamma}(f)=e^{-\frac 12 \|f\|^2_{L^2(X, \gamma)}}, \quad \forall \ f \in X_\gamma^*.
\end{equation}
\end{proposition}
\begin{proof}
Let $f \in X^*_{\gamma}$ and let $g_n:=j(f_n)$, $f_n \in X^*$, be a sequence of functions converging to $f$ in $L^2(X, \gamma)$. Then $a_\gamma(g_n)=0$ for any $n \in \mathbb{N}$. Convergence in $L^2$ implies convergence in distribution, whence by the L\'evy Theorem,
\begin{equation*}
\hat \gamma (f)= \lim_{n \rightarrow \infty}\hat \gamma (g_n)=\lim_{n \rightarrow \infty} e^{-\frac 12 K_\gamma (g_n,g_n)} = e^{-\frac 12 \|f\|^2_{L^2(X, \gamma)}}.
\end{equation*}
\end{proof}
\begin{corollary}
\label{X^* Gauss}
Let $\gamma$ be a Gaussian measure on a separable Banach space $(X, \mathcal{B}(X))$. Then $X_\gamma^*$, endowed with the scalar product $\langle \cdot, \cdot\rangle_{L^2(X, \gamma)}$ is a Gaussian Hilbert space. In particular, every element $f \in X_\gamma^*$ is a centered Gaussian random variable with variance $\|f\|_{L^2(X, \gamma)}^2$.
\end{corollary}
\begin{proof}
$X_\gamma^*$ is an Hilbert space since, by definition, it is a closed subset of the Hilbert space $L^2(X, \gamma)$. 
It remains to prove that every element in $X_\gamma^*$ has centered normal distribution. This is an immediate consequence of Proposition \ref{lem_Gaus_X^*_gamma}. In fact \eqref{CF_gamma_X^*_gamma} yields that, for any $f \in X_\gamma^*$, the law of $f$ in $\mathbb{R}$ is $\gamma \circ f^{-1}=\mathcal{N}(0, \|f\|^2_{L^2(X, \gamma)})$, that is $f$ is a Gaussian random variable with zero mean and variance $\|f\|^2_{L^2(X, \gamma)}$.
\end{proof}

We conclude this part by introducing the operator $R_\gamma: X_\gamma^* \rightarrow (X^*)'$ defined as
\begin{equation}
\label{R_gamma_def}
R_\gamma f(g):=\int_X f(x)[g(x)-a_\gamma(g)]\, \gamma({\rm d}x)=\langle f, g-a_\gamma(g)\rangle_{L^2(X, \gamma)}, \quad f \in X^*_\gamma, \ g \in X^*.
\end{equation}
It is important to notice that $R_\gamma$ maps $X_\gamma^*$ into $X$ (see \cite[Proposition 2.3.6]{Lunardi}). 
Thanks to this fact we can identify $R_\gamma f$ with the element $y \in X$
representing it, i.e. we shall write
\begin{equation}
\label{equality_R_gamma}
R_\gamma f(g)=g(R_\gamma(f)), \quad \forall g \in X^*.
\end{equation}

\subsubsection{The Cameron-Martin space}
\label{CM_sec}
Given a Gaussian measure $\gamma$ on a separable Banach space $X$, it is possible to associate to it a Hilbert space $H \hookrightarrow X$ in a canonical way. $H$ is called the Cameron-Martin space of $\gamma$. 
For $h \in X$ define the shifted measure $\gamma_h$ as 
\begin{equation}
\label{shift_gamma}
\gamma_h(B):=\gamma(B-h), \qquad B \in \mathcal{B}(X).
\end{equation}
The main importance of the Cameron-Martin space is that it consists of the elements $h \in X$ such that the measure $\gamma_h$ is absolutely continuous
\begin{footnote}
{
We recall that a measure $\mu$ is absolutely continuous with respect to a measure $\nu$ (we write $\mu \ll \nu$), if $\nu(B)=0$ yields $\mu(B)=0$. Equivalently, $\mu \ll \nu$ if there exists a density $\rho=\frac{{\rm d}\mu}{{\rm d}\nu} \in L^1(X, \nu)$ such that $\mu(B)=\int_B \rho\, {\rm d}\nu.$ $\mu$ and $\nu$ are said to be equivalent (we write $\mu \sim \nu$) if $\mu \ll \nu$ and $\nu \ll \mu$.
}
\end{footnote}
 with respect to $\gamma$. In other words, $H$ characterizes exactly those directions in $X$ in which translations leave the measure $\gamma$ ‘quasi-invariant’. The space $H$ is strictly smaller than $X$ and in the infinite dimensional case one has $\gamma(H)=0$. Things are thus completely different from the finite-dimensional case where the Lebesgue measure is invariant under translations in \textit{any} direction. 

\begin{definition}
\emph{(Cameron-Martin space).} For every $h \in X$ set
\begin{equation}
\label{def_CM_norm}
\|h\|_H:=\sup\{f(h): f \in X^*, \|j(f)\|_{L^2(X, \gamma)}\le 1\}.
\end{equation}
The Cameron-Martin space is defined as 
\begin{equation*}
H:=\{h \in X: \|h\|_H < \infty\}.
\end{equation*}
\end{definition}

The Cameron-Martin space $H$ is continuously embedded in $X$, in fact
\begin{equation*}
\|h\|_X=\sup\{f(h):\|f\|_{X^*} \le 1\} \le \sup\{f(h): \|j(f)\|_{L^2(X,\gamma)} \le c\}=c\|h\|_H,
\end{equation*}
where $c$ is the norm of $j:X^* \rightarrow L^2(X, \gamma)$. As proved in  \cite[Theorem 3.1.9]{Lunardi}, the embedding $H \hookrightarrow X$ is also compact. Moreover, if $X_\gamma^*$ is infinite-dimensional, then $\gamma(H)=0$.

\bigskip
\textbf{A characterization of the Cameron-Martin space.}

\smallskip
The Cameron-Martin space inherits a natural Hilbert space structure from the space $X_\gamma^*$ through the $L^2(X,\gamma)$ Hilbert structure.
\begin{proposition}
\label{CM_space}
The Cameron-Martin space $H$ is the range of $R_\gamma$, namely an element $h \in X$ belongs to the Cameron-Martin space if and only if there exists an element $\hat h \in X_\gamma^*$ such that $h=R_\gamma \hat h$. In this case
\begin{equation*}
\|h\|_H=\|\hat h\|_{L^2(X, \gamma)}.
\end{equation*}
Therefore, $R_\gamma:X_\gamma^*\rightarrow H$ is an onto isometry and $H$ is a Hilbert space with the inner product
\begin{equation*}
\langle h,k\rangle _H:=\langle \hat h, \hat k\rangle_{L^2(X, \gamma)},
\end{equation*}
whenever $h=R_\gamma \hat h$, $k=R_\gamma \hat k$.
\end{proposition}
\begin{proof}
Let us start by proving that if $h \in H$, then there exists $\hat h \in X^*_\gamma$ such that $h=R_\gamma(\hat h)$.
Given any $h \in X$ with $\|h\|_H< \infty$, from \eqref{def_CM_norm} it follows that for any $f \in X^*$
\begin{equation}
\label{sti1CMproof}
|f(h)|\le \|j(f)\|_{L^2(X, \gamma)}\|h\|_H.
\end{equation} 
As a consequence, the linear map
\begin{equation*}
L_h: j(X^*)\rightarrow \mathbb{R}, \quad L_h(j(f)):=f(h), \quad f \in X^*,
\end{equation*}
is well defined and continuous w.r.t. the $L^2$ topology. Thus $L_h$ can be continuously extended to $X^*_\gamma$ and by the Riesz representation theorem there is a unique $\hat h\in X_\gamma^*$ such that the extension (still denoted by $L_h$) is given by 
\begin{equation*}
L_h(g)=\int_X g(x)\hat h(x)\, \gamma({\rm d}x), \quad g \in X_\gamma^*.
\end{equation*}
In particular, for any $f \in X^*$, 
\begin{equation}
\label{sti2CMproof}
f(h)=L_h(j(f))=\int_Xj(f)(x)\hat h(x)\, \gamma({\rm d}x)=f(R_\gamma\hat h),
\end{equation}
where the last equality comes from \eqref{R_gamma_def} and \eqref{equality_R_gamma}. This yields $R_\gamma \hat h=h$ and  
$
\|h\|_H=\|\hat h\|_{L^2(X, \gamma)}.
$

Let us now prove the converse: if $h =R_\gamma(\hat h)$ for some $\hat h \in R_\gamma^*$, then $h \in H$. Let $h=R_\gamma(\hat h)$, then for ant $f \in X^*$, \eqref{sti2CMproof} yields 
\begin{equation*}
f(h)=\int_X j(f)(x)\hat h(x)\, \gamma({\rm d}x) \le \|\hat h\|_{L^2(X, \gamma)}\|j(f)\|_{L^2(X, \gamma)},
\end{equation*}
whence $\|h\|_H <\infty$ by the definition of the $H$-norm.
\end{proof}


\begin{remark}
The space $L^2(X,\gamma)$ (hence its subspace $X_\gamma^*$) is separable being $X$ separable. Thus $H$ is separable since it is isometric to a separable space.
\end{remark}

As an immediate consequence of Proposition \ref{CM_space} 
we infer the following result, that can be seen as a concrete example of Proposition \ref{Gaussian_isometries}.
\begin{corollary}
\label{char_X^*_gam-H}
The map $\hat \cdot:=R_\gamma^{-1}: H \rightarrow X^*_\gamma$ is a unitary operator and 
\begin{equation*}
X_\gamma^*=\{\hat h: h \in H\},
\end{equation*}
where every $\hat h\in X_\gamma^*$ is a centered Gaussian random variable with variance $\|\hat h\|^2_{L^2(X, \gamma)}=\|h\|^2_H$.
\end{corollary}


\bigskip
\textbf{The Cameron-Martin Theorem.}
\smallskip

The Cameron-Martin Theorem is a  fundamental theorem stating that the Cameron-Martin space characterises precisely those directions in $X$ in which translations leave the measure $\gamma$  ‘quasi-invariant'.

Let us start by observing what happens in finite dimension, that is when $X=\mathbb{R}^d$. Let $\gamma$ be a Gaussian measure on $\mathbb{R}^d$ with mean $a$ and covariance matrix $Q$. If  $\gamma$ is nondegenerate, namely $Q$ is invertible, the measures $\gamma_h$, with $h \in Q(\mathbb{R}^d)$, are all equivalent to $\gamma$, meaning that $\gamma_h \ll \gamma$ and $\gamma \ll \gamma_h$. In fact, a simple computation shows that $\gamma_h=\rho_h \gamma$, where
\begin{equation*}
\rho_h(x)=e^{\langle Q^{-1}h, x\rangle_{\mathbb{R}^d}-\frac 12 \|h\|^2_{Q(\mathbb{R}^d)}},
\end{equation*}
where $\|h\|^2_{Q(\mathbb{R}^d)}= \langle Q^{-1}h, h\rangle_{\mathbb{R}^d}$. In this case the Cameron-Martin space is the range of $Q$, i.e. $H=Q(\mathbb{R}^d)$. 
This is an immediate consequence of Proposition \ref{CM_space}, since $X_\gamma^*=\mathbb{R}^d$ and $R_\gamma$ is nothing but $Q$. According to the notation introduced in Proposition \ref{CM_space}, $h =R_\gamma \hat h$ iff $\hat h(x)=\langle Q^{-1}h,x\rangle_{\mathbb{R}^d}$, so we can write the density $\rho_h$ as 
\begin{equation}
\label{density}
\rho_h(x)=e^{\hat h(x)-\frac 12 \|h\|^2_H}.
\end{equation}
Similarly, in the infinite dimensional case the admissible shifts are the one belonging to the Cameron-Martin space and the density is of the form \eqref{density}.

\begin{theorem}
\label{CM_thm}
(Cameron-Martin Theorem).
If $h \in H$ the measure $\gamma_h$ defined in \eqref{shift_gamma} is equivalent to $\gamma$ and $\gamma_h=\rho_h \gamma$ with 
\begin{equation}
\rho_h(x)=e^{\hat h(x)-\frac 12 \|h\|^2_{H}},
\end{equation}
where $\hat h=R_\gamma^{-1}h$. If $h \notin H$, then $\gamma_h \perp h$. Hence $\gamma_h$ and $\gamma$ are equivalent if and only if $h \in H$.
\end{theorem}
\begin{proof}
See \cite[Theorem 2.4.5]{Bogachev}.
\end{proof}

\subsubsection{Differentiable functions and Sobolev spaces}
\label{diff_Banach_sec}

\begin{definition}
\label{gradient}
Let $\bar x \in X$ and let $\mathcal{O}$ be a neighborhood of $\bar x$. A function $f:\mathcal{O} \rightarrow \mathbb{R}$ is called (Frech\'et) differentiable at $\bar x$ if there exists $l \in X^*$
 such that 
\begin{equation*}
\lim_{\|h\|_X \rightarrow 0}\frac{|f(\bar x+ h)- f(\bar x)-l(h)|}{\|h\|_X}= 0.
\end{equation*}
In this case, $l$ is unique, we set $Df(\bar x):=l$ and call it Frechet derivative of $f$ at $\bar x$.
\\
A function that is differentiable at any point in $X$ is said to be $C^1$ if the function $Df:X \rightarrow X^*$, $\bar x \mapsto Df(\bar x)$ is continuous. 
We denote by $C^1_b(X)$ the set of all bounded continuously differentiable functions.
\end{definition}
Severeal properties of differentiable functions may be proved as in the case $X=\mathbb{R}^d$. If $f$ is differentiable at $\bar x$, then it is continuous at $\bar x$. Moreover, for every $v \in X$, the derivative along $v$, given by
\begin{equation*}
\frac{\partial f}{\partial v}(\bar x):= \lim_{t \rightarrow 0} \frac{f(\bar x +tv)-f(\bar x)}{t},
\end{equation*}
exists and is equal to $Df(\bar x)(v)$.

\begin{definition}
\label{H_gradient}
A function $f:X \rightarrow \mathbb{R}$ is called $H$-differentiable at $\bar x\in X$ if there exists $l_0 \in H^*$ such that 
\begin{equation*}
\lim_{\|h\|_H \rightarrow 0}\frac{|f(\bar x+ h)- f(\bar x)-l_0(h)|}{\|h\|_H}= 0.
\end{equation*}
$l_0$ is called $H$-derivative of $f$ at $\bar x$. 
\end{definition}
There exists a unique $y \in H$ such that $l_0(h)=\langle h,y\rangle_H$ for every $h \in H$. We set
\begin{equation*}
\nabla_H f(\bar x):=y.
\end{equation*} 
If $f$ is $H$-differentiable at $\bar x$, the directional derivative exists for any $h \in H$ and it is given by 
\[
\frac{\partial f(\bar x)}{\partial h}=\langle \nabla_H f(\bar x), h\rangle_H.
\]
\\
Notice the difference between Definitions \ref{gradient} and \ref{H_gradient}: in the first case the increments are taken in $X$, in the second case in $H$. The relation between the two notions of differentiability is provided in the next result.
\begin{proposition}
\label{relation}
If $f$ is differentiable at $\bar x$, then it is $H$-differentiable at $\bar x$, with $H$-derivative given by $h \mapsto Df(\bar x)(h)$ for every $h \in H$. Moreover, we have 
\begin{equation}
\label{nablaRgamma}
\nabla_H f(\bar x)=R_\gamma Df(\bar x).
\end{equation}
\end{proposition}
\begin{proof}
Set $l:= Df(\bar x)$. We have 
\begin{equation*}
\lim_{\|h\|\rightarrow 0} \frac{|f(\bar x+h)-f(\bar x)-l(h)|}{\|h\|_H}
=\lim_{\|h\|\rightarrow 0} \frac{|f(\bar x+h)-f(\bar x)-l(h)|}{\|h\|_X}\frac{\|h\|_X}{\|h\|_H}
=0,
\end{equation*}
since the ratio $\frac{\|h\|_X}{\|h\|_H}$ is bounded by a constant independent of $h$, being $H$ is continuously embedded in $X$. Let us now prove \eqref{nablaRgamma}. Bearing in mind \eqref{R_gamma_def}, \eqref{equality_R_gamma} and Proposition \ref{CM_space}, for every $\varphi \in X^*$ and $h \in H$, we have 
\begin{equation*}
\varphi (h)= \varphi(R_\gamma \hat h)=R_\gamma \hat h(\varphi)=\langle \varphi, \hat h\rangle_{L^2(X, \gamma)}=\langle R_\gamma\varphi,h\rangle_H.
\end{equation*}
In particular, taking $\varphi =Df(\bar x)$ we obtain $\langle \nabla_Hf(\bar x), h\rangle_H=Df(\bar x)(h)=\langle R_\gamma Df(\bar x),h\rangle_H$ for any $h \in H$, thus $\nabla_Hf(\bar x)=R_\gamma D f(\bar x)$.
\end{proof}

The following result is an integration by parts formula for $C_b^1(X)$ functions. 
\begin{proposition}
For every $f \in C_b^1(X)$ and $h \in H$ we have
\begin{equation}
\label{i.b.p.Banach}
\int_X \langle \nabla_H f(x), h\rangle_H \, \gamma({\rm d}x)= \int_X f(x) \hat h(x) \, \gamma ({\rm d}x).
\end{equation}
\end{proposition}
\begin{proof}
By the Cameron-Martin Theorem \ref{CM_thm}, for every $t \in \mathbb{R}$ we have 
\begin{equation*}
\int_Xf(x+th)\,\gamma({\rm d}x)=\int_X f(x)e^{t\hat h(x)-\frac{t^2\|h\|^2_H}{2}}\, \gamma({\rm d}x).
\end{equation*}
Thus, for $0<|t|\le 1$,
\begin{equation*}
\int_X \frac{f(x+th)-f(x)}{t}\, \gamma({\rm d}x) = \int_X f(x)\frac{e^{t\hat h(x)-\frac{t^2\|h\|^2_H}{2}}-1}{t}\, \gamma({\rm d}x).
\end{equation*}
As $t \rightarrow 0$, the integral in the l.h.s. converges to $\int_X\langle \nabla_H f(x), h\rangle_H \, \gamma({\rm d}x)$, by the Dominated Convergence Theorem. For what concerns the r.h.s. we have that
\begin{equation*}
\frac{e^{t\hat h(x)-\frac{t^2\|h\|^2_H}{2}}-1}{t} \rightarrow \hat h(x),
\end{equation*}
for every $x \in X$ as $t \rightarrow 0$. In fact, using the inequality $|(e^{ta}-1)/t|\le |a|e^{|a|}$ for $a \in \mathbb{R}$, we estimate for $0<|t| \le 1$
\begin{equation*}
\left|\frac{e^{t\hat h(x)-\frac{t^2\|h\|^2_H}{2}}-1}{t} \right|
=\left|\frac{e^{-\frac{t^2\|h\|^2_H}{2}}\left(e^{t \hat h (x)}-1\right)}{t}+ \frac{e^{-\frac{t^2\|h\|^2_H}{2}}-1}{t} \right|
\le|\hat h(x)|e^{\left| \hat h(x)\right|} + \sup_{0 < t\le 1}\left| \frac{e^{-\frac{t^2\|h\|^2_H}{2}}-1}{t} \right|,
\end{equation*}
where the function $x \mapsto  \left| \hat h(x)e^{\left| \hat h(x)\right|}\right|$ belongs to $L^1(X, \gamma)$ since $\hat h$ is a Gaussian random variable. Therefore, applying the Dominated Convergence Theorem the thesis follows.
 \end{proof}

The proof of the previous result relies on the fact that the Gaussian measure is quasi-invariant for translations only by moving in the Cameron-Martin directions. The Cameron-Martin space then becomes the natural choice of Hilbert space to be considered for obtaining integration by parts formulae with respect to Gaussian measures.

Using the integration by parts formula \eqref{i.b.p.Banach} one can prove the following result (see e.g. \cite[Proposition 9.3.7]{Lunardi}).
\begin{proposition}
 The operator $\nabla_H:C_b^1(X) \rightarrow L^p(X,\gamma;H)$ is closable as an unbounded operator from $L^p(X, \gamma)$ to $L^p(X, \gamma;H)$, for any $1\le p<\infty$.
\end{proposition}

For any $1\le p<\infty$ the Sobolev space $W^{1,p}(X, \gamma)$ is defined as the domain of the closure of the operator $\nabla_H$, still denoted by $\nabla_H$. $W^{1,p}(X, \gamma)$ is a Banach space with the norm
\begin{equation}
\label{norm_D_H}
\|f\|_{L^p(X, \gamma)}+\|\nabla_Hf\|_{L^p(X,\gamma; H)}.
\end{equation}
Moreover, the integration by parts formula \eqref{i.b.p.Banach} holds for any $\varphi \in W^{1,p}(X, \gamma)$ and $h \in H$ (see \cite[Proposition 9.3.10]{Lunardi}).

\subsubsection{Bogachev's Malliavin derivative $\nabla_H$}
\label{Bogachev}
The operator $\nabla_H$ is the \textit{Malliavin derivative} considered in \cite{Bogachev}, see also \cite{Lunardi}. It can be interpreted as a particular example of the general notion of Malliavin derivative seen in Section \ref{Malliavin_sec}. Referring to Section \ref{Malliavin_sec}, we need to identify the Gaussian Hilbert space and to characterize it in terms of a separable Hilbert space and a unitary operator. 
In \cite{Bogachev} the reference probability space is $(\Omega, \mathcal{A}, \mathbb{P})=(X, \mathcal{B}(X), \gamma)$ where $X$ is a separable Banach space \begin{footnote}{Actually, \cite{Bogachev} considers the case of a locally convex topological space. We consider here the case of Banach spaces for the sake of exposition (see also \cite{Lunardi}).}\end{footnote} and $\gamma$ a Gaussian measure on it: the framework is the one described in Section \ref{Gaus_meas_Ban}. The Gaussian Hilbert space $\mathcal{H}_1$ is  $X_\gamma^*$, the space $\mathcal{H}$ is the Cameron-Martin space $H$ and  the unitary operator $W$ is the operator $\hat \cdot=R_\gamma^{-1}$. 
\\
With these identifications, by comparing the integration by parts formula derived in \cite{Bogachev}:
\begin{equation*}
\int_X \langle \nabla_H \varphi(x), h\rangle_H \, \gamma({\rm d}x)= \int_X \varphi(x) \hat h(x) \, \gamma ({\rm d}x),  \qquad \forall \varphi \in W^{1,2}(X, \gamma)=\text{Dom}(\nabla_H), \ \forall h \in H
\end{equation*}
with the one we obtained in Proposition \ref{i.b.p.}:
\begin{equation*}
\mathbb{E}\left[ \langle D\varphi, h \rangle_{\mathcal{H}}\right]=\mathbb{E}\left [W(h)\varphi\right], \qquad \forall \varphi \in \mathbb{D}^{1,2}=\text{Dom}(D), \ \forall h \in \mathcal{H},
\end{equation*}
we immediately see that the gradient operator $\nabla_H$ of  \cite{Bogachev} is in fact a Malliavin derivative operator in the sense of Section \ref{Malliavin_sec}.

\subsection{Malliavin derivative in the sense of Da Prato}
\label{Gaussian_Hilbert_sec}
For this part we mainly refer to \cite{CDP96} and \cite{DaPrato}.
\\
Throughout this Subsection by $X$ we denote an infinite dimensional separable Hilbert space with inner product $\langle \cdot, \cdot\rangle_X$. As usual, we identify $X$ with $X^*$ via the Riesz representation.  Of course, all the results of Section \ref{Gaussian_Banach_sec} apply in the case of Hilbert spaces. However, the Hilbert space structure allows for further characterizations of a variety of objects in play.

\subsubsection{Gaussian measures in Hilbert spaces}
\label{Gau_Hil_sec}

Let $\gamma$ be a Gaussian measure on $(X, \mathcal{B}(X))$. According to Proposition \ref{CF_gamma_prop} 
 the characteristic function of $\gamma$ is
\begin{equation*}
\widehat \gamma(f)= e^{ia_\gamma(f)- \frac 12 K_\gamma(f,f)}, \quad f \in X^*,
\end{equation*}
where $a_\gamma:X^* \rightarrow \mathbb{R}$ is a linear continuous map and $K_\gamma:X^*\times X^*\rightarrow \mathbb{R}$ is a bilinear symmetric and continuous map.
The identification of $X$ with $X^*$ yields the existence of a vector $a \in X$ and a self-adjoint operator $Q \in \mathcal{L}(X)$ such that $a_\gamma(f)=\langle f, a \rangle_X$ and $K_\gamma (f,g)=\langle Qf,g\rangle_X$ for every $f,g \in X^*=X$.\begin{footnote}{Recall that if $X$ is a real Hilbert space, and $B:X \times X \rightarrow \mathbb{R}$ is bilinear, symmetric
and continuous, then there exists a unique self-adjoint operator $Q \in \mathcal{L}(X)$ such that
$B(x,y)=\langle Qx, y\rangle_X$, for every $x,y \in X$.}\end{footnote} 
As usual, we identify the linear map $x \mapsto \langle x, f\rangle_X$ with the element $f \in X$. Thus 
\begin{equation}
\label{mean_Hilbert}
a_\gamma(f)=\langle a, f \rangle_X= \int_X \langle x, f\rangle\, \gamma({\rm d}x), \quad f \in X,
\end{equation}
\begin{equation}
\label{Q_def}
K_\gamma(f,g)=\langle Qf,g \rangle_X= \int_X \langle f,x-a\rangle\langle g,x-a\rangle\\, \gamma({\rm d}x), \quad f,g \in X,
\end{equation}
and 
\begin{equation}
\label{CF_gamma_Hilbert}
\widehat \gamma (f)
=e^{i \langle f, a\rangle_X - \frac 12\langle Qf,f,\rangle_X}, \quad f \in X.
\end{equation}
After such identifications, the element $a \in X$ is called the \textit{mean} of $\gamma$ and the self-adjoint operator $Q \in \mathcal{L}(X)$ is called the \textit{covariance} of $\gamma$.

The following result is analogous to  Proposition \ref{CF_gamma_prop} but with an important difference. In Proposition \ref{CF_gamma_prop} a measure is given and we have a criterium to see whether it is Gaussian. Proposition \ref{CF_gamma_Hilbert_prop} provides instead a characterization of all Gaussian measures in $X$ (see \cite[Theorem 2.3.1]{Bogachev} for the proof).

\begin{proposition}
\label{CF_gamma_Hilbert_prop}
If $\gamma$ is a Gaussian measure on $X$ then its characteristic function is given by \eqref{CF_gamma_Hilbert}, where $a \in X$ and $Q$ is a nonnegative self-adjoint trace-class operator. 
Conversely, for every $a \in X$ and for every $Q$ nonnegative self-adjoint trace-class operator, the function $\widehat \gamma$ in \eqref{CF_gamma_Hilbert} is the characteristic function of a Gaussian measure with mean $a$ and covariance operator $Q$.
\end{proposition}

From now on by $\gamma\sim \mathcal{N}(0,Q)$ we will denote a centered Gaussian measure with covariance $Q$. We call $\gamma$ non degenerate if Ker$(Q)=\{0\}$.

\subsubsection{A characterization of the The Cameron-Martin space}

The following result provides a characterization of the Cameron-Martin space associated to the Gaussian measure $\gamma \sim \mathcal{N}(0,Q)$; see \cite[Theorem 4.2.7]{Lunardi} 
(see also \cite[Remark 3.3.8]{Bogachev}, \cite[Chapter III.1.2, Lemma 1.2]{vtc} 
 for an extension to the Banach setting).
We fix an orthonormal basis $\{e_k\}_{k \in \mathbb{N}}$ given by eigenvectors of $Q$ such that $Qe_k=\lambda_k e_k$ for any $k \in \mathbb{N}$ (see Proposition \ref{spectral}). For any $x \in X$, $k \in \mathbb{N}$, we set $x_k :=\langle x, e_k\rangle_X$. 

\begin{theorem}\label{Hilbert_varie}Let $\gamma$ be a Gaussian measure with covariance operator $Q$. The Cameron-Martin space is the range of $Q^{1/2}$, that is \begin{equation*}H=Q^{1/2}(X)=\left\{x \in X: \ \sum_{k \in \mathbb{N},\lambda_k \ne 0} \frac{x_k}{\sqrt{\lambda_k}}< \infty\right\},\end{equation*}and the two spaces have the same inner product and norm, that is
 \begin{equation}
\label{eq_scal_pro}
\langle h, l\rangle_H = \sum_{k \in \mathbb{N},\lambda_k \ne 0} \frac{h_kl_k}{\lambda_k}=\langle h, l\rangle_{Q^{1/2}(X)}=\langle Q^{-1/2}h, Q^{-1/2}l\rangle_X, \qquad \|h\|_H=\|Q^{-\frac12}h\|_X,\end{equation}for all $h, l \in H$. 
\\
Moreover, let $h=Q^{1/2}x$ for some $x\in X$. If \emph{Ker}$(Q)=\{0\}$, then \begin{equation}\label{iso_Q^12}\|h\|_H=\|x\|_X.
\end{equation}
\end{theorem}

Notice that \eqref{iso_Q^12} follows from \eqref{eq_scal_pro}, recalling that the pseudo inverse $Q^{-1/2}$ coincides with the inverse of $Q^{1/2}$ if Ker$(Q)=\{0\} \iff \gamma$ is non-degenerate, in virtue of Proposition \ref{inj}.




\subsubsection{The Gaussian Hilbert space $X^*_\gamma$}

When $X$ is an Hilbert space one can, obviously, introduce the Gaussian Hilbert space $X_\gamma^*$ as in Definition \ref{Xstargammadef} and characterize it as in Corollary \ref{char_X^*_gam-H}, that is
\begin{equation*}
X_\gamma^*=\{\hat h: h \in H=Q^{1/2}(X)\}.
\end{equation*}
The operator $R_\gamma$ can be introduced as in \eqref{R_gamma_def}.
Observe that while the covariance operator $Q$ is defined in $X$, the operator $R_\gamma$ is defined in $X_\gamma^*\subset L^2(X, \gamma)$. Recalling \eqref{R_gamma_def} and \eqref{equality_R_gamma}, for any $f,g \in X^*$ we have
\begin{equation*}
R_\gamma(jf)(g)=\int_X (f(x)-a_\gamma(f))(g(x)-a_\gamma g)\, \gamma({\rm d}x)=g(R_\gamma(jf)
\end{equation*}
and keeping in mind \eqref{mean_Hilbert} and \eqref{Q_def}, we get 
\begin{equation*}
\langle g, R_\gamma(jf)\rangle =\langle Qf,g\rangle_X,
\end{equation*}
so that 
\begin{equation*}
R_\gamma(jf)=Qf,
\end{equation*}
that is, the natural extension of $Q$ to $X_\gamma^*$ coincides with $R_\gamma$.


From now on we will restrict to the case of a \emph{non degenerate} Gaussian measure $\gamma$: in this case the Cameron-Martin space $H=Q^{1/2}(X)$ is dense in $X$ in virtue of Proposition \ref{inj}.






 \begin{proposition}
 \label{ext_prop}
 Let $\gamma \sim \mathcal{N}(0,Q)$ be a non degenerate gaussian measure. The map $\hat \cdot =R_\gamma^{-1}:H=Q^{1/2}(X) \rightarrow X_\gamma^*\subset L^2(X, \gamma)$ can be uniquely extended to a linear isometry $\mathcal{W}_{\bullet}$ defined on the whole $X$ with range $X^*_\gamma$. 
 \end{proposition}
 \begin{proof}
 From Corollary \ref{char_X^*_gam-H} we know that $\hat \cdot=R_\gamma^{-1}$ defines a linear isometry from $H$ onto $X_\gamma^* \subset L^2(X, \gamma)$. Since $H=Q^{1/2}(X)$ is dense in $X$ the map $\hat \cdot=R_\gamma^{-1}$ can be uniquely extended to a linear isometry $\mathcal{W}_{\bullet}$ defined on the whole $X$ with range $X^*_\gamma$.
 \end{proof}
The map $\mathcal{W}_\bullet$ is usually called \emph{white noise map}.
As a consequence of the above result we obtain a different characterization of the Gaussian Hilbert space $X_\gamma^*$.
\begin{corollary}
\label{corX_gam_Hil}
The map $\mathcal{W}_{\bullet}:X \rightarrow X^*_\gamma$ is a unitary operator and \begin{equation*}X_\gamma^*=\{\mathcal{W}_{z}: z \in X\},\end{equation*}where every  $\mathcal{W}_z$ is a centered Gaussian random variable with variance $\|\mathcal{W}_z\|^2_{L^2(X, \gamma)}=\|z\|^2_X$.
\end{corollary}

\subsubsection{Differentiable functions and Sobolev spaces}
The notions of Frech\'et and $H$-Frech\'et differentiability of Section \ref{diff_Banach_sec} still apply when $X$ is Hilbert.
We emphasize that if $f$ is differentiable at $\bar x$, the gradient and the $H$-gradient of $f$ at $\bar x$ do not coincide in general. Identifying $X^*$ with $X$, Proposition \ref{relation} yields 
\begin{equation}
\label{nablaQ}
\nabla_Hf(\bar x)=Q\nabla f(\bar x).
\end{equation}
Moreover, we have the following result.

  \begin{lemma}
For any $\varphi \in C^1(X)$, 
 \begin{equation}
 \label{equivalence_H_X}
 \langle \nabla_H\varphi(x), h \rangle_H= \langle Q^{1/2}\nabla \varphi(x), z\rangle_X, \quad \forall z \in X, \ \forall \ h   \in H \ \text{with} \ h=Q^{1/2} z.
  \end{equation}
 In particular,
 \begin{equation}
\label{equality_norm}
 \|\nabla _H\varphi(x)\|^2_H =\|Q^{1/2}\nabla \varphi(x)\|^2_X. 
 \end{equation}
\end{lemma}   
 \begin{proof}

 Let $h\in H$ with $h=Q^{1/2}z$ for $z \in X$.
Bearing in mind \eqref{eq_scal_pro}  and \eqref{nablaQ}, for any $\varphi \in C^1(X)$  we infer 
   \begin{equation*}
 \langle \nabla_H\varphi(x), h \rangle_H=\langle Q^{-1/2}\nabla_H\varphi(x), Q^{-1/2}h\rangle_X=\langle Q^{-1/2}Q\nabla\varphi(x), Q^{-1/2}Q^{1/2}z\rangle_X= \langle Q^{1/2}\nabla \varphi(x), z\rangle_X,
  \end{equation*}  
where in the last equality we used the fact that the pseudo inverse $Q^{-1/2}$ coincides with the inverse of $Q^{1/2}$ being Ker$(Q)=\{0\}$ since the measure is nondegenerate.
 \end{proof}

As a consequence, the integration by parts formula \eqref{i.b.p.Banach} can be equivalently written as
\begin{equation}
\label{ibpnew}
\int_X \langle  Q^{1/2}\nabla \varphi (x), z\rangle_X\, \gamma({\rm d}x) = \int_X \varphi (x) \mathcal{W}_z(x)\,\gamma({\rm d}x), \qquad z \in X, \ \varphi \in C_b^1(X),
\end{equation}
since for $h=Q^{1/2} z$, $\hat h=\mathcal{W}_z$.
Exploiting \eqref{ibpnew} one proves  the following result (see \cite{DaPrato} for a proof).
 
\begin{proposition}
 The operator $Q^{1/2}\nabla:C_b^1(X) \rightarrow L^p(X,\gamma;X)$ is closable as an unbounded operator from $L^p(X, \gamma)$ to $L^p(X, \gamma;X)$, for any $1\le p<\infty$.
\end{proposition}

For any $1\le p<\infty$ the Sobolev space ${W}_{Q^{1/2}}^{1,p}(X, \gamma)$ is defined as the domain of the closure of the operator $Q^{1/2} \nabla$, denoted by $M$. It is a Banach space with the norm
\begin{equation}
\label{norm_M}
\|f\|_{L^p(X, \gamma)}+\|Q^{1/2}\nabla f\|_{L^p(X, \gamma; X)}.
\end{equation}
Moreover, the integration by parts formula \eqref{ibpnew} holds for any $\varphi \in {W}^{1,p}_{Q^{1/2}}(X, \gamma)$ and $z \in X$.

\subsubsection{Da Prato's Malliavin derivative $Q^{1/2}\nabla$}
\label{DaPrato}

The operator $M:=Q^{1/2}\nabla$ is the \textit{Malliavin derivative} considered in \cite{DaPrato}. It can be interpreted as a particular example of the general notion of Malliavin derivative seen in Section \ref{Malliavin_sec}. Referring to Section \ref{Malliavin_sec}, we need to identify the Gaussian Hilbert space and to characterize it in terms of a separable Hilbert space and a unitary operator. 
In \cite{DaPrato} the reference probability space is $(\Omega, \mathcal{A}, \mathbb{P})=(X, \mathcal{B}(X), \gamma)$ where $X$ is a separable Hilbert space and $\gamma=\mathcal{N}(0,Q)$ a centered non degenerate Gaussian measure on it: the framework is the one described in Section \ref{Gau_Hil_sec}. The Gaussian Hilbert space $\mathcal{H}_1$ is $X_\gamma^*$, the space $\mathcal{H}$ is $X$ and  the unitary operator $W$ is the map $\mathcal{W}_\bullet$. 
\\
\\
With these identifications, by comparing the integration by parts formula derived in \cite{DaPrato}:
\begin{equation*}
\int_X \langle Q^{1/2}\nabla \varphi (x), z\rangle_X\, \gamma({\rm d}x) = \int_X \varphi (x) \mathcal{W}_z(x)\,\gamma({\rm d}x),\qquad \forall \varphi \in W_{Q^{1/2}}^{1,2}(X, \gamma)=\text{Dom}(Q^{1/2}\nabla), \ \forall z \in X,
\end{equation*}
with the one we obtained in Proposition \ref{i.b.p.}:
\begin{equation*}
\mathbb{E}\left[ \langle D\varphi, h \rangle_{\mathcal{H}}\right]=\mathbb{E}\left [W(h)\varphi\right], \qquad \forall \varphi \in \mathbb{D}^{1,2}=\text{Dom}(D), \ \forall h \in \mathcal{H},
\end{equation*}
we immediately see that the gradient operator $M$ of \cite{DaPrato} is in fact a Malliavin derivative operator in the sense of Section \ref{Malliavin_sec}.




\subsection{Comparison of the Malliavin derivatives in the sense of Bogachev and Da Prato}
\label{comparison_sec}


If $X$ is a separable Hilbert space endowed with a centered nondegenerate Gaussian measure with covariance operator $Q$, we can compare the Malliavin derivative in the sense of Bogachev with the Malliavin derivative in the sense of Da Prato.

The Malliavin derivatives $\nabla_H$ and $M$ of Sections \ref{Bogachev} and \ref{DaPrato} are different. Indeed, \eqref{nablaQ} yields the relation 
\[
\nabla_H=Q^{\frac 12}M.
\] 
On the other hand, the domain of the two derivatives is the same, that is 
\[W^{1,p}={W}_{Q^{1/2}}^{1,p}\]
 since thanks to \eqref{equality_norm}, the closure of the space $C_b^1(X)$ with respect to the norm \eqref{norm_D_H} is the same as its closure with respect to the norm \eqref{norm_M}.
 
This should not be surprising in light of the general results of Section \ref{Malliavin_sec}: in Sections \ref{Bogachev} and \ref{DaPrato} the reference Gaussian Hilbert space $\mathcal{H}_1$ is the same, that is $X_\gamma^*$; thus Proposition \ref{charD12} ensures the two Malliavin derivatives $\nabla_H$ and $M$ to have the same domain.
What changes in \cite{DaPrato} and \cite{Bogachev} is how the space $\mathcal{H}_1$ is characterized.  In \cite{Bogachev} is considered the unitary operator $\hat \cdot =R_\gamma^{-1}$ between $H$ and $\mathcal{H}_1$, whereas in  \cite{DaPrato} is considered the unitary operator $\mathcal{W}_\bullet$ between $X$ and $\mathcal{H}_1$. This naturally leads to different Malliavin derivatives, having chosen different Hilbert spaces $\mathcal{H}$ and unitary operators $W$.

The following diagram (compare with \eqref{abstract_diagram}) summarizes the results of this Section.

\bigskip
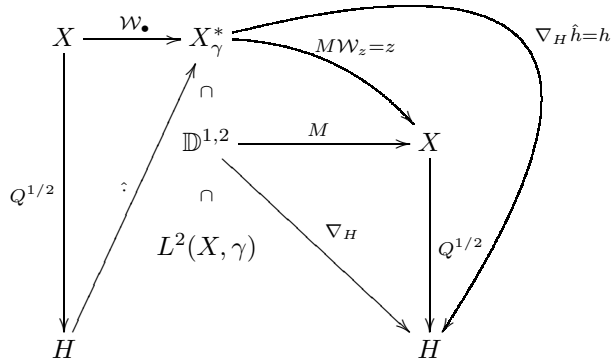
\begin{equation*}
\xymatrix{
X\ar^{\mathcal{W}_{\bullet}}[r] \ar_{Q^{1/2}}[ddd] &X_\gamma^*
\ar@/ ^1pc/^{M\mathcal{W}_z=z}[drr] \ar@/ ^8pc/^{\nabla_H\hat h=h}[dddrr]  
\\
&\mathbb{D}^{1,2}\ar@{} [u] |{\cap} \ar^M[rr] \ar^{\nabla_H}[ddrr] && X \ar^{Q^{1/2}}[dd]
\\
 &L^2(X, \gamma)\ar@{} [u] |{\cap}
  \\
H \ar^{\hat \cdot}[uuur] &&& H
} 
\end{equation*}
\bigskip

The two Malliavin derivatives $\nabla_H$ and $M$ are different but with the same domain:
\begin{equation*}
\label{abstract_relation}
\nabla_H=Q^{1/2} M, \qquad \mathbb{D}^{1,2}=W^{1,2}={W}_{Q^{1/2}}^{1,2} \subset L^2(X, \gamma).
\end{equation*}
In particular, it holds 
\begin{equation*}
\nabla_H\hat h=h, \ \forall h \in H \quad \text{and} \quad M\mathcal{W}_z=z, \ \forall z \in X.
\end{equation*}

To conclude, the following diagram emphasizes that the Bogachev and Da Prato Malliavin derivatives are just two, among the infinitely many derivatives $D$, that one can construct in that framework. Given the Gaussian Hilbert space $X_\gamma^*$ one can consider infinitely many separable Hilbert space $\mathcal{H}$ to put in unitary correspondence with $X_\gamma^*$ through unitary operators $W$. This leads to the construction of infinitely many \textit{different} Malliavin derivative operators $D$ having the \textit{same} Sobolev space $\mathbb{D}^{1,2}$ as their domain.
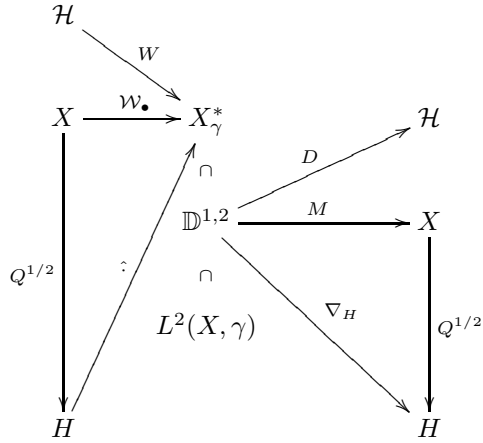
\begin{equation*}
\xymatrix{
\mathcal{H} \ar^W[dr]
\\
X \ar^{\mathcal{W}_{\bullet}}[r] \ar_{Q^{1/2}}[ddd] &X_\gamma^*&&\mathcal{H}
\\
&\mathbb{D}^{1,2} \ar^{D}[rru]\ar@{} [u] |{\cap} \ar^{M}[rr] \ar^{\nabla_H}[ddrr] && X \ar^{Q^{1/2}}[dd]
\\
 &L^2(X, \gamma)\ar@{} [u] |{\cap}
  \\
H \ar^{\hat \cdot}[uuur] &&& H
}
\end{equation*}

\subsection{Final remarks}
We recalled above the construction of the Malliavin derivatives that one can find in the books \cite{Bogachev} and \cite{DaPrato}. This construction highlights the fact that, when working on a separable Banach or Hilbert space (that is a space with a linear topological structure), the Malliavin derivative has the usual interpretation as a gradient operator. Moreover, in this construction it is clear the role played by the Cameron-Martin space.

On the other hand, to introduce the Malliavin derivatives in the sense of Bogachev and Da Prato it would be sufficient to stop at the construction of the Gaussian Hilbert space $\mathcal{H}_1=X_\gamma^*$ and the separable Hilbert spaces $\mathcal{H}$ in correspondence with it through the unitary operators $W$. From here on, the abstract construction of Section \ref{Malliavin_sec} comes into play: the Malliavin derivative is nothing but the extension of the operator $D:=W^{-1}$ from $\mathcal{H}_1$ to $\mathbb{D}^{1,p}$. In particular, this abstract construction, requiring no topological linear assumptions on the space $(\Omega, \mathcal{F}, \mathbb{P})$, does not exploit the density of $C_b^1$ functions (since it makes no sense talking about continuity and differentiability in absence of a linear topological  structure) but of the polynomial functions $\mathcal{P}(\mathcal{H}_1)$.
 

\newpage
\appendix

\section{Some recalls on operator theory}

\subsection{Closed and closable operators}
\label{closable_op_sec}
In this Section we recall the definitions of closed and closable operators. These are important classes of operators that cover almost all interesting unbounded operators occurring in applications. For more details see \cite{Brezis} and \cite{Schmu}.

Let $E$ and $F$ be Banach spaces with the norms $\|\cdot\|_E$ and $\|\cdot \|_F$, respectively. We consider a linear operator $T:E \rightarrow F$.

\begin{definition}
 A linear operator $T$ is called \emph{bounded} if 
 there exists a constant $c\ge 0$ such that 
$
\|Tx\|_F \le c\|x\|_E,\ \forall x \in E.
$
\end{definition}
A linear operator $T$ is bounded iff it is \emph{continuous} that is, if $\lim_{n \rightarrow \infty}x_n=x$ in $E$ for $\{x_n\}_n, x \in E$, then $\lim_{n \rightarrow \infty}Tx_n=Tx$ in $F$.

The space of all bounded linear operators from $E$ to $F$ is denoted by $\mathcal{L}(E,F)$; for simplicity we write $\mathcal{L}(E)$ instead of $\mathcal{L}(E,E)$.
The norm of a bounded linear operator is defined as 
\begin{equation*}
\|T\|_{\mathcal{L}(E,F)}:= \sup_{x \ne 0}\frac{\|Tx\|_F}{\|x\|_E}.
\end{equation*}
\begin{proposition}
\label{lin_Ban}
The space $(\mathcal{L}(E,F), \|\cdot\|_{\mathcal{L}(E,F)})$ is a Banach space.
\end{proposition}

In many applications the linear operators one encounters are not bounded (hence not continuous). As a simple example think of the multiplication by $x$ on the Hilbert space $L^2(\mathbb{R})$. Clearly $\int_\mathbb{R} |xf(x)|^2\, {\rm d}x$ is not finite for all functions in $L^2(\mathbb{R})$, take e.g. $\frac{1}{\sqrt{1+x^2}}$.
One thus introduce the notion of \emph{unbounded} linear operators. As made clear by the previous example, in general, we cannot talk of unbounded everywhere defined operators and we have to restrict the domain of definition. We have a linear operator 
$$T:\text{Dom}(T) \subset E \rightarrow F,$$
where Dom$(T)$ is a linear subspace of $E$, the \emph{domain} of the operator $T$.
\begin{definition}
 An \emph{unbounded linear operator} from $E$ into $F$ is a linear map $T:\emph{Dom}(T)\subset E \rightarrow F$ defined on a linear subspace $\emph{Dom}(T) \subset E$ with values in $F$. 
 \end{definition}
An operator $S$ is an \emph{extension} of $T$ if Dom$(T)\subset $ Dom$(S)$ and $Tu=Su$ for all $u\in $ Dom$(T)$. We write $T \subset S$. Note that the domain becomes part of the definition of the operator. The same formal expression considered on different domains may lead to operators with different properties. Think for example to the Laplace operator considered with different boundary conditions.

In general, linear operators can be wild unless we impose some kind of continuity.
\begin{definition}
A linear operator $T:\emph{Dom}(T) \rightarrow F$ is \emph{closed} if for any sequence $\{x_n\}_n \subset \emph{Dom}(T)$ such that $\lim_{n \rightarrow \infty} x_n=x$ in $E$ and $\lim_{n \rightarrow \infty}T x_n=y$ in $F$, it follows that $x \in \emph{Dom}(T)$ and $Tx=y$.
\end{definition}

Bounded linear operators are obviously closed, in fact the convergence $x_n \rightarrow x$ in $E$ entails the convergence $Tx_n \rightarrow Tx$ in $F$.

 Another way of saying that an operator is closed is the following
 \begin{proposition}
 A linear operator $T:\emph{Dom}(T) \rightarrow F$ is \emph{closed} iff the domain \emph{Dom}$(T)$ endowed with the \emph{graph norm} $\|x\|_T:=\sqrt{\|x\|^2_E+\|Tx\|^2_F}$ is a Banach space.
 \end{proposition}

In practice it often occurs that an unbounded linear operator $T$ is not closed.
In fact, it may happen that $\text{Dom}(T)\ni x_n \rightarrow x \in E$ but $Tx_n$ has no limit. 
Moreover, it could happen that $\text{Dom}\ni x_n \rightarrow x$, $\text{Dom}(T) \ni \widetilde{x}_n \rightarrow x$, but $T x_n$ and $T \widetilde{x}_n$ have different limits. Any of these possibilities prevent $T$ to be extended to all the limits points of $\text{Dom}(T)$. Nevertheless, $T$ can be extended when the following less problematic situation occurs: not along all sequences $\text{Dom}(T)\ni x_n \rightarrow x \in E$, $Tx_n$ has a limit, but for all sequences in $\text{Dom}(T)$ converging to $x$ along which $T$ has a limit, this
limit is unique.
This circumstance make $T$ a closable operator.

\begin{definition}
A linear operator $T:\emph{Dom}(T) \rightarrow F$ is \emph{closable} if it has a closed extension. 
\end{definition}
The following characterization of closability is very useful in the concrete situations.
\begin{proposition}
\label{def_closable}
A linear operator $T:\emph{Dom}(T) \rightarrow F$ is  \emph{closable} if and only if for any sequence $\{x_n\}_n \subset \emph{Dom}(T)$ such that $\lim_{n \rightarrow \infty} x_n=0$ in $E$ and $\lim_{n \rightarrow \infty}T x_n=y$ in $F$, it follows that $y=0$. 
\end{proposition}

\begin{proof}
Assume that $T$ is closable and denote by $S$ its closed extension. If $x_n\in$ Dom$(T)$, then $x_n\in$ Dom$(S)$. Since $Sx_n=Tx_n \rightarrow y$ and since $x_n \rightarrow 0$, it follows that $y
=0$ since $S$ is closed.
\\
Conversely, consider the linear subspace of $E$
\begin{equation*}
D=\left\{x \in E \ : \ \exists \ y \in F \ \text{s.t.  $\forall \ \{x_n\}_n \subset \text{Dom}(T)$ with $x_n \rightarrow x $ in $E$, $Tx_n \rightarrow y$ in $F$} \right\}
.
\end{equation*}
On $D$ we define $\overline Tx=y$. We have to show that $y$ is independent of the sequence $x_n$. Let $\{u_n\}_n \subset$ Dom$(T)$ be another sequence with $u_n \rightarrow x$ and $Tu_n \rightarrow z$. We have to show that $y=z$. Since $x_n -u_n \rightarrow 0$ and since $T(x_n-u_n) \rightarrow y-z$ it follows that $y=z$. Hence $\overline T$ is defined on $D$ and one easily check that it is a linear operator. It remains to show that $\overline T$ is closed. Let $x_n$ be a sequence in $D$ such that $x_n \rightarrow x$ and $\overline Tx_n \rightarrow y$. We have to show that $x \in D$ and $\overline T x=y$. Since, for each $n$, $x_n \in D$, there exists $u_n \in$ Dom$(T)$ such that 
\begin{equation*}
\|x_n-u_n\|_E+ \|\overline Tx_n-Tu_n\|_F < \frac 1 n.
\end{equation*}
Hence $u_n \rightarrow x$ and $T u_n \rightarrow y$, i.e. $x \in D$ and $y=\overline T x$. Thus $\overline T$ is closed and this concludes the proof.
\end{proof}

If $T$ is closable there is a natural candidate for its closure (see the proof of Proposition \ref{def_closable}).
\begin{definition}
\label{def_closure}
If $T$ is closable, the \emph{closure} of $T$ is the operator $\overline T$ whose domain and action are 
\begin{itemize}
\item $\emph{Dom}(\overline T):= \left\{x \in E \ : \ \exists \ y \in F \ \text{s.t.  $\forall \ \{x_n\}_n \subset \emph{Dom}(T)$ with $x_n \rightarrow x $ in $E$, $Tx_n \rightarrow y$ in F} \right\}$,
\item $\overline T x:=y$ for any $x\in \emph{Dom}(\overline T)$.
\end{itemize}
\end{definition}
One can easily check that the above definition is well posed since $y$ is uniquely identified by $x$ and $\overline T$ defines a linear operator. It is clear that $\text{Dom}(T) \subset \text{Dom}(\overline T)$ and $\overline Tx=Tx$ for all $x \in \text{Dom}(T)$. The closure $\overline T$ is the smallest closed extension of $T$ in the sense that if $T \subset S$ and $S$ is closed, then $\overline{T} \subset S$.

From Proposition \ref{def_closable} it is clear that if an operator is continuous, then it is clearly closable. Closability can be considered in fact as a weakening of continuity for linear operators. Let $\{x_n\}_n$ be a sequence in Dom$(T)$ such that $\lim_{n \rightarrow \infty}x_n=0$ in $E$. If the operator $T$ is continuous, then $\lim_{n \rightarrow \infty}Tx_n=0$ in $F$. For $T$ being closable, one requires only that \textit{if} the sequence $\{Tx_n\}_n$ converges in $F$, then it has to converge to the "correct limit", that is $\lim_{n \rightarrow \infty}Tx_n=0$.


\subsection{Trace class and Hilbert-Schmidt operators}
\label{HS_sec}

In this Section we recall the definitions of trace class and Hilbert-Schmidt operators and some of their properties that we will need throughout these lecture notes. For more details see \cite{Conway}. We start by recalling some basic definitions.

Let $H$ and $K$ be two separable Hilbert spaces with scalar products $\langle \cdot, \cdot\rangle_H$, $ \langle \cdot, \cdot\rangle_K$, respectively. The space of continuous linear operators from $H$ to $K$ is denoted by $\mathcal{L}(H,K)$, and we set for convenience $\mathcal{L}(H):=\mathcal{L}(H,H)$. 

\subsubsection*{Adjoint operators}
The adjoint $T^*\in \mathcal{L}(K,H)$ of the operator $T\in \mathcal{L}(H,K)$ is given by $\langle T^*x,y\rangle_H=\langle Ty,x\rangle_K$ for all $x \in K$, $y \in H$. 
\\
This definition extends to unbounded operators. If $T:\text{Dom}(T) \subset H \rightarrow K$ is a linear 
 operator which is densely defined, i.e. Dom$(T)$ is dense in $H$, the \emph{adjoint} operator $T^*:\text{Dom}(T^*)\subset K \rightarrow H$ has domain
\[
\text{Dom}(T^*)=\left\{x \in K \ : \ \exists\ C \ge 0 : \forall y \in \text{Dom}(T) \\\quad |\langle Ty, x\rangle_K|\le C\|y\|_H\right\},
\]
and, if $x \in \text{Dom}(T^*)$, then $T^*x$ is the element of $H$ characterized by
$
\langle Ty,x\rangle_K=\langle y, T^*x\rangle_H$, $\forall y \in \text{Dom}(T).
$
Notice that $\text{Dom}(T^*)$ is non empty since $0 \in \text{Dom}(T^*)$. $T^*$ is a linear closed operator, even if $T$ is not closed or closable. 

\subsubsection*{Self-adjoint nonnegative operators}
The operator $T\in \mathcal{L}(H)$ is \textit{self-adjoint} (or \textit{symmetric}) if $T=T^*$ i.e. $\langle Tx,y\rangle_H= \langle x, Ty\rangle_H$ for all $x,y \in H$ and it is \emph{nonnegative} (\emph{positive}) if $\langle Tx,x\rangle_H \ge 0 \ (>0)$ for all $x \in H$. 

For all $T \in \mathcal{L}(H)$ self-adjoint and nonnegative there exists a unique $T^{1/2}\in \mathcal{L}(H)$ self-adjoint and nonnegative  such that $T=(T^{1/2})^2$.

For the proof of the following result we refer to \cite[Chapter III.1.1, Proposition 1.3]{vtc}). 
\begin{proposition}
\label{inj}
Let $T \in \mathcal{L}(H)$ be a self-adjoint positive operator. Then the following properties of $T$ are equivalent:
\begin{itemize}
\item [i)] $T$ is injective, i.e. \emph{Ker}$T=\{0\}$. 
\item [ii)] The quadratic form $x \mapsto \langle x,Tx\rangle_H$ is non-degenerate (i.e. $\langle x,Tx\rangle_H=0$ implies $x=0$).
\item [iii)] The set $T(X)$ is dense in X.
\end{itemize}
\end{proposition}

\subsubsection*{Pseudo-inverse operators}
Let $T \in \mathcal{L}(H,K)$. We denote by $P$ the orthogonal projection on the kernel of $T$, which is a closed subspace of $H$. $H$ is the direct sum of $P(H)$ and $(I-P)(H)$, so the restriction of $T$ to $(I-P)(H)$ is one to one. We define the \textit{pseudo-inverse} $T^{-1}:T(H)\rightarrow H$ as follows: for every $y \in T(H)$, $T^{-1}(y)$ is the unique $x\in (I-P)(H)$ such that $Tx=y$.
\\
If $T \in \mathcal{L}(H)$ is self-adjoint and nonnegative, the range of $T$ is an Hilbert space with the inner product
\begin{equation}
\label{range_prop}
\langle x,y\rangle_{T(H)}:=\langle T^{-1}x, T^{-1}y\rangle_H, \qquad x, y \in T(H).
\end{equation}

\subsubsection{Trace class operators}
Recall that in finite dimension $n$, we can recover the trace of a matrix $T$ as
\begin{equation*}
\sum_{k=1}^n\langle T e_k, e_k\rangle,
\end{equation*}
where $\{e_1,...,e_n\}$ is an orthonormal basis of the finite dimensional space. We might hope to extend this definition to an operator $T \in \mathcal{L}(H)$, by setting 
\begin{equation*}
\text{Tr}(T)=\sum_{k=1}^\infty\langle Te_k, e_k\rangle_{H},
\end{equation*}
where $\{e_k\}_{k=1}^\infty$ is an orthonormal basis of $H$. We run into two problems: the sum may diverge or it may differ depending on the choice of the orthonormal basis.
If order for the notion of trace to make sense in infinite dimension, we need to restrict ourselves to a class of operators for which we can prove the trace is finite and independent of the chosen orthonormal basis.

Given an operator $T \in \mathcal{L}(H)$ one has that $T^*T\in \mathcal{L}(H)$ is nonnegative and self-adjoint, hence it makes sense to define $|T|:=(T^*T)^{1/2} \in \mathcal{L}(H)$ as the unique nonnegative self-adjoint linear operator such that $|T|^2=T^*T$.
\begin{definition}
An operator $T \in \mathcal{L}(H)$ is \emph{trace class} if there is an orthonormal basis $\{e_k\}_{k=1}^\infty$ such that 
\begin{equation*}
\sum_{k=1}^\infty\langle |T|e_k, e_k\rangle_H < \infty.
\end{equation*}
The set of trace class operators on $H$ is denoted by $\mathcal{L}_1(H)$.
\end{definition}
One has to notice that this definition could depend on the choice of the
orthonormal basis. But there is the following result concerning trace class operators (see \cite[Corollary 18.2]{Conway} for a proof).
\begin{proposition}
Let $T \in \mathcal{L}_1(H)$, then 
\begin{equation*}
\sum_{k=1}^\infty\langle |T|e_k, e_k\rangle_H = \sum_{k=1}^\infty\langle |T|f_k, f_k\rangle_H
\end{equation*}
for all orthonormal bases $\{e_k\}_{k=1}^\infty$ and $\{f_k\}_{k=1}^\infty$ in $H$.
\end{proposition}
In light of the above result, an operator $T$ is trace class if and only if $\sum_{k=1}^\infty\langle |T|e_k, e_k\rangle_H$ is finite for every choice of  the orthonormal basis $\{e_k\}_{k=1}^\infty$ and 
\begin{equation*}
\|T\|_{\mathcal{L}_1(H)}:= \sum_{k=1}^\infty\langle |T|e_k, e_k\rangle_H
\end{equation*}
is well defined, depending only on the operator $T$. The number $\|T\|_{\mathcal{L}_1(H)}$ is called \textit{trace norm} of $T$. The space $\mathcal{L}_1(H)$ is a vector space, the trace norm $\|T\|_{\mathcal{L}_1(H)}$ does indeed define a norm on $\mathcal{L}_1(H)$ and, with respect to this norm, $\mathcal{L}_1(H)$ is complete as stated in the following result.
\begin{proposition}
$\left(\mathcal{L}_1(H), \|\cdot\|_{\mathcal{L}_1(H)}\right)$ is a Banach space.
\end{proposition}
\begin{proof}
See \cite[Theorem 18.11]{Conway}.
\end{proof}
In light of the following Proposition it is possible to define the trace of an operator in $\mathcal{L}_1(H)$.
\begin{proposition}
Let $T \in \mathcal{L}_1(H)$, then 
\begin{equation*}
\sum_{k=1}^\infty\langle Te_k, e_k\rangle_H
\end{equation*}
is finite and independent of the choice of the orthonormal basis $\{e_k\}_{k=1}^\infty$.
\end{proposition}
\begin{proof}
See \cite[Proposition 18.9]{Conway}.
\end{proof}
\begin{definition}
\label{trace_def}
Let $T \in \mathcal{L}_1(H)$. We definite the \emph{trace} of $T$ to be 
\begin{equation*}
\emph{Tr}(T):= \sum^\infty_{k=1} \langle Te_k, e_k\rangle_{H},
\end{equation*}
for some orthonormal basis $\{e_k\}_{k=1}^\infty$ in $H$.
\end{definition}
Notice that when $H$ is finite-dimensional, every operator is trace class and the above definition of trace of $T$ coincides with the definition of the trace of a matrix. 

We conclude this part with an important property of self-adjoint nonnegative trace-class operators (see \cite[Proposition 1.20]{DaPrato} and \cite[Proposition 6.11]{Brezis}).
\begin{proposition}
\label{spectral}
Let $T \in \mathcal{L}_1(H)$ be nonnegative and self-adjoint. Then there exists a complete orthonormal system $\{e_k\}_{k \in \mathbb{N}}$ of $H$ and a nonnegative real sequence $\{\lambda_k\}_{k \in \mathbb{N}}$ such that $Te_k=\lambda_k e_k$ for all $k \in \mathbb{N}$ and $\sum_{k \in \mathbb{N}}\lambda_k < \infty$.
\end{proposition}
As a consequence of the above result, for $T \in \mathcal{L}(H)$ nonnegative and self-adjoint we have the representation formulae
\begin{equation*}
T x= \sum_{k \in \mathbb{N}}\lambda_k \langle x, e_k\rangle_H e_k, \qquad x \in H,
\end{equation*}
\begin{equation*}
T^{1/2} x= \sum_{k \in \mathbb{N}}\sqrt{\lambda_k} \langle x, e_k\rangle_H e_k, \qquad x \in H.
\end{equation*}
The orthogonal projection on the kernel of $T$ is given by $Px=\sum_{k \in \mathbb{N}: \lambda_k = 0} \langle x, e_k\rangle_H e_k$, so $(I-P)x=\sum_{k \in \mathbb{N}: \lambda_k \ne 0} \langle x, e_k\rangle_H e_k$. Thus 
\begin{equation*}
T^{-1/2} x= \sum_{k \in \mathbb{N}: \lambda_k \ne 0}\frac{\langle x, e_k\rangle_H e_k}{\sqrt{\lambda_k}}, \qquad x \in T^{1/2}(H)
\end{equation*}
and 
\begin{equation*}
\langle x,y\rangle_{T^{1/2}(H)}= \sum_{k \in \mathbb{N}: \lambda_k \ne 0}\frac{\langle x, e_k\rangle_H \langle y, e_k\rangle_H}{\lambda_k}, \qquad x,y \in T^{1/2}(H).
\end{equation*}

\subsubsection{Hilbert-Schmidt operators}

\begin{definition}
An operator $T \in \mathcal{L}(H,K)$ is called \emph{Hilbert-Schmidt} if 
\begin{equation*}
\sum_{k=1}^\infty \|Te_k\|^2_K < \infty,
\end{equation*}
where $\{e_k\}_{k=1}^\infty$ is an orthonormal basis of $H$.
\\
The space of Hilbert-Schmidt operators from $H$ to $K$ is denoted by $\mathcal{L}_2(H,K)$; for simplicity we write $\mathcal{L}_2(H)$ instead of $\mathcal{L}_2(H,H)$.
\end{definition}

\begin{proposition}
\label{ind_HS_bas}
The definition of Hilbert-Schmidt operator and the number $\sum_{k=1}^\infty \|Te_k\|^2_K$ does not depend on the choice of the basis.
\end{proposition}
\begin{proof}
Let $\{e_k\}_{k=1}^\infty$, $\{f_k\}_{k=1}^\infty$ and $\{g_k\}_{k=1}^\infty$
be three orthonormal bases of $H$. We compute 
\begin{equation*}
\|T e_k\|^2_H
= \sum_{j=1}^\infty \langle Te_k, g_j\rangle_{\mathcal{H}}\langle Te_k, g_j\rangle_{\mathcal{H}}
=\sum_{j=1}^\infty \langle e_k, T^*g_j\rangle_{\mathcal{H}}\langle e_k, T^*g_j\rangle_{\mathcal{H}}.
\end{equation*}
Thus 
\begin{align*}
\sum_{k=1}^\infty \|Te_k\|^2_H= \sum_{k=1}^\infty\sum_{j=1}^\infty \langle e_k, T^*g_j\rangle_{\mathcal{H}}\langle e_k, T^*g_j\rangle_{\mathcal{H}}
=\sum_{j=1}^\infty \|T^*g_j\|^2_H < \infty.
\end{align*}
By a similar reasoning we get $\sum_{j=1}^\infty \|T^*g_j\|^2_H=\sum_{i=1}^\infty \|Tf_i\|^2_H$ and thus $\sum_{k=1}^\infty \|Te_k\|^2_H=\sum_{k=1}^\infty \|Tf_k\|^2_H$, which concludes the proof.
\end{proof}

\begin{remark}
\label{bound_L_LHS}
It is easy to prove (see e.g. \cite[Proposition 18.6]{Conway}) that, if $T \in \mathcal{L}_2(H)$, then $\|T\|_{\mathcal{L}(H)}\le \|T\|_{\mathcal{L}_2(H)}$.
\end{remark}
In view of Proposition \ref{ind_HS_bas} the number 
\begin{equation}
\label{HS_norm_def}
\|T\|_{\mathcal{L}_2(H,K)}:= \left( \sum_{k=1}^\infty \|Te_k\|^2_K\right)^{\frac 12}
\end{equation}
is well defined, depending only on the operator $T$; this number is called \textit{Hilbert-Schmidt norm} of $T$.
\begin{proposition}
The space $\mathcal{L}_2(H,K)$ equipped with the norm \eqref{HS_norm_def} is a separable Hilbert space with the scalar product
\begin{equation*}
\langle S, T\rangle_{\mathcal{L}_2(H,K)}:=\sum_{k=1}^\infty \langle Se_k, Te_k\rangle_{K}, \qquad S, T \in \mathcal{L}_2(H,K),
\end{equation*}
where $\{e_k\}_{k=1}^\infty$ is an orthonormal basis of $H$.
\end{proposition}
\begin{proof}
We have to prove completeness and separability.
\begin{itemize}
\item [i)] $\mathcal{L}_2(H,K)$ is complete.
\\
Let $\{T_n\}_{n \in \mathbb{N}}$ be a Cauchy sequence in $\mathcal{L}_2(H,K)$. Then, in virtue of Remark \ref{bound_L_LHS}, $\{T_n\}_{n \in \mathbb{N}}$ is also a Cauchy sequence in $\mathcal{L}(H,K)$. Because of the completeness of $\mathcal{L}(H,K)$ (see Proposition \ref{lin_Ban}) there exists an element $T \in \mathcal{L}(H,K)$ such that $\|T_n-T\|_{\mathcal{L}(H,K)} \rightarrow 0$ as $n \rightarrow \infty$. In particular, the sequence $\{\|T_n\|_{\mathcal{L}_2(H,K)}\}_{n \in \mathbb{N}}$ is a real Cauchy sequence and thus convergent to an element $\lambda \in \mathbb{R}$.
By the Fatou lemma we have that, for any orthonormal basis $\{e_k\}_{k \in \mathbb{N}}$ of $H$,
\begin{align*}
\|T\|^2_{\mathcal{L}_2(H,K)}
&= \sum_{k=1}^\infty \|Te_k\|^2_K
=\sum_{k=1}^\infty \lim_{n \rightarrow \infty}\|T_n e_k\|^2_K 
\\
& \le \liminf_{n \rightarrow \infty}\sum_{k=1}^\infty \|T_n e_k\|^2_K 
= \liminf_{n \rightarrow \infty}\|T_n\|^2_{\mathcal{L}_2(H,K)}= \lambda< \infty,
 \end{align*}
which proves that $T \in \mathcal{L}_2(H,K)$.
Moreover, again by the Fatou lemma, we infer that 
\begin{align*}
\limsup_{n \rightarrow \infty} \|T_n-T\|_{\mathcal{L}_2(H,K)} 
&=\limsup_{n \rightarrow \infty} \sum_{k=1}^\infty\|T_ne_k-Te_k\|^2_H
\\
& \le \sum_{k=1}^\infty\limsup_{n \rightarrow \infty} \|T_ne_k-Te_k\|^2_H =0, \end{align*}
which proves that $T_n \rightarrow T$ in $\mathcal{L}_2(H,K)$ when $n \rightarrow \infty$ and concludes the proof of completeness.
\item [ii)] $\mathcal{L}_2(H,K)$ is separable. 
\\
Let $\{e_k\}_{k \in \mathbb{N}}$ and $\{f_j\}_{j \in \mathbb{N}}$ be orthonormal bases in $H$ and $K$, respectively. If we define $f_j \otimes e_k:=f_j\langle e_k, \cdot \rangle_H$ for $j, k \in \mathbb{N}$ \begin{footnote}{See Section \ref{app_ten_pro} for the definition of $v \otimes u$.}\end{footnote}, then it is clear that $f_j \otimes e_k \in \mathcal{L}_2(H,K)$ for all $j, k \in \mathbb{N}$ and for an arbitrary $T \in \mathcal{L}_2(H,K)$ we get that 
\begin{equation*}
\langle f_j \otimes e_k, T\rangle_{\mathcal{L}_2(H,K)}
= \sum_{n=1}^\infty \langle e_k, e_n\rangle_H \langle f_j, Te_n\rangle_K = \langle f_j,Te_k\rangle_K.
\end{equation*}
Therefore, it is easy to see that $f_j \otimes e_k$, $j, k \in \mathbb{N}$, is an orthonormal system in $\mathcal{L}_2(H,K)$. Moreover, $T=0$ if $\langle f_j \otimes e_k, T\rangle_{\mathcal{L}_2(H,K)}=0$ for all $j,k \in \mathbb{N}$, and therefore $\text{Span}\left( f_j \otimes e_k \ \text{s.t.} \ j,k \in \mathbb{N}\right)$ is a dense subspace of $\mathcal{L}_2(H,K)$.
\end{itemize}
\end{proof}

In the following result we collect some properties of trace class and Hilbert-Schmidt operators. For other properties see \cite{Conway}.
\begin{proposition}
\label{proposition-HS}
\begin{itemize}
\item [(i)] If $T \in \mathcal{L}_2(H)$, then $T^* \in \mathcal{L}_2(H)$ and $\|T\|_{\mathcal{L}_2(H)}=\|T^*\|_{\mathcal{L}_2(H)}$.
\item [(ii)] If $A \in \mathcal{L}(H)$ and $T \in \mathcal{L}_2(H)$, then $TA, AT \in \mathcal{L}_2(H)$ and $\|AT\|_{\mathcal{L}_2(H)}$, $\|TA\|_{\mathcal{L}_2(H)} \le \|A\|_{\mathcal{L}(H)}\|T\|_{\mathcal{L}_2(H)}$.
\item[(iii)] If $T \in \mathcal{L}_2(H,K)$ and $S \in \mathcal{L}_2(K,H)$ then $ST\in \mathcal{L}_1(H)$ and $\|ST\|_{\mathcal{L}_1(H)} \le \|S\|_{\mathcal{L}_2(K,H)}\|T\|_{\mathcal{L}_2(H,K)}$.
\item [(iv)] Let $T \in \mathcal{L}_2(H)$, then $TT^* \in \mathcal{L}_1(H)$ and $\emph{Tr}(TT^*)=\|T\|_{\mathcal{L}_2(H)}^2$.
\item [(v)] Trace class and Hilbert-Schmidt operators are compact operators.
\end{itemize}
\end{proposition}
\begin{proof}
For items (i)-(ii) see \cite[Proposition 18.6]{Conway}, for item (iii) see \cite[Proposition 18.8]{Conway}, item (iv) is a consequence of items (i) and (iii), for item (v) see \cite[Corollary 18..7. and Theorem 18.11]{Conway}.
\end{proof}

\section{Tensor products}
\label{app_ten_pro}
As pointed out in \cite[Appendix E]{Janson} (to which we mainly refer for this part; see also \cite[Chapter II]{RS}), “ the name \textit{tensor product} denotes an idea rather than a specific construction". In fact, the tensor product of two vector spaces or Hilbert spaces is defined only up to isomorphism. By a standard abuse of language we usually refer to any of the isomorphic choices as the \textit{tensor product}. To be precise we should refer to different choices as different \textit{realizations} of the tensor product.

\subsection{Algebraic tensor product of vector spaces}
\begin{definition}
\label{def_TP}
If $V$ and $W$ are two vector spaces, then their \emph{(algebraic) tensor product} is a vector space, denoted by $V \otimes W$, together with a bilinear map $\tau: V \times W \rightarrow V \otimes W$, written as $(v,w)\mapsto v \otimes w \in V\otimes W$, with the following universal property: if $U$ is any vector space and $\varphi: V \times W \rightarrow U$ is a bilinear map, then there exists a unique linear map $\psi : V \otimes W\rightarrow U$ such that  the diagram 
\begin{equation*}
\xymatrix{
V \times W \ar^{\tau}[r] \ar_{\varphi}[d] & V \otimes W \ar@{-->}^{\psi}[dl]
 \\
U 
}
\end{equation*}
commutes, that is $\varphi= \psi \circ \tau$ (i.e. , $\varphi(v,w)=\psi(v \otimes w)$ for $v \in V$, $w \in W$).
\end{definition}
\begin{remark}
Notice that not every element in $V \otimes W$ is of the form $v \otimes w$, but it is always a finite sum of such elements. In other words, the \emph{elements $\{v \otimes w, v \in V, w \in W\}$ span $V \otimes W$}
. Notice also that the representation $v \otimes w$ is not unique, even when it exists. For example, $v \otimes 0=0=0 \otimes w$ for any $v,w$.
\end{remark}

The following results, which follows from the universal property, guarantees the uniqueness \textit{up to isomorphisms} of the tensor product. 
\begin{proposition}
Tensor products $V \otimes W$ are unique up to isomorphisms. That is, given two tensor products 
\begin{equation*}
\tau_1: V \times W \rightarrow T_1,
\end{equation*}
\begin{equation*}
\tau_1: V \times W \rightarrow T_2,
\end{equation*}
there is a unique isomorphism $i:T_1 \rightarrow T_2$ such that the diagram 
\begin{equation*}
\xymatrix{
&T_1 \ar@{-->}^i [dd]
\\
V \times W \ar^{\tau_1}[ur] \ar_{\tau_2}[dr] 
 \\
&T_2 
}
\end{equation*}
commutes i.e. $\tau_2=i \circ \tau_1$.
\end{proposition}
\begin{proof}
First we show that for a tensor product $\tau: V \times W \rightarrow T$ , the only map $f: T \rightarrow T$ that the makes the following diagram 
\begin{equation*}
\xymatrix{
&T \ar@{-->}^f[dd]
\\
V \times W \ar^{\tau}[ur] \ar_{\tau}[dr] 
 \\
&T 
}
\end{equation*}
commute is the identity. The definition of a tensor product demands that, given a bilinear map $\tau: V \times W \rightarrow T$ (with $T$ in place of the space $U$ of the definition) there is a unique linear map $\psi: T \rightarrow T$ such that the diagram 
\begin{equation*}
\xymatrix{
V \times W \ar^{\tau}[r] \ar_{\tau}[d] & T \ar@{-->}^{\psi}[dl]
 \\
T 
}
\end{equation*}
commutes i.e. $\tau= \psi \circ \tau$. The identity map on $T$ has this property, so it is the \textit{only} map $T \rightarrow T$ with this property.

Looking at two tensor products, first take $\tau_2: V \times W \rightarrow T_2$ in place of the $\varphi: V \times W \rightarrow U$ in the definition. That is, there is a unique linear $\Phi_1:T_1 \rightarrow T_2$ such that the diagram 
\begin{equation*}
\xymatrix{
V \times W \ar^{\tau_1}[r] \ar_{\tau_2}[d] & T_1 \ar@{-->}^{\Phi_1}[dl]
 \\
T_2
}
\end{equation*}
commutes, i.e. $\tau_2= \Phi_1 \circ \tau_1$. Similarly, reversing the roles, there is a unique linear map $\Phi_2: T_2 \rightarrow T_1$ such that the diagram
\begin{equation*}
\xymatrix{
V \times W \ar^{\tau_2}[r] \ar_{\tau_1}[d] & T_2 \ar@{-->}^{\Phi_2}[dl]
 \\
T_1
}
\end{equation*}
commutes, i.e. $\tau_1= \Phi_2 \circ \tau_2$.
Therefore, $\Phi_2 \circ \Phi_1$ makes the following diagram 
\begin{equation*}
\xymatrix{
V \times W \ar^{\tau_1}[r] \ar_{\tau_1}[ddr] & T_1 \ar@{-->}^{\Phi_1}[d]
\\
&T_2 \ar@{-->}^{\Phi_2}[d]
\\
&T_1
}
\end{equation*}
commute and so it is the identity, from the first part of the proof. Simmetrically, $\Phi_1 \circ \Phi_2$ is the identity too. Thus the maps $\Phi_1$ and $\Phi_2$ are mutually inverse i.e. $\Phi_2 \circ \Phi_1=\text{id}_{T_1}$ and $\Phi_1 \circ \Phi_2=\text{id}_{T_2}$, thus $\Phi_1$ and $\Phi_2$ are isomorphisms and this concludes the proof.
\end{proof}

The existence of the tensor product can be shown in several ways, for example by the example below.
\begin{example}
\label{ex_tp_1}
A useful construction of the tensor product is the following. Choose bases $\{e_i\}_{i \in I}$ in $V$ and $\{f_j\}_{j \in J}$ in $W$. Let $U$ be any vector space whose dimension equals the cardinality of $V\times W$, select a basis of $U$ and denotes its elements as $\{u_k\}_{k \in K}$, where $|K|=|I \times J|$. Then, a basis $\{e_i \otimes f_j\}_{(i,j) \in I \times J}$ for $V \otimes W$ can be constructed associating to each element $e_i \otimes f_j$ an element $u_k$. $V \otimes W$ can be thus constructed as the space generated by the elements of the basis $\{e_i \otimes f_j\}_{(i,j) \in I \times J}$. Notice that $\emph{dim}(V \otimes W)=\emph{dim}(V)\emph{dim}(W)$.

\end{example}

\begin{example}
As a particular case of Example \ref{ex_tp_1} consider two vector spaces $V$ and $W$ with \emph{dim}$(V)=3$ and \emph{dim}$(W)=2$. Let $U=\mathbb{R}^6$ and select its canonical basis $\{e_i\}_{i=1}^6$. We can construct a basis for $V \otimes W$ by setting $v_1 \otimes w_1=e_1$, $v_2 \otimes w_1=e_2$, $v_3 \otimes w_1=e_3$, $v_1 \otimes w_2=e_4$, $v_2 \otimes w_2=e_5$ and $v_3 \otimes w_2=e_6$.
 \end{example}

\subsection{Tensor product of Hilbert spaces}

If now $V$ and $W$ are Hilbert spaces, then their scalar products can be used to define a scalar product on their algebraic tensor product $V \otimes W$. This is in general not complete, but taking the completion we obtain a \emph{Hilbert space tensor product}, still denoted by $V \otimes W$.

\begin{definition}
\label{Htp_def}
Let $V$ and $W$ be Hilbert spaces. It can be seen (e.g. using the universal property twice) that there exists a unique inner product on the algebraic tensor product $V \otimes W$ such that
\begin{equation*}
\langle v \otimes w, \tilde v \otimes \tilde w\rangle_{V \otimes W}= \langle v, \tilde v\rangle_V \langle w,\tilde w\rangle_W.
\end{equation*}
The \emph{Hilbert space tensor product} of $V$ and $W$ is the Hilbert space obtained by completing the algebraic tensor product in the norm corresponding to this inner product.
\end{definition}
We still denote this tensor product by $V \otimes W$. One can show that the above definition is equivalent to the following one.
\begin{definition}
If $H_1$ and $H_2$ are two Hilbert spaces, their \emph{tensor product} is an Hilbert space $H_1 \otimes H_2$ together with a bilinear map $H_1 \times H_2 \rightarrow H_1 \otimes H_2$ denoted by $(f_1,f_2) \mapsto f_1 \otimes f_2\in H_1 \otimes H_2$ such that the range of this map is total
\begin{footnote}{We recall that a set in a topological vector space is said to be \emph{total} if the family of (finite) linear combinations of elements of the set is dense in the space. 
}\end{footnote} 
in $H_1\otimes H_2$ and for all $f_1, g_1\in H_1$, $f_2,g_2 \in H_2$
\begin{equation}
\label{sca_pro_ts}
\langle f_1 \otimes f_2, g_1 \otimes g_2\rangle_{H_1 \otimes H_2}= \langle f_1, g_1\rangle_{H_1} \langle f_2, g_2\rangle_{H_2}.
\end{equation}
\end{definition}
It follows from \eqref{sca_pro_ts} that for any $f_1 \in H_1$, $f_2 \in H_2$, 
\begin{equation*}
\|f_1 \otimes f_2\|_{H_1 \otimes H_2}=\|f_1\|_{H_1}\|f_2\|_{H_2}.
\end{equation*}

The Hilbert space tensor product can be constructed by completing the algebraic tensor product as in Definition \ref{Htp_def}. There are however also direct constructions. 
 In the following examples we provide some important constructions of tensor product Hilbert spaces in some relevant situations (we omit the detailed verifications).

\begin{example}
If $\{e_i\}_{i \in I}$ and $\{f_j\}_{j \in J}$ are orthonormal bases in $H_1$ and $H_2$, then $\{e_i \otimes f_j\}_{(i,j) \in I \times J}$ is an orthonormal basis in $H_1 \otimes H_2$. Notice that, if $H_1$ and $H_2$ are separable, then $H_1\otimes H_2$ is also separable.
\end{example}

\begin{example}
The tensor product of two $L^2$-spaces $L^2(X, \mu)$ and $L^2(Y, \nu)$, where $\mu$ and $\nu$ are $\sigma$-finite measures, has one possible concrete realization as $L^2(X, \mu)\otimes L^2(Y, \nu)=L^2(X\times Y, \mu\times \nu)$ with $(f \otimes g)(x, y)=f(x)g(y)$, for any $x \in X$, $y \in Y$.
\end{example}

\begin{example}
The tensor product $L^2(X, \mu) \otimes H$, where $(X, \mu)$ is a measurable space and $H$ is an Hilbert space, can be realized as $L^2(X, \mu) \otimes H=L^2(X, \mu;H)$ with $f \otimes h=fh$.
\end{example}

\begin{example}
\label{ex_TP_HS}
The tensor product of two Hilbert spaces $H_1$ and $H_2$ can be realized as the (Hilbert) space of Hilbert-Schmidt operators from $H_1$ to $H_2$, i.e. $H_1 \otimes H_2=\mathcal{L}_2(H_1, H_2)$ with $h_1 \otimes h_2= T_{h_1,h_2}\in \mathcal{L}_2(H_1, H_2)$ where
\begin{equation}
\label{def_TP_HS}
T_{h_1,h_2}(v_1):= \langle h_1, v_1\rangle_{H_1}h_2,\quad \text{for any} \  v_1 \in H_1.
\end{equation} 
\\
It is easy to verify that $T_{h_1,h_2}$ is a bilinear operator that satisfies the universal property of Definition \ref{def_TP} and that it is an Hilbert-Schmidt operator.
\end{example}

\begin{remark}
By comparing the previous examples we have that the tensor product of two $L^2$-spaces $L^2(X, \mu)$ and $L^2(Y, \nu)$, where $\mu$ and $\nu$ are $\sigma$-finite measures, has several realizations, all of which are isomorphic to each other,
\begin{equation*}
L^2(X) \otimes L^2(Y) 
\simeq L^2(X \times Y) \simeq L^2(X;L^2(Y)) \simeq \mathcal{L}_2(L^2(X),L^2(Y)).
\end{equation*} 
Plainly, one has to select the identification of the tensor product space 
that is well adapted to the specific problem at hand. 
\end{remark}

We have so far considered only tensor products of two spaces, but the definition is extendable to several spaces by considering multilinear maps instead of bilinear. Alternatively, one can define tensor products of several spaces by induction starting from the case of two spaces. 
For example, given the Hilbert spaces $H_1, H_2, H_3$, one can construct $H_1 \otimes H_2 \otimes H_3$ exploiting the isomorphisms $H_1 \otimes H_2 \otimes H_3 \simeq H_1 \otimes (H_2 \otimes H_3)  \simeq (H_1 \otimes H_2) \otimes H_3$.  In particular,  we can construct the tensor powers $H^{\otimes n}$ of any Hilbert space.

\section{Some recalls on semigroups}
\label{semigroup_sec}

We recall here the definition and some basic properties of $C_0$-semigroups. We mainly refer to \cite{Pazy}.

Let $H$ be a Hilbert space
\begin{footnote}{All the definitions and results of this Section can be stated working with Banach spaces (see \cite{Pazy}): for our purpose it is enough to deal with Hilbert spaces.}\end{footnote}. 
\begin{definition}
A \emph{semigroup of bounded linear operators} acting on $H$ is a family of operators $T_t \in \mathcal{L}(H)$ depending on the parameter $t \ge 0$ satisfying the following properties 
\begin{itemize}
\item [(S1)] $T_0=I$,
\item [(S2)] $T_{t+s}=T_tT_s$, for all $s,t \ge 0$.
\end{itemize}
\end{definition}

It is understood that $T_tT_s$ means $T_t \circ T_s$. Observe that (S2) implies that $T_s$ and $T_t$ commute for all $s,t \ge 0$.

\begin{definition}
A semigroup $\{T_t\}_{t \ge 0}$ of bounded linear operators on $H$ is a \emph{strongly continuous} or \emph{$C_0$-semigroup of bounded linear operators} if
\begin{itemize}
\item [(S3)] $\lim_{t \rightarrow 0} T_tx=x$, for all $x \in H$.
\end{itemize}
\end{definition}

\begin{proposition}
\label{bound_S_prop}
Let $\{T_t\}_{t \ge 0}$ be a $C_0$-semigroup. There exist constants $w \ge 0$ and $M \ge 1$ such that 
\begin{equation*}
\|T_t\|_{\mathcal{L}(H)} \le Me^{\omega t}, \qquad \forall \ t \ge 0.
\end{equation*}
\end{proposition}
\begin{proof}
We first show that there exists $r>0$ such that 
\begin{equation}
\label{sti1S}
\sup_{t \in [0,r]}\|T_t\|_{\mathcal{L}(H)} \le M,
\end{equation}
for some positive constant $M$. If not, there exists a sequence $t_n \rightarrow 0$ for which
\begin{equation}
\label{est_below}
\|T_{t_n}\|_{\mathcal{L}(H)} \ge n.
\end{equation}
Since $t_n \rightarrow 0$, from (S3) we infer $\|T_{t_n}x- x\|_H \rightarrow 0$ as $n \rightarrow \infty$, for any $x \in H$. By the triangular inequality this yields
\begin{equation*}
\sup_{n \in \mathbb{N}}\|T_{t_n} x\|_H \le \sup_{n \in \mathbb{N}}\|T_{t_n} x-x\|_H+ \|x\|_H \le \|x\|_H, \quad \forall \ x \in H.
\end{equation*}
By the Uniform Boundedness Principle we thus infer $\sup_{n \in \mathbb{N}}\|T_{t_n}\|_{\mathcal{L}(H)}< \infty$ which contradicts \eqref{est_below}. 
Thus estimates \eqref{sti1S} holds true and since $T_0=I$ it must be $M \ge 1$.

Next we set $\omega =\frac 1r\log(M)$. Given $t \ge 0$ we have $t=nr+\delta$ with $0\le \delta <r$ and therefore by the semigroup property (S2) 
\begin{align*}
\|T_t\|_{\mathcal{L}(H)}=\|(T_r)^nT_\delta\|_{\mathcal{L}(H)}\le \|T_r\|^n_{\mathcal{L}(H)}\|T_\delta\|_{\mathcal{L}(H)} \le M^{n+1}=Me^{\omega nr}=Me^{\omega t}Me^{-\delta t} \le Me^{\omega t}.
\end{align*}
\end{proof}

\begin{corollary}
If $\{T_t\}_{t \ge 0}$ is a $C_0$-semigroup then for every $x \in H$, $t \mapsto T_tx$ is a continuous function from $[0, + \infty)$ into $H$.
\end{corollary}
\begin{proof}
Let $t, h \ge 0$. From Proposition \ref{bound_S_prop} we infer
\begin{equation*}
\|T_{t+h}x-T_tx\|_H\le \|T_t\|_{\mathcal{L}(H)}\|T_h x-x\|_H \le Me^{\omega t}\|T_h x-x\|_H
\end{equation*}
and from (S3) we infer the right-continuity of $t \mapsto T_tx$. 
\\
Let $t \ge h\ge 0$, 
then 
\begin{equation*}
\|T_{t-h}x-T_tx\|_H \le \|T_{t-h}\|_{\mathcal{L}(H)}\|x-T_hx\|_H \le Me^{\omega t}\|x-T_hx\|_H
\end{equation*}
and from (S3) we obtain the left-continuity of $t \mapsto T_tx$.
\end{proof}

\begin{definition}
The linear operator $A$ of domain
\begin{equation*}
\emph{Dom}(A)=\left\{ x \in H \ : \ \exists \ \lim_{t \rightarrow 0} \frac{T_tx-x}{t}\right\}
\end{equation*}
defined by 
\begin{equation*}
Ax=\lim_{t \rightarrow 0} \frac{T_tx-x}{t},
\end{equation*}
is the \emph{infinitesimal generator} of the semigroup $\{T_t\}_{t \ge 0}$. \end{definition}

\begin{proposition}
\label{lem_S}
Let $T_t$ be a $C_0$-semigroup and let $A$ be its infinitesimal generator. Then 
\begin{itemize}
\item [a)] for $x \in H$, 
\begin{equation*}
\lim_{h \rightarrow 0} \frac 1 h \int_t^{t+h} T_sx\, {\rm d}s=T_tx,
\end{equation*}
\item [b)] for $x \in H$, $\int_0^t T_sx\, {\rm d}s \in \emph{Dom}(A)$ and 
\begin{equation*}
A\left(\int_0^tT_sx\,{\rm d}s\right)=T_tx-x,
\end{equation*}
\item [c)] for $x \in \emph{Dom}(A)$, $T_t x \in \emph{Dom}(A)$ and 
\begin{equation*}
\frac{{\rm d}}{{\rm d}t}T_tx=AT_tx=T_tAx,
\end{equation*}
\item [d)] for $x \in \emph{Dom}(A)$,
\begin{equation*}
T_tx-T_sx=\int_s^t T_rAx\,{\rm d}r=\int_s^tAT_rx\, {\rm d}r.
\end{equation*}
\end{itemize}
\end{proposition}
\begin{proof}
See \cite[Theorem 2.4]{Pazy}.
\end{proof}

\begin{proposition}
If $A$ is the infinitesimal generator of a $C_0$-semigroup $\{T_t\}_{t \ge 0}$, then $\emph{Dom}(A)$ is dense in $H$ and $A$ is a closed linear operator.
\end{proposition}

\begin{proof}
For every $x \in H$ set $x_t:=\frac 1t \int_0^t T_sx\,{\rm d}s$. Proposition \ref{lem_S} b) yields $x_t \in$ Dom$(A)$ for $t>0$ and Proposition \ref{lem_S} a) yields $x_t \rightarrow x$ as $t \rightarrow 0$. Thus the closure of Dom$(A)$ equals $H$, that is Dom$(A)$ is dense in $H$. One can easily see that $A$ is a linear operator, it remains to prove that it is closed. To prove the closedness of $A$ let $\{x_n\}_n \in$ Dom$(A)$ be a sequence such that $x_n \rightarrow x$ and $Ax_n \rightarrow y$ as $n \rightarrow \infty$. From Proposition \ref{lem_S}(d) we infer 
\begin{equation}
\label{star_S}
T_tx_n-x_n =\int_0^t T_sAx_n\, {\rm d}s.
\end{equation}
The integrand on the r.h.s. of \eqref{star_S} converges to $T_sy$ uniformly on bounded intervals. Therefore, letting $n \rightarrow \infty$  in \eqref{star_S} yields
\begin{equation}
\label{star_S_bis}
T_tx-x= \int_0^t T_sy \, {\rm d}s.
\end{equation}
Dividing \eqref{star_S_bis} by $t>0$ and letting $t$ go to zero, Proposition \ref{lem_S} a) yields $x \in$ Dom$A$ and $Ax=y$.
\end{proof}

\newpage


\bibliographystyle{abbrv}

\end{document}